\documentclass[11pt,twoside]{article}
\usepackage{authblk}
\usepackage{amsmath}%
\usepackage{amssymb}%
\usepackage{stix}
\usepackage{graphicx}
\usepackage{hyperref}
\usepackage[usenames,dvipsnames]{color}
\allowdisplaybreaks
\newcommand{\abbr}[1]{{\sc\lowercase{#1}}}

\providecommand{\U}[1]{\protect\rule{.1in}{.1in}}
\newtheorem{theorem}{Theorem}

\newtheorem{condition}[theorem]{Condition}

\newtheorem{corollary}[theorem]{Corollary}

\newtheorem{definition}[theorem]{Definition}

\newtheorem{lemma}[theorem]{Lemma}

\newtheorem{proposition}[theorem]{Proposition}
\newtheorem{remark}[theorem]{Remark}

\newenvironment{proof}[1][Proof]{\noindent\textbf{#1.} }{\ \rule{0.5em}{0.5em}}

\newcommand{\red}[1]{{\color{red} #1}}

\newcommand{\E}{\mathbb{E}}

\newcommand{\K}{\mathbb{K}}
\newcommand{\M}{\mathbb{M}}
\newcommand{\N}{\mathbb{N}}

	\renewcommand{\P}{\mathbb{P}}

\newcommand{\R}{\mathbb{R}}

\newcommand{\wt}{\widetilde}
\newcommand{\ovl}{\overline}
\newcommand{\ep}{\epsilon}
\newcommand{\cA}{\mathcal{A}}

\newcommand{\cD}{\mathcal{D}}

\newcommand{\cF}{\mathcal{F}}
\newcommand{\cG}{\mathcal{G}}

\newcommand{\cL}{\mathcal{L}}
\newcommand{\cM}{\mathcal{M}}
\newcommand{\cN}{\mathcal{N}}

\newcommand{\cR}{\mathcal{R}}

\begin{document}
\title{A non-inertial model for particle aggregation under turbulence }
\author{Franco Flandoli\footnote{Scuola Normale Superiore di Pisa. Piazza Dei Cavalieri 7, Pisa 56126, Italy. Email: franco.flandoli@sns.it.}   \footnote{Corresponding Author.}  \,  and Ruojun Huang\footnote{Universit\"at M\"unster. Einsteinstr. 62, M\"unster 48149, Germany. Email: ruojun.huang@uni-muenster.de.} }

\maketitle

\abstract{We consider an abstract non-inertial model of aggregation
under the influence of a Gaussian white noise with prescribed
space-covariance, and prove a formula for the mean collision rate $R$, per
unit of time and volume. Specializing the abstract theory to a non-inertial
model obtained by an inertial one, with physical constants, in the limit of
infinitesimal relaxation time of the particles, and the white noise obtained
as an approximation of a Gaussian noise with correlation time $\tau_{\eta}$,
up to approximations the formula reads $R\sim\tau_{\eta}\left\langle
\left\vert \Delta_{a}u\right\vert ^{2}\right\rangle a\cdot n^{2}$ where $n$ is
the particle number per unit of volume and $\left\langle \left\vert \Delta
_{a}u\right\vert ^{2}\right\rangle $ is the square-average of the increment of
random velocity field $u$ between points at distance $a$, the particle radius.
If we choose the Kolmogorov time scale $\tau_{\eta}\sim\left(  \frac{\nu
}{\varepsilon}\right)  ^{1/2}$ and we assume that $a$ is in the dissipative
range where $\left\langle \left\vert \Delta_{a}u\right\vert ^{2}\right\rangle
\sim\left(  \frac{\varepsilon}{\nu}\right)  a^{2}$, we get Saffman-Turner
formula for the collision rate $R$.}

\tableofcontents

\section{Introduction}
\subsection{The abstract mathematical model and the main
results\label{subsect math model and result}}

We consider non-inertial particles $x_{1}^{N}\left(  t\right)  ,...,x_{N}%
^{N}\left(  t\right)  $ in $\R^3$, subject to the dynamics%
\begin{equation}
\frac{dx_{i}^{N}\left(  t\right)  }{dt}=\sigma\partial_{t}W_{\epsilon/\xi
}\left(  x_{i}^{N}\left(  t\right)  ,t\right)  +\sigma_{0}\frac{dB_{i}\left(
t\right)  }{dt} \label{abstract SDE}%
\end{equation}
where $B_{i}\left(  t\right)  $ are $\R^3$-valued independent
Brownian motions, $\xi,\sigma,\sigma_{0}$ are positive real numbers, $W\left(
x,t\right)  $ is a Gaussian $\R^3$-valued random field, Brownian in
time
\[
\mathbb{E}\left[  W^{i}\left(  x,t\right)  W^{j}\left(  y,s\right)  \right]
=C^{i,j}\left(  x-y\right)  t\wedge s
\]
for $i,j=1,2,3$, where the covariance matrix $C\left(  x\right)  $ is
described in Section \ref{subsect covariance} below (hence $\partial
_{t}W\left(  x,t\right)  $ is a white noise in time) and $W_{\epsilon/\xi
}\left(  x,t\right)  =W\left(  \frac{\xi x}{\epsilon},t\right)  $.

The parameter $\epsilon$ is linked to $N$ by the relation
\begin{equation}
N\epsilon=1 \label{abstract scaling}%
\end{equation}
which describes a specific degree of sparseness of the system (finite mean
free path, in the Brownian case), described below in Section
\ref{subsect mean free path}. When we take the limit as $N\rightarrow\infty$,
hence $\epsilon\rightarrow0$, the noise covariance is rescaled with $\epsilon
$; hence, mathematically speaking, there would be no reason to rescale the
noise by $\epsilon/\xi$ instead of just $\epsilon$ (it is a matter of
redefining the covariance). But $\epsilon$, in the example described below in
Sections \ref{subsect inertial model}-\ref{subsect reduction to non inertial},
is linked to the particle radius, hence the viewpoint in the scaling limit
(\ref{abstract scaling}) is that the particle radius goes to zero meanwhile
the number of particles goes to infinity (roughly speaking, particles occupy a
bounded region, independently of $N$). The fact that also the noise scaling
parameter $\epsilon/\xi$ goes to zero is a subsequent main assumption, not the
structural link between the particles and the domain. We believe that for the
interpretation of the results it is important to think that we discuss a
scaling limit of a particle system, with a link (\ref{abstract scaling})
between number of particles and radius of interaction (a local interaction),
in which there is an additional element, a turbulent fluid, with certain
scaling properties related to the particle size.

When two particles $x_{i}^{N},x_{j}^{N}$, $i\neq j$, are sufficiently close,
they may disappear, in a Markovian way, with rate%
\begin{equation}
\frac{R_{0}}{N}\epsilon^{-d}\theta\left(  \epsilon^{-1}\left(  x_{i}^{N}%
-x_{j}^{N}\right)  \right)  \label{abstract rate}%
\end{equation}
where $\theta$ is a smooth compact support probability density function.
Therefore $\epsilon$ is linked to the interaction radius of the particles. The
positive number $R_{0}$ modulates the pairwise interaction rate;\ perhaps the
most interesting regime is when $R_{0}$ is very large, case where the final
average collision rate does not diverge to infinity with $R_{0}$ but converges
to finite value (thanks to the assumption (\ref{abstract scaling})).

We prove that the empirical distribution of the particle
system converges to a function $f\left(  x,t\right)  $, weak solution of the
equation%
\begin{equation}
\partial_{t}f=\frac{1}{2}\left(  \sigma^{2}+\sigma_{0}^{2}\right)  \Delta
f-\overline{\mathcal{R}}f^{2}.\label{PDE1}%
\end{equation}
The quantity $\overline{\mathcal{R}}$ is the key one to capture the average
collision rate but it is important to notice, in connection with real
applications, that what is usually called average collision rate, say $R$, in
Physics, namely the number of collisions per unit of time and volume, is
linked to $\overline{\mathcal{R}}$ by the formula%
\begin{equation}
R=\frac{\overline{\mathcal{R}}}{N}n^{2}\label{transf rule}%
\end{equation}
where $N$ is the total number of particles in the system and $n$ is the
average number of particles per unit of volume. Indeed, $f$ is the limit of a
normalized empirical distribution, hence the $Nf$ corresponds to $n$, then
equation (\ref{PDE1}) gives us (up to the redistribution term $\frac{1}%
{2}\left(  \sigma^{2}+\sigma_{0}^{2}\right)  \Delta f$), $\partial_{t}%
n\sim-\frac{\overline{\mathcal{R}}}{N}n^{2}$. In the abstract development of
the theory we shall call $\overline{\mathcal{R}}$, sometimes, collision rate,
but the correct transformation rule is (\ref{transf rule}).

In itself, equation (\ref{PDE1}) is a general form of quadratic limit
equations, valid also under different scaling limits than
\eqref{abstract scaling} and different dimensions. For instance, we get the same
equation if we perform a purely mean field limit with long range interaction
and then send to zero the range in the final PDE. But the main feature of the
scaling limit assumption \eqref{abstract scaling}, discovered by \cite{hr} (for
the model above without the term $\sigma\partial_{t}W_{\epsilon/\xi}\left(
x_{i}^{N}\left(  t\right)  ,t\right)  $)\ is that $\overline{\mathcal{R}}$ is
given by a non-trivial formula much different from the mean field case. One
has
\[
\overline{\mathcal{R}}=R_{0}\int\theta\left(  x\right)  \left(  1+u\left(
x\right)  \right)  dx
\]
where $u$ is the unique vanishing at infinity solution of class $C^{2}$ of the
equation%
\begin{equation}
\sigma_{0}^{2}\Delta u\left(  x\right)  +\sigma^{2}\operatorname{div}\left(
\omega\left(  \xi x\right)  \nabla u\left(  x\right)  \right)  =R_{0}\theta\left(
1+u\right)  \label{PDE2}%
\end{equation}
where
\[
\omega\left(  x\right)  =I_3-C\left(  x\right)  .
\]
Getting explicit information from this equation is not easy. But, in the limit
as $R_{0}\rightarrow\infty$, one also has%
\[
\overline{\mathcal{R}}\rightarrow\overline{\mathcal{R}}_{\infty}={\rm Cap}\left(
K_{\theta},A\right)
\]
where $K_{\theta}$ is the support of $\theta$ and%
\[
A\left(  x\right)  =\sigma_{0}^{2}I_3+\sigma^{2}\omega\left(  \xi x\right)  .
\]
The capacity ${\rm Cap}\left(  K_{\theta},A\right)  $ is defined as the infimum of
\[
\int_{\R^3}\nabla\phi^{T}\left(  x\right)  A\left(  x\right)
\nabla\phi\left(  x\right)  dx
\]
over all smooth functions $\phi\left(  x\right)  $, vanishing at infinity,
that are larger than 1 on a neighbor of $K_{\theta}$. This characterization is
more suitable for quantitative information (see Section
\ref{Subsect physical interpretation}). The limit as $R_{0}%
\rightarrow\infty$ is a very natural one, corresponding to ask for a certain
coalescence, when particles meet. In the sequel and in Section
\ref{Subsect physical interpretation} we shall always only use this characterization.

It is not easy to summarize in a sentence the intuition behind this result. A
key of interpretation could be that equation (\ref{PDE1}) contains the
infinitesimal generator of the ``1-particle motion'', while the differential
operator on the left-hand-side of (\ref{PDE2}) is like the infinitesimal
generator of the difference in position between two particles. Therefore
equation (\ref{PDE2}) contains information about the mutual behavior of
different particles.

With this remark, we may explain a deep fact that, at the beginning and with
the wrong eye, may even look counter-intuitive. First, notice that minimizing
a quadratic form like $\int_{\R^3}\nabla\phi^{T}\lambda^{2}%
I_3\nabla\phi dx$ produces a result of the form $C\lambda^{2}$. Therefore
the capacity ${\rm Cap}\left(  K_{\theta},A\right)  $ is (roughly speaking) linear
in the size of $A$. Therefore, in the case of the matrix function $A\left(
x\right)  =\sigma_{0}^{2}I_3+\sigma^{2}\omega\left(  \xi x\right)  $, it may look
strange that it increases linearly with $\sigma_{0}^{2}$, namely with the
noise term $\sigma_{0}\frac{dB_{i}\left(  t\right)  }{dt}$, and not linearly
with the covariance $\sigma^{2}C\left(  \xi x\right)  $, corresponding to the
other noise term; it increases with the opposite, $\omega\left(  \xi x\right)
=I_3-C\left(  \xi x\right)  $. The deep reason stands in the "two-particle
motion" of the equation (\ref{abstract SDE}). Two different particles
$x_{i}^{N}\left(  t\right)  ,x_{j}^{N}\left(  t\right)  $, $i\neq j$, are
driven by two independent Brownian motions $\sigma_{0}\frac{dB_{i}\left(
t\right)  }{dt}$ and $\sigma_{0}\frac{dB_{j}\left(  t\right)  }{dt}$ and by
the ``common noise'' $\sigma\partial_{t}W_{\epsilon/\xi}\left(  x,t\right)  $,
just evaluated at the two positions $x_{i}^{N}\left(  t\right)  ,x_{j}%
^{N}\left(  t\right)  $. If the correlation of the common noise at those
positions is very small, it also behaves like two independent Brownian
motions, hence its effect is similar to the one of $\sigma_{0}\frac
{dB_{i}\left(  t\right)  }{dt}$ and $\sigma_{0}\frac{dB_{j}\left(  t\right)
}{dt}$. And indeed, $\omega\left(  \xi x\right)  \sim I_3$. If the
space-correlation of the common noise, on the contrary, was very strong, it
should not contribute to change the distance of the two particles; and in this
case $\omega\left(  \xi x\right)  \sim0$. Therefore it is $\omega\left(  \xi x\right)
$, not $C\left(  \xi x\right)  $, which should be directly proportional to the
collision rate. In a sense, the term $\sigma_{0}^{2}I_3$ in $A\left(
x\right)  $ is there because of the independence of the ``internal'' Brownian
motions of particles, which is like a common noise with infinitesimal space correlation.

In equation (\ref{PDE1}), on the contrary, we have the infinitesimal generator
$\frac{1}{2}\left(  \sigma^{2}+\sigma_{0}^{2}\right)  \Delta$ of one-particle
motion, where $C\left(  \xi x\right)  |_{x=0}$ plays a role. Equation
(\ref{PDE1}) describes the particle density, in a sense the probability that a
single particle is here or there and alive, hence governed by the one-particle motion.

The motivation of the study of this paper comes from the physical theory of
aggregation of particles in turbulent fluids, see for instance  \cite{Abrahan,
Bec, Dou, Falkovich, Mehlig, Papini0, Papini, Pumir, Wilkin}. However, in such
a theory, particles are inertial, inertia being quantified by means of the
Stokes number $St$ (defined as the ratio between the relaxation time of the
particle over the relaxation time of a turbulent eddy). When inertia is large,
particles feel turbulence like a light pollutant feels molecules and thus
perform a Brownian motion with damping; more inertia, less collisions. On the
other side of the ``inertial scale of the particle'', when inertia is almost
zero, particles move almost adhese to the fluid, which again leads to few
collisions. It is only in the intermediate range, when inertia is still
small-moderate but starts to show up, that we observe more collisions.

Our results, being based on a non-inertial model (but obtained from an
inertial one by an approximation procedure, hence preserving most physical
constants), meets necessarily only the case of small Stokes numbers; more
precisely, the Stokes number formally disappears from the model and appears a
link between two space-scales, the radius of the particle and the Kolmogorov
scale. The final results will be expressed in these terms. One of the results,
which shows the correctness of some approximations, is that we recover
Saffmann-Turner formula \cite{st} in the case when the particle radius falls
in the dissipation range of the fluid.

We complete the introduction by means of several subsections devoted mostly to
physical interpretations of the model and of the results. First, we describe
the covariance function $C\left(  x\right)  $ and the noise $W$ above, Section
\ref{subsect covariance}. Then we describe the inertial model often considered
in the literature, driven by a certain model of turbulent fluid, Section
\ref{subsect inertial model}. In that section, we deviate from our main
purpose and show briefly how Abrahamson theory arises, summary that may help
to clarify that we have no chance to get such a result when we neglect
inertia, see Subsection \ref{sub sub section large Stokes}. Then we describe
how we derive, heuristically, a non-inertial model from the inertial one,
Section \ref{subsect reduction to non inertial}. Finally, we interpret the
main theoretical results of this paper (those summarized above), in the case
of the non-inertial model with parameters $\xi,\sigma,\sigma_{0}$ coming from
the reduction from inertial to non-inertial, which contains some Physics; see
Section \ref{Subsect physical interpretation}.

Let us also stress that the limit (\ref{PDE1}) is deterministic in spite of
the presence of an environmental noise. Systems with environmental noise
attracted a lot of attention in recent years, in particle systems scaling
limit, mean field games and other theories. Usually the macroscopic dynamics
is stochastic, the environmental noise is not averaged out by the scaling
limit, opposite to the case of independent noise on each particle; see for
instance \cite{fh-alea} for an example in aggregation phenomena. But here we
assume that the environmental noise has a scaling structure in the number of
particles such that, in the limit, it behaves like independent noise on each
particle -- still preserving the covariance structure in the cell equation,
since it is an equation for the two-point motion! An example of similar result
for point vortices was obtained in \cite{FlaLuo}, inspired by a comment in
\cite{FouHauMis}.

\subsection{Space covariance and noise\label{subsect covariance}}

Let us start with the main building block (see for instance \cite[Section 1.1]{LXZ} for more
details). We assume to have a covariance matrix function $C:\mathbb{R}%
^{3}\rightarrow\mathbb{R}^{3\times3}$ having Fourier transform
\[
\widehat{C}\left(  z\right)  =g\left(  z\right)  \left(  I_{3}-\frac{z\otimes
z}{\left\vert z\right\vert ^{2}}\right)
\]
for a non-negative real valued function $g$ bounded, integrable, smooth,
depending only on $\left\vert z\right\vert $; here $I_{3}$ is the identity
matrix in $\mathbb{R}^{3}$. For technical reasons hopefully eliminable in
future more advanced works, we assume $C$ compact support, a property that can
be translated into a property of analyticity and growth of $g$ by a
Paley-Wiener-Schwartz theorem. To this covariance matrix function $C$, it is
associated a Gaussian vector valued process $W\left(  x,t\right)
\in\mathbb{R}^{3}$, with $\operatorname{div}W\left(  x,t\right)  =0$, that is
Brownian in time and has space-covariance $C$: denoting by $\mathbb{E}$\ the
expected value with respect to the randomness,
\[
\mathbb{E}\left[  W^{i}\left(  x,t\right)  W^{j}\left(  y,s\right)  \right]
=C^{i,j}\left(  x-y\right)  t\wedge s.
\]
Its time derivative $\partial_{t}W\left(  x,t\right)  $ is a solenoidal vector
valued white noise (formally speaking, the covariance of $\partial_{t}W\left(
x,t\right)  $ is $C\left(  x-y\right)  \delta\left(  t-s\right)  $).

One way to construct $W\left(  x,t\right)  $ from $C$, relevant for the
mathematical analysis below and possibly for numerical simulations, is the
following one. To $C$ we associate the covariance operator $\mathcal{C}$ in
the space $L_{s}^{2}\left(  \mathbb{R}^{3},\mathbb{R}^{3}\right)  $ of
solenoidal vector fields on $\mathbb{R}^{3}$; this operator is positive
selfadjoint and compact, hence it has a spectral decomposition $\mathcal{C}%
v=\sum_{k=0}^{\infty}\lambda_{k}\left\langle v,e_{k}\right\rangle _{L^{2}%
}e_{k}$, $v\in L_{s}^{2}\left(  \mathbb{R}^{3},\mathbb{R}^{3}\right)  $, where
$\left(  e_{k}\right)  _{k\in\mathbb{N}}$ is a complete orthonormal system of
$L_{s}^{2}\left(  \mathbb{R}^{3},\mathbb{R}^{3}\right)  $ and $\lambda_{k}%
\geq0$ for every $k\in\mathbb{N}$. Taken a sequence $\left(  W_{k}\left(
t\right)  \right)  _{k\in\mathbb{N}}$ of independent scalar Brownian motions,
we define%
\begin{align}\label{series-repn}
W\left(  x,t\right)  =\sum_{k=0}^{\infty}\lambda_{k}e_{k}\left(  x\right)
W_{k}\left(  t\right)  .
\end{align}

From the assumption on the Fourier transform of $C$, one can prove that
$x\mapsto W\left(  x,t\right)  $ is an isotropic random field. In particular
$C\left(  x\right)  $ is a function of $\left\vert x\right\vert $ only; and
one has $C\left(  0\right)  =\frac{1}{3}\left\Vert g\right\Vert _{L^{1}}I_{3}$
where $\left\Vert g\right\Vert _{L^{1}}=\int_{\mathbb{R}^{3}}g\left(
x\right)  dx$. We assume $\frac{1}{3}\left\Vert g\right\Vert _{L^{1}}=1$,
hence%
\[
C\left(  0\right)  =I_{3}.
\]

\subsection{Inertial model\label{subsect inertial model}}

Behind the non-inertial model of this paper there is an inertial one that we
describe now. We use new notations, to avoid confusion.

Consider $N$ particles in $\mathbb{R}^{3}$, with positions and velocities
$\left(  x_{1}^{N}\left(  t\right)  ,v_{1}^{N}\left(  t\right)  \right)  $,
..., $\left(  x_{N}^{N}\left(  t\right)  ,v_{N}^{N}\left(  t\right)  \right)
$, subject to a turbulent fluid by Stokes law:
\begin{align*}
\frac{dx_{i}^{N}\left(  t\right)  }{dt}  &  =v_{i}^{N}\left(  t\right) \\
\frac{dv_{i}^{N}\left(  t\right)  }{dt}  &  =-\frac{1}{\tau_{P}}\left(
v_{i}^{N}\left(  t\right)  -u\left(  x_{i}^{N}\left(  t\right)
,t\right)  \right)  +\frac{u_{P}}{\sqrt{\tau_{P}}}\frac{dB_{i}\left(
t\right)  }{dt}%
\end{align*}
where $u\left(  x,t\right)  $ is a space-time stationary Gaussian vector
random field, divergence free, with%
\[
\mathbb{E}\left[  u_{i}\left(  x,t\right)  u_{j}\left(  y,s\right)  \right]
=e^{-\frac{\left\vert t-s\right\vert }{\tau_{\eta}}}P_{ij}\left(  x-y\right)
\]
for $i,j=1,2,3$. With this first identity we have specified the relaxation
time\ $\tau_{\eta}$ of the turbulent fluid. Then we want to specify, by a
proper choice of $P\left(  x\right)  $, that the fluid is very correlated when
$\left\vert x-y\right\vert \leq\eta$, interpreted as Kolmogorov scale (below
which dissipation dominates and the velocity field $u\left(  x,t\right)  $ is
smooth and mildly varying), then less and less correlated as $\left\vert
x-y\right\vert $ increases above $\eta$. One way, accepting moreover isotropy
and other idealizations, is to assume
\[
P\left(  x\right)  =u_{\eta}^{2}C\left(  \frac{L}{\eta}x\right)
\]
where $u_{\eta}^{2}$ is a measure of the turbulent kinetic energy, $C\left(
x\right)  $, $\left(  e_{k},\lambda_{k}\right)  _{k\in\mathbb{N}}$ and
$W\left(  x,t\right)  $ are as in the previous section and $L$ is a
macroscopic length scale. For simplicity (although it is not mandatory) we
choose $u_{\eta}=\frac{\eta}{\tau_\eta}$. If Kolmogorov scale is adopted, we
may use the formulae $\eta=\left(  \frac{\nu^{3}}{\varepsilon}\right)  ^{1/4}%
$, $\tau_{\eta}=\left(  \frac{\nu}{\varepsilon}\right)  ^{1/2}$, $u_{\eta
}=\left(  \nu\varepsilon\right)  ^{1/4}$, where $\nu$ is the kinematic
viscosity of the fluid and $\varepsilon$ the average dissipation energy rate.

Let us also introduce the matrix%
\[
D\left(  x\right)  =u_{\eta}^{2}\left(  I_3-C\left(  \frac{L}{\eta}x\right)
\right)  .
\]
One has
\begin{align*}
D_{ij}\left(  x-y\right)   &  =u_{\eta}^{2}\delta_{ij}-P_{ij}\left(
x-y\right) \\
&  =\frac{1}{2}\mathbb{E}\left[  \left(  u_{i}\left(  x,t\right)
-u_{i}\left(  y,t\right)  \right)  \left(  u_{j}\left(  x,t\right)
-u_{j}\left(  y,t\right)  \right)  \right]
\end{align*}
so it is related to the second order structure function, in particular to the
quantity $\left\langle \left\vert \Delta_{\left\vert x-y\right\vert
}u\right\vert ^{2}\right\rangle $ anticipated in the Abstract, defined as the
square-average of the increment of random velocity field $u$ between points at
distance $\left\vert x-y\right\vert $:%
\[
\left\langle \left\vert \Delta_{\left\vert x-y\right\vert }u\right\vert
^{2}\right\rangle =2{\rm Tr}D\left(  x-y\right)  .
\]

One way to ``realize'' such field $u\left(  x,t\right)  $ is by taking $u\left(
x,t\right)  =u_{\eta}Z_{\eta}\left(  x,t\right)  $ where $Z_{\eta}\left(
x,t\right)  $ is a stationary solution of
\[
\partial_{t}Z_{\eta}\left(  x,t\right)  =-\frac{1}{\tau_{\eta}}Z_{\eta}\left(
x,t\right)  +\frac{1}{\sqrt{\tau_{\eta}}}\partial_{t}W_{\eta/L}\left(
x,t\right)  .
\]
One can easily show that
\begin{equation}
\mathbb{E}\left[  Z_{\eta}^{i}\left(  x,t\right)  Z_{\eta}^{j}\left(
y,s\right)  \right]  =e^{-\frac{\left\vert t-s\right\vert }{\tau_{\eta}}%
}C^{i,j}\left(  \frac{L}{\eta}\left(  x-y\right)  \right)  . \label{correl Z}%
\end{equation}
The process $Z_{\eta}\left(  x,t\right)  $ has the same space-covariance as
$W_{\eta/L}\left(  x,t\right)  $, but it is stationary with time-correlation
$\tau_{\eta}$. We keep the notation $\eta/L$ for $W_{\eta/L}$ for comparison
with $W_{\epsilon/\xi}$ of Section \ref{subsect math model and result}, but
drop $L$ from all other notations like $u_{\eta},\tau_{\eta},Z_{\eta}$ for
simplicity of notations.

This completes the description of the fluid velocity field. Concerning the
particles, we have assumed in the model above that they have an inertial
motion, with relaxation time $\tau_{P}$ and Stokes law $-\frac{1}{\tau_{P}%
}\left(  v_{i}^{N}\left(  t\right)  -u_{\eta}Z_{\eta}\left(  x_{i}^{N}\left(
t\right)  ,t\right)  \right)  $ linking them to the fluid motion, plus a very
small ``internal'' noise $\frac{u_{P}}{\sqrt{\tau_{P}}}\frac{dB_{i}\left(
t\right)  }{dt}$, phenomenologically taking into account particle geometry and
internal degrees of freedom interacting with the environment in a complex way.
Neglecting the fluid, the standard deviation of $v_{i}^{N}\left(  t\right)  $
would be proportional to $u_{P}$, providing a phenomenological interpretation
of the parameter $u_{P}$; that we shall assume much smaller than all other
contributions, just needed to mathematical reasons of ``diffusion regularization''.

Particles have an interaction radius $a$: when the centers have distance less
or equal to $2a$, they undergo a collision which may produce coalescence of
the particles, with a certain rate of collision (since we do not develop the
full inertial theory, we skip the precise description of the rate). We assume
to work in a fluid-particle regime where $\eta$ and $a$ are roughly
comparable, maybe differing multiplicatively by a large or small length-scale
ratio%
\[
\xi_{0}=\frac{a}{\eta}%
\]
but not being different infinitesimal orders when we take the limit
$a\rightarrow0$. Therefore we assume $\xi$ finite in the limit $a\rightarrow0$
and we can write the covariance $C\left(  \frac{L}{\eta}x\right)  $ of the
noise in the form%
\[
C\left(  \frac{\xi_{0}L}{a}x\right)  .
\]

This model, heuristically speaking, captures a number of properties of
aggregation in turbulent fluids, the we describe in the following three cases
$1_{iner},2_{iner},3_{iner}$ (certainly non exhausting the important
informations on the problem, but useful for the conceptualizations made later on).

$1_{iner}$: Case $\tau_{P}\ll \tau_{\eta}$. In this case, particles are almost
adhese to the fluid and thus do not collide so often. This case is classically
covered by Saffman-Turner law $\overline{\mathcal{R}}\sim a^{3}$ (up to
constants involving the factor $\left(  \frac{\varepsilon}{\nu}\right)
^{1/2}$) \cite{st}.

$2_{iner}$: Case $\tau_{P}\sim\tau_{\eta}$ (up to finite constants). In this
case, particles are less adhese, with some (but not large) inertia. More
collisions are observed.

$3_{iner}$: Case $\tau_{P}\gg\tau_{\eta}$. In this case inertia dominates more
and more as $St$ increases; the behavior of particles is similar to
independent Brownian motions, but with inertia. More inertia produces less
collisions, hence the mean collision rate decreases with $St$. See Subsection
\ref{sub sub section large Stokes} below for additional informations (like
Abrahamson law $\frac{u_{\eta}}{\sqrt{St}}$).

In this paper the condition $\tau_{P}\ll1$ is essential for the non-inertial
approximation. Although a priori this could still allow us to consider
$\tau_{P}\gg\tau_{\eta}$, mathematically the information about $\tau_{P}$ is
lost and our results do not cover such a case. Our paper is restricted to
cases $1_{iner}$ and $2_{iner}$.

However, even with the restriction to $1_{iner}$ and $2_{iner}$, we do not
fully recover the rapid increase of mean collision rate happening at some
Stokes level, when we pass from $1_{iner}$ to $2_{iner}$. Our theory gives
only a relatively smooth transition from $1_{iner}$ to $2_{iner}$. We presume
it is due to the homogeneous Gaussian approximation of the velocity field,
which lacks very important concentration effects of the density of particles.

\subsubsection{The case $\tau_{P}\gg\tau_{\eta}$
\label{sub sub section large Stokes}}

Let us expand this case for comparison, case unfortunately not covered by the
theory of this paper.

In this case, from the observational viewpoint of a particle, $Z_{\eta}\left(
x,t\right)  $ is like a white noise, say of the form $\sqrt{\tau_{\eta}}%
d\beta_{i}\left(  t\right)  $ (see (\ref{Z approx}) below), hence the velocity
$v_{i}^{N}\left(  t\right)  $ behaves like the solution of a damped Langevin
equation with double noise:%
\[
\frac{dv_{i}^{N}\left(  t\right)  }{dt}=-\frac{1}{\tau_{P}}v_{i}^{N}\left(
t\right)  +\frac{u_{\eta}\sqrt{\tau_{\eta}}}{\tau_{P}}\frac{d\beta_{i}\left(
t\right)  }{dt}+\frac{u_{P}}{\sqrt{\tau_{P}}}\frac{dB_{i}\left(  t\right)
}{dt}.
\]
For notational and conceptual simplicity, let us neglect the very small
contribution $\frac{u_{P}}{\sqrt{\tau_{P}}}\frac{dB_{i}\left(  t\right)  }%
{dt}$ and rewrite the main coefficient as%
\begin{equation}
\frac{dv_{i}^{N}\left(  t\right)  }{dt}=-\frac{1}{\tau_{P}}v_{i}^{N}\left(
t\right)  +\frac{1}{\sqrt{\tau_{P}}}\frac{u_{\eta}}{\sqrt{St}}d\beta
_{i}\left(  t\right)  \label{OU}%
\end{equation}
where the Stokes number is defined as%
\[
St=\frac{\tau_{P}}{\tau_{\eta}}.
\]
For large $St$, the inertia dominates the picture and collisions are not so
frequent. To see this from equation (\ref{OU}), first let us assume that, for
two different particles $i\neq i$, $\beta_{i}\left(  t\right)  $ and
$\beta_{j}\left(  t\right)  $ are independent Brownian motions; this is a
reasonable approximation, using the rough argument that a strong order between
the time scales $\tau_{P}\gg \tau_{\eta}$ corresponds also to a strong order
between the length scales $a\gg \eta$. Hence, for the difference $v_{ij}%
^{N}\left(  t\right)  =v_{i}^{N}\left(  t\right)  -v_{j}^{N}\left(  t\right)
$ we have exactly the same equation (\ref{OU}), just with a factor $\sqrt{2}$
in front of the noise. The standard deviation of a stationary solution is
proportional to $\frac{u_{\eta}}{\sqrt{St}}$, hence (up to a multiplicative
factor) we infer that the average collision rate is $\frac{u_{\eta}}{\sqrt
{St}}$. This is the famous Abrahamson law \cite{Abrahan}. Unfortunately, this
law is lost in the non-inertial approximation below (in spite of several
attempts; it seems an intrinsic weakness of the non-inertial model, for the
obvious reason that it corresponds to an approximation around $\tau_{P}=0$).
Similarly, we lose the change of monotonicity which is known from theories and
experiments, namely the decrease of the average collision rate in $St$ for
large $St$, after its increase for smaller $St$.

\subsection{Reduction to a non-inertial
model\label{subsect reduction to non inertial}}

The model above, suitable for numerical simulations, is too difficult at
present for a rigorous mathematical investigation of the limit as
$N\rightarrow\infty$. We apply two simplifications, based on the fact that
both $\tau_{P}$ and $\tau_{\eta}$ are very small in practical examples. We
have a first approximation%
\begin{equation}
Z_{\eta}\left(  x,t\right)  dt\sim\sqrt{\tau_{\eta}}\partial_{t}W_{\eta
/L}\left(  x,dt\right)  \label{Z approx}%
\end{equation}
coming from the equation
\[
-\tau_{\eta}\partial_{t}Z_{\eta}\left(  x,t\right)  =Z_{\eta}\left(
x,t\right)  -\sqrt{\tau_{\eta}}\partial_{t}W_{\eta/L}\left(  x,t\right)
\]
as $\tau_{\eta}\rightarrow0$; and a second approximation
\[
v_{i}^{N}\left(  t\right)  \sim u_{\eta}Z_{\eta}\left(  x_{i}^{N}\left(
t\right)  ,t\right)  +u_{P}\sqrt{\tau_{P}}\frac{dB_{i}\left(  t\right)  }{dt}%
\]
coming from the equation
\[
-\tau_{P}\frac{dv_{i}^{N}\left(  t\right)  }{dt}=v_{i}^{N}\left(  t\right)
-u_{\eta}Z_{\eta}\left(  x_{i}^{N}\left(  t\right)  ,t\right)  -u_{P}%
\sqrt{\tau_{P}}\frac{dB_{i}\left(  t\right)  }{dt}%
\]
as $\tau_{P}\rightarrow0$. Inserting further the first approximation in the
second one, we get%
\[
v_{i}^{N}\left(  t\right)  \sim u_{\eta}\sqrt{\tau_{\eta}}\partial_{t}%
W_{\eta/L}\left(  x_{i}^{N}\left(  t\right)  ,t\right)  +u_{P}\sqrt{\tau_{P}%
}\frac{dB_{i}\left(  t\right)  }{dt}.
\]
Substituting this into the first equation we get%
\begin{equation}
\frac{dx_{i}^{N}\left(  t\right)  }{dt}=u_{\eta}\sqrt{\tau_{\eta}}\partial
_{t}W_{\eta/L}\left(  x_{i}^{N}\left(  t\right)  ,t\right)  +u_{P}\sqrt
{\tau_{P}}\frac{dB_{i}\left(  t\right)  }{dt} \label{simplified SDE}%
\end{equation}
which is a closed equation. Understanding the degree of these approximations
and their relative importance is very difficult and much beyond the scope of
this paper.

Concerning coalescence, first we oversimplify certain aspects of the problem
(to focus on difficult others) and therefore assume that particles, after
coalescence, ``exit the system''. Namely we focus only on the number of
particles of radius $a$ and neglect not only what happens to the result of
coalescence but also the important fact that, in reality, larger particles
produced by coalescence usually may coalesce also with the smaller particles
of radius $a$, contributing to the decrease of their number. But all these
elements may be accounted for by the formalism of Smoluckowski equations,
which we omit here. By the way, with the present simplified formalism the
model may be applied to fragmentation too.

Concerning the collision rate between two particles $x_{i}^{N}$ and $x_{j}%
^{N}$, $i\neq j$, in the case of the non-inertial model a natural choice is
the form%
\begin{equation}
\frac{R_{0}}{N}a^{-3}\theta_{0}\left(  a^{-1}\left(  x_{i}^{N}-x_{j}%
^{N}\right)  \right)  \label{individual rate inertial}%
\end{equation}
where $R_{0}$ has dimension $\left[  L\right]  ^{3}\left[  T\right]  ^{-1}$ to
give dimension $\left[  T\right]  ^{-1}$ to the rate. Here $\theta_{0}$ is a
smooth non negative function, with $\int\theta_{0}\left(  x\right)  dx=1$,
roughly equal to $\frac{3}{4\pi}$ in the ball $B\left(  0,1\right)  $, and
equal to zero outside.

Before we close this section, let us mention heuristically what the previous
non-inertial model may hope to tell us.

We mention this fact, along with the next one, to say that the general formula
(\ref{final}) or the previous ones of this section, more precise and rigorous,
are only a first step based on an alternative methodology, still far from
providing a good understanding of the small-to-moderate range of particle
radius. Particles, just described by their position $x_{i}^{N}\left(
t\right)  $, are driven only by two Brownian sources, $W_{\eta/L}\left(
x_{i}^{N}\left(  t\right)  ,t\right)  $ and $B_{i}\left(  t\right)  $. Thus
$x_{i}^{N}\left(  t\right)  $ behaves roughly like a Brownian motion (not
precisely, since $W_{\eta/L}\left(  x_{i}^{N}\left(  t\right)  ,t\right)  $
depends on the particle position). Particles will collide, because of the
random motion, with intensity increasing with the coefficient $u_{\eta}%
\sqrt{\tau_{\eta}}$ (also with $u_{P}\sqrt{\tau_{P}}$, but we assume $u_{P}$
much smaller than all other constants).

$1_{non-iner}$: case $a\ll \eta$. The process $W_{\eta/L}\left(  x,dt\right)  $
is very correlated below the Kolmogorov scale $\eta$, hence particles which
have radius $a$ much smaller than the Kolmogorov scale and are very close each
other, feel the same random impulses, thus the random strength of $W_{\eta
/L}\left(  x,t\right)  $ on them is very reduced, it does not put them closer;
they travel parallely, so to speak. Therefore aggregation, which increases
with $u_{\eta}\sqrt{\tau_{\eta}}$, is depleted if particles have a radius $a$
much smaller than $\eta$.

$2_{non-iner}$: case $a\uparrow\eta$. Both obstructions discussed in
$1_{non-iner}$ are progressively relaxed when $a$ increases to $\eta$.
Therefore the mean collision rate should monotonically increase, when
$a\uparrow\eta$.

Let us now make precise the scaling between $N$ and $a$ chosen here, which is
a critical issue. We separate it into a subsection for the reader's convenience.

\subsubsection{The mean free path regime\label{subsect mean free path}}

We assume the initial positions of the particles located in a unitary volume,
approximately, so the same is true in a finite time interval approximately.
Therefore, if we increase $N$, particles are more packed together, not just
distributed in a larger volume; and their radius $a$ is assumed to go to zero,
if we increase $N$ towards infinity, to get a sharp mathematical result. The
link between $N$ and $a$ will be a finite number (finite in the limit when
$N\rightarrow\infty$ and $a\rightarrow0$).%
\begin{equation}
Na=\rho_{0}.\label{scaling assumption}%
\end{equation}
Here $\rho_{0}$ is a given positive number with dimension $\left[  L\right]
$: It is a ``volume-occupation-constant'': the volume occupied by the balls of
radius $a$ centered at the particle positions is
\[
\text{volume of particles }\sim\text{ }\frac{4\pi}{3}Na^{3}=\frac{4\pi}{3}%
\rho_{0}a^{2}%
\]
hence the volume of particles linearly increases with $\rho_{0}$.

Since $a$ is small, this volume is very small, hence we are talking about a
very sparse regime. The intuition then is of a very sparse collection of
droplets. However, they perform, roughly speaking, Brownian motions (mostly
due to the turbulent contribution $u_{\eta}\sqrt{\tau_{\eta}}\partial
_{t}W_{\eta/L}\left(  x_{i}^{N}\left(  t\right)  ,t\right)  $), hence each one
of their centers explores in unitary time a subset of $\mathbb{R}^{3}$ which
has Hausdorff dimension 2 (with 2-dimensional Hausdorff measure zero). Hence,
a ball of radius $a$ centered at such a point, explores a set $A\subset
\mathbb{R}^{3}$ (called the Brownian sausage) having volume roughly of size
$a$, up to multiplication by the parameters of the Brownian motion. During the
excursion performed in a unit of time, the particle meets all particles that
(with their radius) intersect $A$. If the particles are roughly uniformly
distributed, then the number of particles which intersect $A$ is of order
$N\cdot {\rm Vol}(A)$, namely $Na$, which is a finite number $\rho_{0}$. This
implies that each particle has a finite number of encounters in the unit of
time. This is the so-called mean free path regime.

This is the only regime where we can find non trivial limit results. Indeed,
when $Na\rightarrow\infty$ we have a mean field result, since each particle
would meet infinitely many others in a unit of time; one can show that the
result is simply%
\[
\overline{\mathcal{R}}=R_{0}%
\]
which is a trivial and physically wrong result. On the contrary, for
$Na\rightarrow0$, we would not observe sufficient interactions in the scaling
limit; the average collision rate $\overline{\mathcal{R}}$ would vanish.

Therefore we investigate a \textit{non-mean-field regime}, and the
mathematical correction in the final formulae coming from this choice is essential.

\subsection{Application of our results and their physical interpretation
\label{Subsect physical interpretation}}

\subsubsection{Conversion of the notations}

Let us first summarize the notation of the concrete example of Section
\ref{subsect reduction to non inertial} in comparison with the abstract
notations of Section \ref{subsect math model and result}. The concrete
dynamics of particles is given by equation (\ref{simplified SDE}), the rate of
collision by (\ref{individual rate inertial}) and the scaling regime by
(\ref{scaling assumption}). In the abstract formulation of Section
\ref{subsect math model and result} the dynamics is given by
(\ref{abstract SDE}), the rate of collision by (\ref{abstract rate}) and the
scaling by (\ref{abstract scaling}). Therefore the conversion table is (on the
left the abstract notations, on the right those of the concrete example):%

\begin{align*}
\epsilon &  =\frac{a}{\rho_{0}}\\
\theta\left(  x\right)   &  =\rho_{0}^{-3}\theta_{0}\left(  \rho_{0}%
^{-1}x\right)  \\
\sigma &  =u_{\eta}\sqrt{\tau_{\eta}}\\
\sigma_{0} &  =u_{P}\sqrt{\tau_{P}}\\
\xi &  =\frac{L}{\rho_{0}}\frac{a}{\eta}%
\end{align*}
and recall%
\begin{align*}
P\left(  x\right)   &  =u_{\eta}^{2}C\left(  \frac{L}{\eta}x\right)  =u_{\eta
}^{2}C\left(  \frac{\xi_{0}L}{a}x\right)  \\
D\left(  x\right)   &  =u_{\eta}^{2}\left(  I_3-C\left(  \frac{L}{\eta
}x\right)  \right)  \\
R &  =\frac{\overline{\mathcal{R}}}{N}n^{2}%
\end{align*}
(the first identity comes from comparison of the scaling limit
assumption;\ then the second identity from comparison of the collision rates;
the last identity is chosen in order to have $W_{\epsilon/\xi}=W_{\eta/L}$,
recall the definition $W_{\epsilon/\xi}\left(  x,t\right)  =W\left(  \frac
{\xi}{\epsilon}x,t\right)  $ from Section \ref{subsect math model and result}
and $W_{\eta/L}\left(  x,t\right)  =W\left(  \frac{L}{\eta}x,t\right)  $ from
Section \ref{subsect reduction to non inertial}).

\subsubsection{The result in the notations of Section
\ref{subsect reduction to non inertial}}

The normalized empirical measure of the particle system converges to a density
$f$ satisfying
\[
\partial_{t}f=\frac{1}{2}\left(  u_{\eta}^{2}\tau_{\eta}+u_{P}^{2}\tau
_{P}\right)  \Delta f-\overline{\mathcal{R}}f^{2}%
\]
where
\[
\overline{\mathcal{R}}=R_{0}\int\rho_{0}^{-3}\theta_{0}\left(  \rho_{0}%
^{-1}x\right)  \left(  1+u\left(  x\right)  \right)  dx
\]
where $u$ is the unique vanishing at infinity solution of class $C^{2}$ of the
equation%
\[
u_{P}^{2}\tau_{P}\Delta u\left(  x\right)  +u_{\eta}^{2}\tau_{\eta
}\operatorname{div}\left(  \left(  I_3-C\left(  \frac{L}{\rho_{0}}\frac
{a}{\eta}x\right)  \right)  \nabla u\left(  x\right)  \right)  =R_{0}\rho
_{0}^{-3}\theta_{0}\left(  \rho_{0}^{-1}x\right)  \left(  1+u\right)
\]
namely%
\[
u_{P}^{2}\tau_{P}\Delta u\left(  x\right)  +\tau_{\eta}\operatorname{div}%
\left(  D\left(  \frac{a}{\rho_{0}}x\right)  \nabla u\left(  x\right)
\right)  =R_{0}\rho_{0}^{-3}\theta_{0}\left(  \rho_{0}^{-1}x\right)  \left(
1+u\right)  .
\]
In the limit as $R_{0}\rightarrow\infty$, one has%
\[
\overline{\mathcal{R}}\rightarrow\overline{\mathcal{R}}_{\infty}= {\rm Cap}\left(
K_{\theta_{0}\left(  \rho_{0}^{-1}\cdot\right)  },A\right)  = {\rm Cap}\left(
B\left(  0,\rho_{0}\right)  ,A\right)
\]
where $K_{\theta_{0}\left(  \rho_{0}^{-1}\cdot\right)  }=B\left(  0,\rho
_{0}\right)  $ is the support of $\theta_{0}\left(  \rho_{0}^{-1}\cdot\right)
$, $A$ is the matrix function $A\left(  x\right)  =u_{P}^{2}\tau_{P}I_3%
+\tau_{\eta}D\left(  \frac{a}{\rho_{0}}x\right)  $ and the capacity
${\rm Cap}\left(  B\left(  0,\rho_{0}\right)  ,A\right)  $ is defined as the infimum
of $\int_{\R^3}\nabla\phi^{T}\left(  x\right)  A\left(  x\right)
\nabla\phi\left(  x\right)  dx$ over all smooth functions $\phi\left(
x\right)  $, vanishing at infinity, that are larger than 1 on a neighbor of
$B\left(  0,\rho_{0}\right)  $.

A very heuristic computation in spherical coordinates, for the simplified
example of minimizing $\int_{B\left(  0,\rho_{0}\right)  ^{c}}\nabla\phi
^{T}\left(  x\right)  I_3\nabla\phi\left(  x\right)  dx$ over all smooth
functions $\phi\left(  x\right)  $, vanishing at infinity, equal to 1 on
$\partial B\left(  0,\rho_{0}\right)  $, shows that the minimizer is not so
different from the radial function linearly decaying to zero on $r\in\left[
\rho_{0},2\rho_{0}\right]  $, with derivative modulus $\frac{1}{\rho_{0}}$.
The minimal value is roughly equal to $\rho_{0}$, up to a constant (this is
the well known result on the electric capacity of a sphere). If, in place of
$A\left(  x\right)  =I_3$ we take a slowly varying function with radial
symmetry, a rough approximation of the infimum is $\frac{1}{3}{\rm Tr}A\left(
\rho_{0}\right)  \rho_{0}$, since $A$ matters only when $r\in\left[  \rho
_{0},2\rho_{0}\right]  $. Therefore, the infimum in our case of interest can
be approximated by
\begin{align*}
\frac{1}{3}{\rm Tr}A\left(  \rho_{0}\right)  \rho_{0} &  =u_{P}^{2}\tau_{P}\rho
_{0}+\tau_{\eta}\frac{1}{3}{\rm Tr}D\left(  a\right)  \rho_{0}\\
&  =u_{P}^{2}\tau_{P}\rho_{0}+\tau_{\eta}\frac{1}{6}\left\langle \left\vert
\Delta_{a}u\right\vert ^{2}\right\rangle \rho_{0}%
\end{align*}
that (to simplify the formulae) we replace by%
\[
\tau_{\eta}\left\langle \left\vert \Delta_{a}u\right\vert ^{2}\right\rangle
\rho_{0}%
\]
since $u_{P}^{2}$ is very small and we have also missed other multiplicative
constants like $\frac{1}{6}$.

Therefore, taking into account $R=\frac{\overline{\mathcal{R}}}{N}n^{2}$ and
$Na=\rho_{0}$, we get%
\begin{equation}
R\sim\tau_{\eta}\left\langle \left\vert \Delta_{a}u\right\vert ^{2}%
\right\rangle a\cdot n^{2} \label{final}%
\end{equation}
as stated in the Abstract. Then, in particular, if we accept Kolmogorov time
scale $\tau_{\eta}\sim\left(  \frac{\nu}{\varepsilon}\right)  ^{1/2}$ and we
assume that $a$ is in the dissipative range where
\[
\left\langle \left\vert \Delta_{a}u\right\vert ^{2}\right\rangle
\sim\left\langle \left\vert \nabla u\right\vert ^{2}\right\rangle
a^{2}=\left(  \frac{\varepsilon}{\nu}\right)  a^{2}%
\]
(because of the balance relation $\nu\left\langle \left\vert \nabla
u\right\vert ^{2}\right\rangle =\varepsilon$), we get Saffman-Turner formula%
\begin{align*}
R  &  \sim\left(  \frac{\nu}{\varepsilon}\right)  ^{1/2}\left(  \frac
{\varepsilon}{\nu}\right)  a^{2}a\cdot n^{2}\\
&  =\left(  \frac{\varepsilon}{\nu}\right)  ^{1/2}a^{3}\cdot n^{2}.
\end{align*}

The Saffman-Turner regime was already well accepted in the Physical
literature, hence we just confirm the result in a more complex way than the
original paper \cite{st}. Part of our procedure is rigorous, but many steps
were hand waving approximations (very reasonable), hence we cannot claim that
we have rigorously proved Saffman-Turner result from first principle. We have
only provided an alternative derivation.

The final formula (\ref{final}), besides being correct in the dissipation
range, provides a new view on the behavior for slightly larger $a$. We should
however remember that we took a limit as $\tau_{P}\rightarrow0$, hence larger
$a$ compared to $\eta$ are meaningful only in few physical examples, precisely
when the density of a particle in much less than the density of the fluid,
which is not true for rain droplets in air or small dust clusters in star
dust; maybe it may find more applications in fragmentation processes. If
$\tau_{P}\rightarrow0$ by $a$ is in the inertial range, where
\[
\left\langle \left\vert \Delta_{a}u\right\vert ^{2}\right\rangle
\sim\varepsilon^{2/3}a^{2/3}%
\]
according to K41 theory (otherwise one should include\ intermittent
corrections) one get (with the same arguments above)%
\[
R\sim\nu^{1/2}\varepsilon^{1/6}a^{5/3}\cdot n^{2}%
\]
which is not a common result and not of easy interpretation. More generally,
even if we do not use the Kolmogorov scaling law of structure function, it is
clear that $\left\langle \left\vert \Delta_{a}u\right\vert ^{2}\right\rangle
$, which starts as $a^{2}$ for very small $a$, tends to saturate for larger
$a$, providing an ``S-shape''. This is very reasonable. It is a monotone
behavior with saturation. There is no chance however to see changes of
monotonicity, with the given approximations.

\begin{remark}
Even if limited to small Stokes numbers, in principle it should be possible to
incorporate in a precise mathematical model some ingredient which produces
particle concentration, a phenomenon that is observed and is among those
responsible for a strong increase of collision rate, when certain parameters
increase with respect to Saffman-Turner regime. This remains an open problem
in our framework, perhaps the most interesting one. See \cite{Bec2005} for
results in this direction.
\end{remark}

\section{Statement of main result}

In this section, we start by recapping our model with additional details (that were omitted in the Introduction due to their technical nature), and then we state our main Theorem \ref{thm:main}.

About the covariance function $C(x)$, recall that by construction  this function is radially symmetric, $C(0)=I_3$ and $\omega(x) = I_3-C(x)$ is nonnegative definite. The latter property is due to $\omega(x) = \frac{1}{2}\sum_{k\in \N}\big(e_k(0)-e_k(x)\big)^{\otimes2}$. We also have
\begin{align}\label{covariance}
C(x-y)=\sum_{k\in  \N}\lambda_ke_{k}\left(  x\right)  \otimes \lambda_ke_{k}\left(  y\right),
\end{align}
where we recall that our Gaussian random field $W(x, dt)$ can be written as a series \eqref{series-repn}. Henceforth, we will write 
\begin{align*}
C_\ep(x)= C\big(\frac{x}{\ep}\big), \quad e_k^\ep(x) = e_k\big(\frac{x}{\ep}\big).
\end{align*}
By the series representation of $W_{\ep/\xi}(x, dt)$, we can rewrite the dynamics of the particle system \eqref{abstract SDE} as 
\begin{align}\label{sde}
dx_i^N(t)&=\sigma\sum_{k\in  \N} \lambda_ke_{k}^{\epsilon}\left(\xi x_i^N(t)\right) \circ \,  dW_k(t)+\sigma_0\, dB_i(t),\quad i\in\cN(t),
\end{align}
where $\circ$ denotes Stratonovich integration. We make this choice in accordance with Wong-Zakai principle. Due to $e_k^\ep$ being divergence free, it turns out that in our setting it coincides with It\^o integration, since the Stratonovich-to-It\^o corrector is zero.  $\cN(t)\subset\{1,2,...,N\}=\cN(0)$ is the (random) index set of active particles at time $t$.

Let us formally state our model: Given a probability space $(\Omega, \cF, \P)$, for every $N\in \N$, we consider a system of initially $N$ particles evolving in $\R^3$, whose positions $x_i(0)$ at time $t=0$ are independent and identically distributed (i.i.d) with probability density $f_0(x)$. Particles are enumerated from $i=1,2,...,N$ at the beginning, and their indices are fixed once and for all. As long as a particle $i\in\{1,2,...,N\}$ is alive, its spatial position $x_i^N(t)\in\R^3$ follows \eqref{sde}. Any pair of distinct particles $(i,j)$, $i\neq j$,  may coalesce and both disappear, with a stochastic rate \eqref{abstract rate} which we denote by $r_{ij}^N$. When a particle is no longer active, we assign it a ``cemetery'' state $\emptyset$ which is absorbing, and thus the cardinality of active particles in the system at $t>0$ are less or equal than $N$ and non-increasing in time. 

We impose the following assumption on the initial density $f_0(x)$ of an individual particle.
\begin{condition}\label{cond:ini}
There exist finite constants $L_0$ and $\Gamma$ such that $f_0$ is supported in (the Euclidean ball of radius $L_0$ around the origin) $B(0,L_0)\subset\R^3$, and $f_0(x)\le \Gamma$ for all $x\in\R^3$. In particular, $f_0\in L^1\cap L^\infty(\R^3; \R_+)$. 
\end{condition}
The probability space is endowed with the canonical filtration
\begin{align*}
\cG_t:=\sigma\left\{\{W_k(s)\}_{k\in  \N}, \{B_i(s)\}_{i=1}^\infty, s\le t, \{x_i(0)\}_{i=1}^\infty\right\}, \quad t\ge0.
\end{align*} 
Denoting by $\zeta$ a generic configuration of the particle system, i.e. $\zeta=(x_1,x_2,...,x_N)\in\left(\R^3\cup\{\emptyset\}\right)^N$, and by $\zeta_t^N$ the configuration of the particle system at time $t$,  $(\zeta_t^N)_{t\ge0}$ is a Markov process with infinitesimal generator 
\begin{align}
\mathcal L^NF(\zeta)&=\cL^N_DF(\zeta)+\cL^N_JF(\zeta)\label{generator}\\
\cL^N_DF(\zeta)&:=\frac{1}{2}\sigma^2\sum_{k\in  \N}\sum_{i,j\in\cN(\zeta)}\lambda_k^2e_k^\ep(\xi x_i)\cdot\nabla_{x_i} \big(e_k^\ep(\xi x_j)\nabla_{x_j}F(\zeta)\big)+\frac{1}{2}\sigma_0^2\sum_{i\in\cN(\zeta)}\Delta_{x_i}F(\zeta)\nonumber\\
&=\frac{1}{2}\sigma^2\sum_{k\in  \N}\sum_{i,j\in\cN(\zeta)}\lambda_k^2\nabla_{x_i}\cdot \big(e_k^\ep(\xi x_i)\otimes e_k^\ep(\xi x_j)\nabla_{x_j}F(\zeta)\big)+\frac{1}{2}\sigma_0^2\sum_{i\in\cN(\zeta)}\Delta_{x_i}F(\zeta)\label{gen-1}\\
&=\frac{1}{2}\sigma^2\sum_{i,j\in\cN(\zeta)}\nabla_{x_i}\cdot \big(C_\ep(\xi(x_i-x_j))\nabla_{x_j}F(\zeta)\big)+\frac{1}{2}\sigma_0^2\sum_{i\in\cN(\zeta)}\Delta_{x_i}F(\zeta)\nonumber\\
\cL^N_JF(\zeta)&:=\frac{1}{2}\sum_{i\in\cN(\zeta)}\sum_{j\in\cN(\zeta): j\neq i}r_{ij}^N\left[F(\zeta^{-ij})-F(\zeta)\right]. \label{gen-2}
\end{align}
Here, $F: \left(\R^3\cup\{\emptyset\}\right)^N\to\R$ is a functional on the configuration space, $\nabla_{x_i}$ denotes partial gradient with respect to variable $x_i\in\R^3$ and similarly $\nabla_{x_i}\cdot$ denotes partial divergence.  We denote the index set of active particles in $\zeta$ by
\begin{align*}
\cN(\zeta)&:=\{i: x_i\in\zeta, x_i\neq\emptyset\}\subset\{1,2,...,N\},
\end{align*}
whereas $\zeta^{-ij}$ is a configuration obtained from $\zeta$ by setting its $i$-th and $j$-th slots to $\emptyset$, and keeping other slots unchanged. The $1/2$ in \eqref{gen-2} is due to double counting in the double sum of an annihilation event. We have used the fact that $e_k^\ep(x)$ are divergence free when calculating $\cL^N_D$ at step \eqref{gen-1}. Our assumptions actually imply (see \eqref{equiv-3iv}) that $\cL^N_DF(\zeta)$ can be equivalently written as non-divergence form
\begin{align*}
\cL^N_DF(\zeta)=\frac{1}{2}\sigma^2\sum_{i,j\in\cN(\zeta)}\text{Tr} \left(C_\ep(\xi (x_i-x_j))\nabla^2_{x_ix_j}F(\zeta)\right)+\frac{1}{2}\sigma_0^2\sum_{i\in\cN(\zeta)}\Delta_{x_i}F(\zeta).
\end{align*}
\begin{remark}
\red{$\cL_D^N$ accounts for the diffusion part of the dynamics, describing the diffusive movement of the particles between two consecutive annihilation events. The annihilation events happen at isolated times, which are sudden jumps in the particle configuration $\{\zeta_t^N\}_{t\ge0}$, and are described by the jump part of the generator $\cL_J^N$. While the form of $\cL_D^N$ follows by an application of It\^o's formula to \eqref{sde}, the jump operator \eqref{gen-2} can be written in terms of the jump rates $r_{ij}^N$ in a way that is familiar in interacting particle systems \cite{KL}. The total generator $\cL^N$ is then the sum of the two. The It\^o-Dynkin formula states that $t\mapsto F(\zeta_t^N)-F(\zeta_0^N)-\int_0^t \cL^NF(\zeta_s^N)ds$ is a $\cG_t$-martinagle. Naturally this martingale can be decomposed into two parts, one due to diffusion and the other one due to jump.}
\end{remark}

For the kernel function $\theta: \R^3\to\R_+$ that regulates the coalescence rate \eqref{abstract rate}, it is enough to assume that it is H\"older continuous $\mathsf C^\alpha(\R^3)$ for some $\alpha\in(0,1)$, compactly supported in the unit ball $B(0,1)\subset\R^3$ with $\theta(x)=\theta(-x)$, $\int_{\R^3} \theta (x) dx =1$,  and $\theta(0)=0$. Throughout, we  denote its rescaled version by
\[
\theta^\ep(x):=\ep^{-3}\theta(\ep^{-1}x),
\]
which has compact support in a ball $B(0,\ep)$ of radius $\ep$. With this notation, the coalescence rate can be written as
\begin{align}\label{rate}
r^N_{ij}=\frac{R_{0}}{N}\theta^\ep(x_i-x_j).
\end{align}

Denote the process of empirical measure of the {\it active} particles by
\begin{align}\label{empirical}
\mu^N_t(dx):=\frac{1}{N}\sum_{i\in\cN(\zeta_t^N)}\delta_{x_i^N(t)}(dx) \quad \in \cD\left([0,\infty); \cM_{+,1}(\R^3)\right),\quad t\ge0,
\end{align}
where $\cM_{+,1}(\R^3)$ denotes the space of sub-probability measures over $\R^3$ endowed with the weak topology, and $\cD\left([0,\infty); \cM_{+,1}(\R^3)\right)$ denotes the Skorohod space of c\`adl\`ag processes taking values in $\cM_{+,1}(\R^3)$ endowed with the Skorohod topology.

We introduce the limiting \abbr{PDE} with unknown $f(t,x): [0,T]\times\R^3\to\R_+$:
\begin{align}\label{limit-pde}
\begin{cases}
\partial_tf(t,x)&=\frac{1}{2}\left(\sigma_0^2+\sigma^2\right)\Delta f(t,x)-\ovl\cR f(t,x)^2,\\[5pt]
f(0,x)&=f_0(x),
\end{cases}
\end{align}
where $\ovl\cR\in\R_+$ is given by
\begin{align}\label{corr-bis}
\ovl\cR=\int_{\R^3}R_0\theta(x)\left(1+u(x)\right)dx.
\end{align}
and $u:\R^3\to\R$ solves the auxiliary \abbr{PDE}
\begin{align}\label{cell-intro}
\sigma_0^2\Delta u(x)+\sigma^{2}\nabla\cdot\big(\omega(\xi x)\nabla u(x)\big)  =R_{0}%
\theta(x)\big(  1+u(x) \big).
\end{align}
We will comment later that \eqref{cell-intro} has a $\mathsf C^{2, \alpha}(\R^3)$ solution that  satisfies $u(x)\in[-1,0]$ for all $x\in\R^3$, see \eqref{exp-sol} and Lemma \ref{lem:sol-bounds}, and is unique under the constraint that $|u(x)|\to0$ as $|x|\to\infty$.
\begin{remark}
An alternative expression for $\ovl\cR$ is
\begin{align}\label{corr}
\ovl\cR:=\left(\sigma_0^2+\sigma^2\right)\int_{\R^3}\Delta u(x)dx,
\end{align}
and it is in this form that our proof will yield.  It can be seen from integration by parts in \eqref{cell-intro}; indeed since $\omega(x)=I_3-C (x)$ and $C $ has compact support, we have 
\begin{align*}
(\sigma_0^2+\sigma^2)\int_{\R^3}\Delta u(x)dx& = \sigma^2\int_{\R^3}\nabla\cdot \Big(C (\xi x)\nabla u(x)\Big)dx+\int_{\R^3}R_0\theta(x)(1+u(x))dx\\
&=\int_{\R^3}R_0\theta(x)\big(1+u(x)\big)dx.
\end{align*}
\end{remark}

Solutions to \eqref{limit-pde} are understood in the following weak sense. 
\begin{definition}\label{def:weak}
Assume $f_0\in L^1\cap L^\infty(\R^3;\R_+)$. We say that $f\in L^\infty\left([0,T]; L^1(\R^3;\R_+)\right)$ is a weak solution to \eqref{limit-pde} if 
\begin{align*}
\int_{\R^3} f(t,x)\phi(x)dx &=\int_{\R^3} f_0(x)\phi(x)dx \\&\quad\quad + \frac{1}{2}(\sigma_0^2+\sigma^2)\int_0^t\int_{\R^3} f(s,x)\Delta \phi(x)dxds-\ovl\cR\int_0^t\int_{\R^3} \phi(x)f(s,x)^2dxds
\end{align*}
holds for any $t\in[0,T]$ and all test function $\phi\in C^2_c(\R^3)$.
\end{definition}
One can show that such weak solutions are unique, for a proof in a very similar context see \cite[Section 3]{fh}. The main result of this article is as follows.

\begin{theorem}\label{thm:main}
Assume that the initial density $f_0$ satisfies Condition \ref{cond:ini} and the scaling relation \eqref{abstract scaling} holds. Then for every $T<\infty$, the empirical measure $(\mu^N_t(dx))_{t\in[0,T]}$  \eqref{empirical} of the particle system converges in probability as $N\to\infty$, in the space $\cD\big([0,T]; \cM_{+,1}(\R^3)\big)$, to a limit $\{\ovl \mu_t(dx)\}_{t\in[0,T]}$. The limit is deterministic, absolutely continuous with respect to Lebesgue measure for every $t$, i.e. $\ovl \mu_t(dx)=f(t,x)dx$, where the density $f(t,x)$ is the unique weak solution of the \abbr{PDE} \eqref{limit-pde}, in the sense of Definition \ref{def:weak}.
\end{theorem}
This result, in particular the form of the average coalescence rate $\ovl\cR$, is a variant of \cite{hr}. Compared to that work, we have simplified the model by focusing only on particles of the same mass, but we have introduced a random environment to model the turbulent fluid. In our previous model \cite{fh-alea} in a random environment, we have chosen the dense particle scaling $\ep^3 N=1$ which, in the framework studied here, would lead to $\ovl\cR=R_0$.

\subsection{Sketch of proof}\label{sec:sketch}
Since the proof of our main result is somewhat technical, we will first provide a general sketch in this section. 

Our goal is to show that the empirical measure process, $(\mu_t^N(dx))_{t\in[0,T]}$, which is a measure-valued stochastic process, converges to the unique weak solution of the \abbr{PDE} \eqref{limit-pde}. Granted that we have relative compactness of this sequence, we need to show that any subsequential limit is a weak solution of the \abbr{PDE}. Since weak formulation of the \abbr{PDE} is  in terms of integration against test functions $\phi$, it is natural and standard to look at the real-valued stochastic process 
\[
\langle \mu_t^N, \phi\rangle
\]
for any fixed $\phi$, and show that asymptotically it satisfies  the weak formulation of the \abbr{PDE} with test function $\phi$.

At the level of fixed $N$, by applying It\^o-Dynkin formula to $\langle \mu_t^N, \phi\rangle$, we can readily get that it satisfies a so-called semimartingale equation. Such equation is still random, with fluctuation terms present, and it only partially looks like the limit \abbr{PDE}. In Section \ref{sec:emp}, we will show that the fluctuation terms (they are martingales) vanish in the limit $N\to\infty$ in a probabilistic sense, so that the limit is deterministic and not a stochastic \abbr{PDE}. We will also see that all the other terms, but one of them \eqref{id-emp}, are in the form of the empirical measure $\mu_t^N$ integrated against certain test function. These are terms linear in the empirical measure, and will converge to corresponding linear terms of the \abbr{PDE}, just because the empirical measure converges weakly (along subsequences). 

There is however one term \eqref{id-emp} which is not linear in the empirical measure. The desired limit for this term is quadratic in the \abbr{PDE} solution, with a particular multiplicative constant $\ovl \cR$, but it is not at all clear if and how this stochastic term will converge. This difficulty is common in the theory of interacting particle systems, see \cite{KL}, namely the main work is always concentrated on proving that the nonlinear terms converge. For particle systems evolving on lattices, there are methods bearing the name of ``replacement lemma'' that deal with such issues, but  that requires detailed knowledge of invariant measures of the particle system. Lacking such knowledge in our setting of interacting diffusions in the continuum, we instead rely on a very special technique due to Hammond-Rezakhanlou \cite{hr}. We call it It\^o-Tanaka trick, with roots in Tanaka's formula in stochastic calculus; this is developed in Sections \ref{sec:cell}-\ref{sec:neg} of our paper. The key statement to prove here is Proposition \ref{ppn:stoss}, which directly dictates the form of the multiplicative constant $\ovl\cR$ in the limit \abbr{PDE}, see the argument following it, \eqref{interm-2}. However, we do not have enough intuitive or physical understanding of what this proposition means. Vaguely, we may say that it is also a sort of replacement lemma. From Proposition \ref{ppn:stoss}, it is relatively easy to deduce the desired convergence of the nonlinear stochastic term, because it allows to disentangle $\ep$ and $N$ by introducing the auxiliary variable $z$.

The remaining effort is thus concentrated on proving Proposition \ref{ppn:stoss}. Let us look back at the nonlinear stochastic term \eqref{id-emp}, which is in the form of an averaged double-sum involving $\theta^\ep(x_i^N(t)-x_j^N(t))$. Recall that $\ep\to0$ very fast as $N\to\infty$, and $\theta^\ep$ is approaching the Dirac delta. Furthermore, the argument of $\theta^\ep$ is the two-point motion that is responsible for coalesence. Due to the singular nature of $\theta^\ep$, this expression is highly unstable. The idea of It\^o-Tanaka trick is that instead of working with $\theta^\ep$, we can equivalently work with a much more regular function. Note that regularity here should be understood as not for fixed $\ep$, but as $\ep\to0$. In Section \ref{sec:cell}, we construct an auxiliary function $u^\ep$, solution of a carefully-chosen, uniformly elliptic, second-order, divergence form \abbr{PDE} \eqref{cell} (the ``cell equation''), whose right-hand side contains $\theta^\ep$. As a result, $u^\ep$ is ``two derivatives'' more regular than $\theta^\ep$ and not a singular function at all. By rescaling the cell equation, we can obtain an $\ep$-independent elliptic equation \eqref{cell-unscaled} that can be solved using Green's function.  Here, the capacity scaling relation \eqref{abstract scaling} between $\ep$ and $N$ plays an important role in correctly rescaling the cell equation. Then, in Section \ref{sec:tanaka} we consider a new averaged double-sum process \eqref{functional-3}, similar to the nonlinear term \eqref{id-emp} we started with, but now involving the more regular function $u^\ep$. By applying the It\^o-Dynkin formula to this new process, and performing substitution using the cell equation,  we obtain an identity in which the pre-limit of Proposition  \ref{ppn:stoss} appears, together with many other (presumably minor) terms. Now, to show Proposition \ref{ppn:stoss}, it is enough to control all these minor terms, since we have an identity in this manouvre. More precisely, we need to prove that all the minor terms are negligible in suitable probabilistic sense. Since they all involve $u^\ep$ and $\nabla u^\ep$, for which we have enough control through Green's functions, we can prove their negligibility which is the content of Section \ref{sec:neg}.

\medskip
The rest of the article is organized as follows. In Section \ref{sec:emp} we write the equation satisfied by the empirical measure of the particle system, and identify the nonlinear term whose convergence is the main issue. In Sections \ref{sec:cell} and \ref{sec:tanaka} we introduce the cell equation and the It\^o-Tanaka trick, main ingredients in the proof of Theorem \ref{thm:main}. In Section \ref{sec:neg} we show the negligibility of all the minor terms coming out of It\^o-Tanaka trick. We omit the proof of tightness of the particle system, the existence of limit density and the uniqueness of the limit \abbr{PDE}, as they are similar to arguments detailed in our previous works \cite{fh, fh-alea, h-ejp}. In Section \ref{sec:potential} we tie our main theorem with certain facts in potential theory, so as to deduce the dependence of coalesence efficiency on $\xi$.

\section{Equation satisfied by the empirical measure}\label{sec:emp}

Now, we are ready to derive a semimartingale equation for the empirical measure process \eqref{empirical}. Fixing $T<\infty$ and any $\phi\in C^2_c(\R^3)$, we apply the It\^o-Dynkin formula (with infinitesimal generator given in \eqref{generator}) to the functional of configurations
\begin{align*}
F_1(\zeta_t^N):=\langle\mu_t^N, \phi\rangle = \frac{1}{N}\sum_{i\in\cN(\zeta_t^N)}\phi\left(x_i^N(t)\right), \quad t\ge0,
\end{align*}
where we use the shorthand $\langle \mu, f\rangle:=\int f d\mu$ for pairings between a function and a measure,
and we obtain that 
\begin{align}
&\langle\mu_T^N, \phi\rangle- \langle\mu_0^N, \phi\rangle \label{linear-1}\\
&=\int_0^T\frac{1}{N}\frac{\sigma^2}{2}\sum_{i\in\cN(\zeta_t^N)}\nabla\cdot\big(C_\ep(\xi(x_i^N(t)-x_i^N(t)))\nabla \phi(x_i^N(t))\big)dt\\
&+\int_0^T\frac{1}{N}\frac{\sigma^2_0}{2}\sum_{i\in\cN(\zeta_t^N)}\Delta\phi(x_i^N(t))dt\label{linear-3}\\
&+\int_0^T\frac{1}{N}\sigma\sum_{i\in\cN(\zeta_t^N)}\nabla \phi\left(x_i^N(t)\right)\cdot\sum_{k\in  \N}\lambda_ke^\ep_k(\xi x_i^N(t))dW_k(t)\label{diff-mg-1}\\
&+\int_0^T\frac{1}{N}\sigma_0\sum_{i\in\cN(\zeta_t^N)}\nabla \phi(x_i^N(t))\cdot dB_i(t)\label{diff-mg-2}\\
&-\int_0^T\frac{1}{2}\frac{R_0}{N^2}\sum_{i\in\cN(\zeta_t^N)}\sum_{j\in\cN(\zeta_t^N): j\neq i}\theta^\ep(x_i^N(t)-x_j^N(t))\big[\phi(x_i^N(t))+\phi(x^N_j(t))\big]dt  \label{id-emp}\\
&+M_J^N(T),\nonumber
\end{align}
where, \red{while \eqref{diff-mg-1}-\eqref{diff-mg-2} are martinagles associated with diffusion dynamics,}  $M_J^N(t)$ is a martingale associated with jumps, that we do not write out explicitly. Since $C_\ep(0)=I_3$, we have 
\begin{align*}
\frac{1}{N}\frac{\sigma^2}{2}\sum_{i\in\cN(\zeta_t^N)}\nabla\cdot\big(C_\ep(\xi(x_i^N(t)-x_i^N(t)))\nabla \phi(x_i^N(t))\big)=\frac{\sigma^2}{2}\langle\mu_t^N, \Delta\phi\rangle.
\end{align*}
Also, since we assumed that $\theta$ is symmetric, we have
\begin{align*}
&\frac{1}{2}\frac{R_0}{N^2}\sum_{i\in\cN(\zeta_t^N)}\sum_{j\in\cN(\zeta_t^N): j\neq i}\theta^\ep(x_i^N(t)-x_j^N(t))\big[\phi(x_i^N(t))+\phi(x^N_j(t))\big]\\
&=\frac{R_0}{N^2}\sum_{i\in\cN(\zeta_t^N)}\sum_{j\in\cN(\zeta_t^N): j\neq i}\theta^\ep(x_i^N(t)-x_j^N(t))\phi(x_i^N(t)).
\end{align*}
It turns out that the convergence of the linear terms \eqref{linear-1}-\eqref{linear-3} is not an issue (see e.g. \cite{fh-alea, h-ejp}), and we will prove in Lemmas \ref{lem:mg-3iff}-\ref{lem:mg-jump} that the martingale terms vanish, as $N\to\infty$. The only difficulty is the convergence of the nonlinear term \eqref{id-emp}.

\medskip
We introduce a tool that will be useful also for later sections. Consider on the same filtered probability space $\big(\Omega, \cF, \{\cG_t\}_{t\ge0}, \P\big)$ an anxiliary ``free'' particle system $y_i^\ep(t)$, $i\in \N$ which start at $t=0$ with the same initial condition as the true system $y_i^\ep(0)=x_i(0)$, perform the same \abbr{SDE} \eqref{sde} driven by the same collections of Brownian motions $\{W_k(t)\}_{k\in  \N}$ and $\{B_i(t)\}_{i=1}^\infty$, but here the particles do not interact/annihilate, i.e. 
\begin{align}\label{free-sys}
dy_i^\ep(t)=\sigma\sum_{k\in  \N}\lambda_k e_{k}^{\epsilon}\left(\xi y_i^\ep(t)\right)  dW_k(t)+\sigma_0\, dB_i(t), \quad t\ge0, \;i=1,2,...
\end{align}
with infinite life time. Note that $y_i^\ep$ still depends on $\ep$.

Since the initial conditions and dynamics of the true and auxiliary systems agree, there is an obvious coupling between them under which whenever $x_i^N(t)$ is alive at time $t$, its trajectory on the time interval $[0,t]$ coincides with that of $y_i^\ep(t)$, for every $i=1,2,..,N$. By construction, the auxiliary system is exchangeable, \red{meaning that the law of $\big(y_{i_1}^\ep(t), y_{i_2}^\ep(t), ..., y_{i_L}^\ep(t)\big)_{t\ge0}$ is unchanged under any permutation of the indices $(i_1,i_2,...,i_L)$ and any finite $L\in \N$}. Now consider a pair $\big(y_1^\ep(t), y^\ep_2(t)\big)$ and it is known classically that it has a joint density $f_t^{12,\ep}(y_1,y_2)$ that evolves by the forward Kolmogorov equation  (or Fokker-Planck equation)
\begin{align*}
\begin{cases}
\partial_tf_t^{12,\ep}(\underline y)&=\cL_\ep^*f_t^{12,\ep}(\underline y), \\[5pt]
 f_0^{12,\ep}(\underline y)&=f_0(y_1)f_0(y_2), \quad\quad  \underline y:=(y_1,y_2)\in(\R^3)^2,
 \end{cases}
\end{align*}
where $\cL^*_\ep$ is a uniformly negative divergence-form operator that acts on functions $g:(\R^3)^2\to\R$ by
\begin{align*}
\cL_\ep^*g\big(\underline y\big)=\nabla_{\underline y}\cdot \big(A_\ep\big(\underline y\big)\nabla_{\underline y}\big) 
\end{align*}
where 
\begin{align*}
\nabla_{\underline y}&:=\big(\nabla_{y_1}, \nabla_{y_2}\big), \\
A_\ep(\underline y)&:=\frac{1}{2}\sigma_0^2I_6+\frac{1}{2}\sigma^2\begin{bmatrix}
I_3  &C_\ep(\xi(y_2-y_1))\\
C_\ep(\xi(y_1-y_2)) & I_3\end{bmatrix}\in\M_6.
\end{align*}
Although $\cL_\ep^*$ depends on $\ep$, its matrix coefficient $A_\ep$ are measurable and  bounded uniformly in $\ep$. Indeed, $C_\ep(x)=C(x/\ep)$ and $C(x)$ is assumed smooth and compactly supported. By Aronson's estimate \cite[Theorem 1]{aron} for fundamental solutions of uniformly parabolic second-order divergence-form operator with measurable bounded coefficients, there exist some finite constants $\mathsf K_i=\mathsf K_i \left(\sigma_0, \sigma, \|C\|_{L^\infty}\right)$ (depending only on the bounds on the coefficients, hence independent of $\ep$), $i=1,2$, such that  
\begin{align*}
q^{\ep}\big(t; \underline y_0, \underline{y}\big)\le \frac{\mathsf K_1}{(2\pi \mathsf K_2t)^3}e^{-\frac{|\underline y-\underline{y}_0|^2}{2\mathsf K_2t}},
\end{align*}
where $q^{\ep}\big(t; \underline y_0, \underline{y}\big)$ denotes the fundamental solution of $\partial_t-\cL_\ep^*$, and $\|C\|_{L^\infty}:=\max_{\alpha,\beta=1}^3 \|C^{\alpha\beta}\|_{L^\infty(\R^3)}$.

This allows to derive exponential tail bounds on $f_t^{12,\ep}$. Indeed, let $Y_t$ denote a Gaussian random vector on $(\R^3)^2$ with density $\frac{1}{(2\pi \mathsf K_2t)^3}e^{-\frac{|\underline y|^2}{2\mathsf K_2t}}$, then 
\begin{align*}
f_t^{12,\ep}\big(\underline y\big)&=\int_{(\R^3)^2} f_0^{12,\ep}\big(\underline {y}_0\big)q^{\ep}\big(t; \underline y_0, \underline{y}\big)d\underline{y}_0\\
&\le \int_{(\R^3)^2} f_0^{12,\ep}\big(\underline {y}_0\big)\frac{\mathsf K_1}{(2\pi \mathsf K_2t)^3}e^{-\frac{|\underline y-\underline{y}_0|^2}{2\mathsf K_2t}} d\underline{y}_0\\
&=\mathsf{K_1}\E \left[f_0^{12,\ep}(\underline{y}-Y_t)\right],
\end{align*}
where $\E$ acts on the random variable $Y_t$.
Since $f_0^{12,\ep}(\underline y)=f_0(y_1)f_0(y_2)$ which is uniformly bounded by $\Gamma^2$ and compactly supported in $B(0, \sqrt{2}L_0)$ by Condition \ref{cond:ini}, we have 
\begin{align*}
f_t^{12,\ep}\big(\underline y\big)\le \mathsf{K_1}\E \left[f_0^{12,\ep}(\underline{y}-Y_t)\right]&\le \mathsf{K_1}\Gamma^2\P\left(|\underline{y}-Y_t|\le \sqrt{2}L_0\right)\\
&\le \mathsf{K_1} \Gamma^2\P\left(|Y_t|\ge |\underline y|-\sqrt{2}L_0\right)\\
&\le 2\mathsf{K_1} \Gamma^2 e^{-\frac{\left( |\underline y|-\sqrt{2}L_0\right)_+^2}{4d\mathsf{K_2} t}}\le \mathsf{K}'e^{-|\underline y|^2/\mathsf{K}'},
\end{align*}
for any $t\in[0,T]$ and some finite constant $\mathsf{K}'$ that depends only on $\mathsf{K_1}, \mathsf{K_2}, \Gamma, T, L_0$, where we used the tail bound for a Gaussian vector.

A similar argument can also show that the joint density $f_t^{12...\ell,\ep}(y_1,...,y_\ell)$ of an $\ell$-tuple of free particles $(y_1^\ep(t), ..., y_\ell^\ep(t))$, for any fixed $\ell\in \N$, also satisfies an exponential decay, uniformly for $t\in[0,T]$ and $\ep$. We will only use $\ell=1,2,3,4$ in the sequel. That is, we have proven that
\begin{proposition}\label{ppn:gaussian}
There exists a finite constant $C_0=C_0\left(\sigma_0, \sigma, \|C\|_{L^\infty}, \Gamma, T, L_0, \ell\right)$ such that for any $t\in[0,T]$ and $\ep\in(0,1)$, we have 
\begin{align}
f_t^{1...\ell,\ep}(y_1,y_2,...,y_\ell)&\le C_0e^{-\frac{|y_1|^2+...+|y_\ell|^2}{C_0}}, \label{2-joint}
\end{align}
for $\ell=1,2,3,4$.
\end{proposition}

We proceed to show next that the two martingale terms associated with diffusion vanish in $L^2(\P)$. 
\begin{lemma}\label{lem:free}
\begin{align*}
\E\Bigg|\int_0^T\frac{1}{N}\sum_{i\in\cN(\zeta_t^N)}\nabla \phi\left(x_i^N(t)\right)\cdot\sum_{k\in  \N}\lambda_ke^\ep_k(x_i^N(t))dW_k(t)\Bigg|^2=O\left(N^{-1}\right)+O\left(\ep^3\right).
\end{align*}
\end{lemma}
\begin{proof}
By It\^o isometry and the independence between $W_k(t)$, $k\in  \N$, 
\begin{align*}
&\E\Bigg|\int_0^T\frac{1}{N}\sum_{i\in\cN(\zeta_t^N)}\nabla \phi\left(x_i^N(t)\right)\cdot\sum_{k\in  \N}\lambda_ke^\ep_k(\xi x_i^N(t))dW_k(t)\Bigg|^2\\
&=\E\int_0^T\frac{1}{N^2}\sum_{k\in  \N}\Bigg|\sum_{i\in\cN(\zeta_t^N)}\nabla \phi\left(x_i^N(t)\right)\cdot \lambda_ke^\ep_k(\xi x_i^N(t))\Bigg|^2dt\\
&=\E\int_0^T\frac{1}{N^2}\sum_{k\in  \N}\sum_{i,j\in\cN(\zeta_t^N)}\nabla\phi(x_i^N(t))^T \lambda_ke^\ep_k(\xi x_i^N(t))\otimes \lambda_ke^\ep_k(\xi x_j^N(t))\nabla\phi(x_j^N(t))dt\\
&=\E\int_0^T\frac{1}{N^2}\sum_{i,j\in\cN(\zeta_t^N)}\nabla\phi(x_i^N(t))^TC_\ep(\xi (x_i^N(t)-x_j^N(t)))\nabla\phi(x_j^N(t))dt\\
&\le \frac{T}{N}\|\phi\|_{C^1}^2+\E\int_0^T\frac{9}{N^2}\sum_{i\neq j\in\cN(\zeta_t^N)}\|\phi\|_{C^1}^2\Big\|C_\ep(\xi(x_i^N(t)-x_j^N(t)))\Big\|_\infty dt,
\end{align*}
where for a $3\times 3$ matrix $A$ we denote 
\begin{align}\label{matrix-max}
\|A\|_\infty:=\max_{\alpha,\beta=1}^3|A^{\alpha\beta}|.
\end{align}
Recall the auxiliary free particle system \eqref{free-sys} just introduced, which is exchangeable in law and coupled to $x_i^N$ for $i\in\cN(\zeta_t^N)$ in a natural way, hence we have the bound
\begin{align*}
&\E\int_0^T\frac{1}{N^2}\sum_{i,j\in\cN(\zeta_t^N)}\left\| C_\ep(\xi(x_i^N(t)-x_j^N(t)))\right\|_\infty dt\\
&\le\E\int_0^T\frac{1}{N^2}\sum_{i\neq j=1}^N\left\| C_\ep(\xi(y_i^\ep(t)-y_j^\ep(t)))\right\|_\infty dt\\
&\le\E\int_0^T\left\|C_\ep\left(\xi(y_1^\ep(t)- y_2^\ep(t))\right)\right\|_\infty dt.
\end{align*}
By the tail bound on the pair density \eqref{2-joint},
\begin{align*}
&\le T\iint_{(\R^3)^2}\| C\left(\ep^{-1}\xi(y_1-y_2)\right)\|_\infty C_0e^{-\frac{|y_1|^2+|y_2|^2}{C_0}} dy_1dy_2dt\\
&=C_0T\int_{\R^3}\|C\left(\ep^{-1}\xi y_1\right)\|_\infty dy_1\int_{\R^3}e^{-\frac{|y_2|^2}{C_0}} dy_2\\
&\lesssim \ep^3 \int_{\R^3} \|C\left(\xi y\right)\|_\infty dy=O(\ep^3),
\end{align*}
where the factor $\ep^3$ comes from change of variables, and the last integral is finite since every component of $C$ is smooth with compact support (hence bounded). This completes the proof.
\end{proof}

\begin{lemma}\label{lem:mg-3iff}
\begin{align*}
\E\Big|\int_0^T\frac{1}{N}\sigma_0\sum_{i\in\cN(\zeta_t^N)}\nabla \phi(x_i^N(t))\cdot dB_i(t)\Big|^2=O\left(N^{-1}\right).
\end{align*}
\end{lemma}

\begin{proof}
By It\^o isometry, we have 
\begin{align*}
&\E\Big|\int_0^T\frac{1}{N}\sigma_0\sum_{i\in\cN(\zeta_t^N)}\nabla \phi(x_i^N(t))\cdot dB_i(t)\Big|^2\\
&=\E\sum_{i\in\cN(\zeta_t^N)}\int_0^T\frac{1}{N^2}\sigma_0^2 |\nabla \phi(x_i^N(t))|^2dt\\
&\le T\sigma_0^2N^{-2}\|\phi\|_{C^1}^2 \E|\cN(\eta(t))|\le T\sigma_0^2N^{-1}\|\phi\|_{C^1}^2,
\end{align*}
since $|\cN(\eta(t))|\le |\cN(\eta(0))|=N$.
\end{proof}

\medskip
We show next that the martinagle associated with jumps also vanishes in $L^2(\P)$.
\begin{lemma}\label{lem:mg-jump}
$\E|M_J^N(T)|^2=O\left(N^{-1}\right)$.
\end{lemma}
\begin{proof}
By the carr\'e-du-champ formula cf. \cite[Proposition 8.7]{DN}, we have
\begin{align*}
\E|M_J^N(T)|^2&\le \E\int_0^T\frac{R_0}{2N}\sum_{i\in\cN(\zeta_t^N)}\sum_{j\in\cN(\zeta_t^N):j\neq i}\theta^\ep(x_i^N(t)-x_j^N(t))\frac{1}{N^2}\big[\phi(x_i^N(t))+\phi(x_j^N(t))\big]^2dt \\
&\le \E\int_0^T\frac{2R_0\|\phi\|_{L^\infty}^2}{N^3}\sum_{i\in\cN(\zeta_t^N)}\sum_{j\in\cN(\zeta_t^N):j\neq i}\theta^\ep(x_i^N(t)-x_j^N(t))dt.
\end{align*}
On the other hand, we consider the functional $F_2(\zeta_t^N)=|\cN(\zeta_t^N)|$, apply the generator \eqref{generator} to it, and we get 
\begin{align*}
\E|\cN(\eta_T^N)|-\E|\cN(\eta_0^N)|= -\E\, 2\int_0^T\frac{R_0}{2N}\sum_{i\in\cN(\zeta_t^N)}\sum_{j\in\cN(\zeta_t^N):j\neq i}\theta^\ep(x_i^N(t)-x_j^N(t))dt.
\end{align*}
Indeed, the diffusion part of the generator does not affect $F_2(\zeta_t^N)$, only the jump part does. It yields that 
\begin{align*}
\E\int_0^T\frac{R_0}{N}\sum_{i\in\cN(\zeta_t^N)}\sum_{j\in\cN(\zeta_t^N):j\neq i}\theta^\ep(x_i^N(t)-x_j^N(t))dt\le \E|\cN(\eta_0)|=N.
\end{align*}
Therefore we have $\E|M_J^N(T)|^2\le 2\|\phi\|_{L^\infty}^2N^{-1}$.
\end{proof}

\section{The cell problem and the nonlinear term}\label{sec:cell}
The remaining task is to prove the convergence of the nonlinear term \eqref{id-emp}. We aim to establish a form of local equilibrium as in Proposition \ref{ppn:stoss} below. For this purpose, we introduce for every fixed $\ep\in(0,1)$ an auxiliary equation $u^{\ep}(x):\R^3\to\R$ (so-called cell problem in the terminology of homogenization):
\begin{align}\label{cell}
\sigma_0^2\Delta u^\ep (x)+\sigma^{2}\nabla\cdot\Big(\omega\big(\frac{\xi x}{\ep}\big)  \nabla u^{\ep}\left(  x\right)\Big)  =R_{0}%
\theta^\ep\left(  x\right)  \Big(  1+\frac{u^{\ep}\left(  x\right) }{N} \Big) ,
\end{align}
which is a rescaling  of an $\ep$-independent equation \eqref{cell-unscaled}. Let $u(x):\R^3\to\R$ be a particular $\mathsf C^{2, \alpha}(\R^3)$-solution given in \eqref{exp-sol}, of 
\begin{align}\label{cell-unscaled}
\sigma_0^2\Delta u(x)+\sigma^{2}\nabla\cdot\big(\omega(\xi x)\nabla u(x)\big)  =R_{0}%
\theta(x)\big(  1+u(x) \big),
\end{align}
then it can be readily checked, using $N=\ep^{-1}$, that the rescaled function
\begin{align}\label{u-rescaling}
u^\ep(x)=\ep^{-1}u\left(\frac{x}{\ep}\right)
\end{align}
solves \eqref{cell}. This is the analogous cell equation as used in \cite[Eq. (1.7)]{hr}.

The following proposition will be proved in Section \ref{sec:tanaka} using the It\^o-Tanaka trick. Note that we have added an extra test function $\psi$, with respect to  \eqref{id-emp}, to have more localization. It is actually without loss of generality. Indeed, $\theta^\ep(x)$ is  compactly supported in $B(0,\ep)$ and from \eqref{cell} it can be seen \eqref{cell-separate} that $\Delta u^\ep (x)$ is compactly supported in $B(0,\ep(1\vee\xi^{-1}))$, and we can always assume that $|z|\le 1/2$, which implies that each summand 
\[
R_0\theta^\ep\big(x_i^N(t)-x_j^N(t)\big)-\left(\sigma_0^2+\sigma^2\right)\Delta u^\ep\big(x_i^N(t)-x_j^N(t)+z\big)
\]
is zero unless $|x_i^N(t)-x_j^N(t)|\le 1/2+\ep(1\vee\xi)$. Since $x_i^N(t)\in \text{supp}(\phi)$, this forces $x_j^N(t)$ to be inside some other compact set $\K=\K(\phi,\xi)\subset\R^3$. Hence we can choose a compactly supported test function $\psi$ that is identically $1$ on $\K$ and decays smoothly to $0$ outside. Such a $\psi(x_j^N(t))$ does not alter \eqref{stoss}.
\begin{proposition}\label{ppn:stoss}
Let $u^{\ep}(x)\in\mathsf C^{2, \alpha}(\R^3)$ be given by \eqref{u-rescaling}  and $\phi, \psi\in C_c^2(\R^3)$ be two test functions, then we have
\begin{align}\label{stoss}
&\lim_{|z|\to0}\limsup_{N\to\infty}\nonumber\\
&\E\Bigg|\int_0^T\frac{1}{N^2}\sum_{i, j\in\cN(\zeta_t^N): j\neq i}\Big[R_0\theta^\ep\big(x_i^N(t)-x_j^N(t)\big)-\left(\sigma_0^2+\sigma^2\right)\Delta u^\ep\big(x_i^N(t)-x_j^N(t)+z\big)\Big]\phi(x_i^N(t))\psi(x_j^N(t))dt\Bigg|=0.
\end{align}
\end{proposition}

We proceed directly to showing that Proposition \ref{ppn:stoss} yields the convergence of the nonlinear term, as in \cite[page 42-43]{hr}. Firstly, we note that we can add the terms with indices $i=j$ into the double summation of \eqref{stoss}, without changing its conclusion. Indeed, there are at most $N$ diagonal terms. With $\theta^\ep(0)=0$, $|\Delta u^\ep(z)|\lesssim |z|^{-3}$ provided $|z|\ge 2\ep$ by \eqref{unif-hess-bd}, and a factor of $1/N^2$ in front, we see that as $\ep\to0$ first (and $z$ fixed), the whole contribution of 
\[
\frac{1}{N^2}\sum_{i\in\cN(\zeta_t^N)}\Big|R_0\theta^\ep\big(x_i^N(t)-x_i^N(t)\big)-\left(\sigma_0^2+\sigma^2\right)\Delta u^\ep\big(x_i^N(t)-x_i^N(t)+z\big)\Big|\lesssim N^{-1}|z|^{-3} \to 0.
\]
Hence, we can start our subsequent argument with a version of \eqref{stoss} with full double summation:
\begin{align}\label{stoss-bis}
&\lim_{|z|\to0}\limsup_{N\to\infty}\nonumber\\
&\E\Big|\int_0^T\frac{1}{N^2}\sum_{i,j \in\cN(\zeta_t^N)}\Big[R_0\theta^\ep\big(x_i^N(t)-x_j^N(t)\big)-\left(\sigma_0^2+\sigma^2\right)\Delta u^\ep\big(x_i^N(t)-x_j^N(t)+z\big)\Big]\phi(x_i^N(t))\psi(x_j^N(t))dt\Big|=0.
\end{align}
Denote by $\zeta:\R^3\to\R_+$ a fixed smooth probability density function with compact support in $B(0,1)$ and denote $\zeta^\delta(x):=\delta^{-3}\zeta(x/\delta)$ for $\delta>0$.  Then, \eqref{stoss-bis} yields that 
\begin{align}\label{averaging}
&\int_0^TR_0\left\langle \theta^\ep(x_1-x_2)\phi(x_1)\psi(x_2), \mu_t^N(dx_1)\mu_t^N(x_2)\right\rangle dt \nonumber\\
 &=\int_0^T\left(\sigma_0^2+\sigma^2\right)\iint_{(\R^3)^2}\left\langle \Delta u^\ep(x_1-x_2-z_1+z_2)\phi(x_1)\psi(x_2), \mu_t^N(dx_1)\mu_t^N(dx_2)\right\rangle \zeta^\delta(z_1)\zeta^\delta(z_2)dz_1dz_2 dt \nonumber\\ 
 &\quad\quad\quad\quad+\text{Err}(\delta,\ep),
\end{align}
where $\text{Err}(\delta,\ep)$ is a stochastic error that vanishes in the following limit: 
\begin{align}\label{error}
\lim_{\delta\to0}\limsup_{\ep\to0}\E|\text{Err}(\delta,\ep)|=0, 
\end{align}
and  may change from line to line in the sequel. In \eqref{averaging}, we have introduced two averaging on variable $z_1,z_2$ with  compact supports in $B(0,\delta)$, hence in view of \eqref{stoss-bis}, the rate of convergence of the error term \eqref{error} is uniformly controlled by $\delta$.

Next, we shift the argument of $\phi(x_1)$ and $\psi(x_2)$ on the right-hand side of \eqref{averaging} to $\phi(x_1-z_1)$ and $\psi(x_2-z_2)$ respectively, causing an error of $O(\delta)$, since $z_1,z_2$ are in the compact support of the function $\zeta^\delta$. Then, we perform a change of variable $w_1=x_1-z_1$, $w_2=x_2-z_2$, and the right-hand side of \eqref{averaging} (in the sequel written without the time integral, to ease notation) now becomes
\begin{align}\label{after-shift-1}
&\left(\sigma_0^2+\sigma^2\right)\iint_{(\R^3)^2}\Delta u^\ep(w_1-w_2)\phi(w_1)\psi(w_2)\left\langle\zeta^\delta(x_1-w_1)\zeta^\delta(x_2-w_2) , \mu_t^N(dx_1)\mu_t^N(x_2)\right\rangle dw_1dw_2+\text{Err}(\delta,\ep).
\end{align}
Next, we shift the argument of $\zeta^\delta(x_2-w_2)$ to $\zeta^\delta(x_2-w_1)$ in \eqref{after-shift-1},  causing an error on the order of $|w_2-w_1|\delta^{-4}$ by the mean-value theorem and $|\nabla\zeta^\delta|\lesssim \delta^{-4}$, and the argument of $\psi(w_2)$ to $\psi(w_1)$, causing an error on the order of $|w_2-w_1|$, and \eqref{after-shift-1} is equal to
\begin{align}
&=\left(\sigma_0^2+\sigma^2\right)\iint_{(\R^3)^2} \Delta u^\ep(w_1-w_2)\phi(w_1)\psi(w_1)\left\langle\zeta^\delta(x_1-w_1)\zeta^\delta(x_2-w_1), \mu_t^N(dx_1)\mu_t^N(x_2)\right\rangle dw_1dw_2\label{interm-2}\\
&\quad\quad\quad +\text{Err}(\delta,\ep)+\text{Err}_1(\delta,\ep),\nonumber
\end{align}
with an additional error term $\text{Err}_1$ that  we need to show is also negligible. Indeed, recalling that $\mu_t^N$ is a sub-probability, we firstly have 
\begin{align}
&|\text{Err}_1(\delta,\ep)|\lesssim \left(\sigma_0^2+\sigma^2\right)\iint_{(\R^3)^2}| \Delta u^\ep(w_1-w_2)||\phi(w_1)||\psi(w_2)||w_1-w_2|\delta^{-7}dw_1dw_2 \nonumber\\
&\quad\quad\quad  + \left(\sigma_0^2+\sigma^2\right)\iint_{(\R^3)^2} |\Delta u^\ep(w_1-w_2)||\phi(w_1)||w_1-w_2|dw_1dw_2. \label{interm-3}
\end{align}
Then, note  by the cell equation \eqref{cell} and since $\omega(x)=I_3-\ovl{C}(x)$, we have  
\begin{align}\label{cell-separate}
\left(\sigma_0^2+\sigma^2\right)\Delta u^\ep(x) = \sigma^{2}\nabla\cdot\Big(C _\ep(\xi x)  \nabla u^{\ep}\left(  x\right)\Big)  +R_{0}%
\theta^\ep\left(  x\right)  \Big(  1+\frac{u^{\ep}\left(  x\right) }{N} \Big).
\end{align}
Since $C _\ep$ and $\theta^\ep$ are compactly supported in $B(0,\ep/\xi)$ and $B(0,\ep)$ respectively, we have $\Delta u^\ep$ compactly supported in $B(0,\ep(1\vee\xi^{-1}))$.  Hence, recalling that $\Delta u^\ep(x) = \ep^{-3}\Delta u(x/\ep)$,   \eqref{interm-3} can be bounded by
\begin{align*}
|\text{Err}_1(\delta,\ep)|&\lesssim\iint_{(\R^3)^2}\big|\ep^{-3} \Delta u\big(\frac{w_1}{\ep}\big)\big||\phi(w_1+w_2)||\psi(w_2)||w_1|\delta^{-7}dw_1dw_2 \\
&\lesssim \ep \iint_{(\R^3)^2}\big|\ep^{-3} \Delta u\big(\frac{w_1}{\ep}\big)\big||\psi(w_2)|\delta^{-7}dw_1dw_2\\
&\lesssim\ep \int_{\R^3}|\Delta u(w_1)|\delta^{-7}dw_1=O(\ep\delta^{-7}),
\end{align*}
where since $\Delta u\in \mathsf C^{\alpha}(\R^3)$ and compactly supported in $B(0, 1\vee\xi^{-1})$, it is integrable. Thus, we see that $$\lim_{\delta\to0}\limsup_{\ep\to0}|\text{Err}_1(\delta,\ep)|=0.$$

Using the coupling with the auxiliary free system \eqref{free-sys}, see \cite[Sections 5-6]{fh-alea} or \cite[Section 6]{h-ejp} for arguments in similar contexts, it can be shown that the sequence of laws of $\{\mu_t^N\}_N$ is weakly compact, and any subsequential limit of these laws is concentrated on those measure-valued processes that are absolutely continuous  with respect to Lebesgue measure for every $t\ge0$, with a density that is uniformly bounded by the right-hand side of \eqref{2-joint} with $\ell=1$ (i.e. by the $1$-particle density of the free particle system). Let $\{\mu_t^{N_k}\}_k$ be such a converging subsequence, and let the limit density be called $f(t,x)$. Then for fixed $\delta$ we have as $N_k\to\infty$ (and $\ep(N_k)\to0$),
\[
\left\langle\zeta^\delta(x_1-w_1)\zeta^\delta(x_2-w_1), \mu_t^N(dx_1)\mu_t^N(x_2)\right\rangle\to \iint_{(\R^3)^2} \zeta^\delta(x_1-w_1)\zeta^\delta(x_2-w_1)f(t, x_1)f(t, x_2)dx_1dx_2.
\]
Recall the definition of $\ovl\cR$ \eqref{corr}. Now, taking $\ep\to0$ and then $\delta\to0$ in \eqref{interm-2},  the error terms  vanish in $L^1(\P)$ and the main integral tends first to
\begin{align*}
\ovl\cR \int_{\R^{d}} \phi(w_1)\psi(w_1)\iint_{(\R^3)^2} \zeta^\delta(x_1-w_1)\zeta^\delta(x_2-w_1)f(t, x_1)f(t, x_2)dx_1dx_2dw_1,
\end{align*}
and then to
\begin{align}\label{weak-form-nonlin}
\ovl\cR \int_{\R^3} \phi(w_1)\psi(w_1)f(t, w_1)f(t,w_1)dw_1,
\end{align}
since $\zeta^\delta$ approximates the delta-Dirac (here one needs the apriori fact that $f$ is bounded uniformly by a non-random constant, hence one can apply the dominated convergence theorem). 
In \eqref{weak-form-nonlin} we have obtained the desired weak formulation (in space) of the nonlinearity of \eqref{limit-pde}. 

In Section \ref{sec:emp} we already saw that in the identity satisfied by the empirical measure, the linear terms converge easily to the desired limiting linear terms (since they are in the form of a measure integrated against a test function) and the martingale terms vanish in $L^2(\P)$; the only issue is the convergence of the nonlinear term \eqref{id-emp}. In this section, we saw that provided Proposition \ref{ppn:stoss} is proved, we can also prove the convergence of the nonlinear term to the desired limiting nonlinear term, along every weakly converging subsequence of $\{\mu_t^N\}_{t\in[0,T]}$. Thus a quite standard weak convergence argument as in e.g.  \cite[page 636-637]{fh} or \cite[page 26-27]{h-ejp} establishes the subsequential convergence in distribution of $\{\mu_t^N\}_{t\in[0,T]}$ to a weak solution of \eqref{limit-pde}. Since it can be proved (see e.g. \cite[Section 3]{fh}) that weak solutions to the \abbr{PDE} in the sense of Definition \ref{def:weak} is unique hence deterministic, we conclude that the whole sequence $\{\mu_t^N\}_{t\in[0,T]}$ converges in probability. This is the proof of Theorem \ref{thm:main} (provided Proposition \ref{ppn:stoss} is shown).

\section{It\^o-Tanaka procedure}\label{sec:tanaka}
In order to prove the key Proposition \ref{ppn:stoss}, we need to introduce a particular functional on the particle configuration that contains the macroscopic variable $z$, and by applying the generator to such a functional, we wish to make appear the pre-limit in \eqref{stoss}. It is clear only a posteriori that it should contain an auxiliary function $v^{\ep,z}(x)$ \eqref{def-v} that is in the form of a difference of the function $u^\ep(x)$, shifted by $z$.

For every fixed $\ep\in(0,1)$ and $z\in\R^3$ with $|z|\le 1/2$, denote
\begin{align}\label{def-v}
v^{\ep,z}(x):=u^{\ep}(x+z)-u^{\ep}(x),
\end{align}
where $u^{\ep}(x)$ is a $\mathsf C^{2, \alpha}(\R^3)$-solution of the $\ep$-dependent cell problem  \eqref{cell}, given in \eqref{u-rescaling}. In particular, it is regular enough to apply It\^o's formula to. Consider the following functional on configurations
\begin{align}\label{functional-3}
F_3(\zeta_t^N):=\frac{1}{N^2}\sum_{i\in\cN(\zeta_t^N)}\sum_{j\in\cN(\zeta_t^N): j\neq i}v^{\ep,z}(x_i^N(t)-x_j^N(t))\phi(x_i^N(t))\psi(x_j^N(t)), \quad t\ge0,
\end{align}
and we proceed to apply the It\^o-Dynkin formula (with infinitesimal generator given in \eqref{generator}) to it. We note the following property of the function $C(x)$:
\begin{align}\label{equiv-3iv}
\sum_{\alpha=1}^3\partial_\alpha C^{\alpha\beta}(x) =0, \quad \text{for all }\beta=1,2,3, \;\; x\in\R^3.
\end{align}
Indeed, by \eqref{covariance} $C(x)=\sum_{k\in  \N}\lambda_ke_k(x)\otimes \lambda_ke_k(0)$.  Since $e_k(x)$ is divergence free, for every $k\in  \N$, we have 
\begin{align*}
\sum_{\alpha=1}^3\partial_\alpha C^{\alpha\beta}(x)=\sum_{k\in  \N}\lambda_k^2\sum_{\alpha=1}^3\partial_\alpha \left((e_k)^\alpha(x)(e_k)^\beta(0)\right)&=\sum_{k\in  \N}\lambda_k^2\sum_{\alpha=1}^3\left(\partial_\alpha (e_k)^\alpha(x)\right)(e_k)^\beta(0)\\
&=\sum_{k\in \N}\lambda_k^2{\rm{div}}\left(e_k\right) (e_k)^\beta(0)=0.
\end{align*}
The property \eqref{equiv-3iv} effectively implies that all the divergence forms in the expansion \eqref{long-expansion} below, are also equivalently non-divergence forms. 

\begin{remark}
The formula below is long, so let us first explain how one obtains it. Consider the first part of the diffusion generator $\cL^N_D$ \eqref{gen-1} (i.e. the part with covariance in it; the second part being more classical). It is in the form of a double sum, hence by linearity one can consider the action of each ``sub-generator'' $\frac{1}{2}\sigma^2\nabla_{x_{i_0}}\cdot \big(C_\ep(\xi(x_{i_0}-x_{j_0}))\nabla_{x_{j_0}}F(\eta)\big)$, for each fixed pair of indices $(i_0,j_0)$. Note here that $i_0,j_0$ may not be distinct. One acts this ``sub-generator'' on our functional \eqref{functional-3}. Now, only those terms in \eqref{functional-3} that contain either index $i_0$ or $j_0$ (or both) are affected by the action of this ``sub-generator''. As we will exhaust all choices of $i_0, j_0$, for clarity we prefer to separate the case when $i_0=j_0$ from $i_0\neq j_0$. By the property $C_\ep(0)=I_3$, in the case $i_0=j_0$ we do not see the function $C_\ep$ appear since we have $I_3$. This is how we obtained the first 4 terms on the right-hand side of the explansion \eqref{long-expansion} below. 

Similarly, the jump generator \eqref{gen-2} is also a double sum. By linearity, one can consider the action of each ``sub-generator'' $\frac{1}{2}r_{i_0j_0}^N\left[F(\eta^{-i_0j_0})-F(\eta)\right]$, for each fixed pair of indices $(i_0,j_0)$ (here they are necessarily distinct). As one acts this ``sub-generator'' on our functional \eqref{functional-3}, only those terms that contain either index $i_0$ or $j_0$ (or both) are affected. Here, the effect of the action is that some terms of \eqref{functional-3} are removed.  A moment's thought convinces oneself that triple sums may appear, as in the 6th term on the right-hand side of \eqref{long-expansion}.
\end{remark}
We get
\begin{align}\label{long-expansion}
&F_3(\eta_T^N)-F_3(\eta_0^N)\nonumber\\
&=\frac{1}{N^2}\int_0^T\sum_{i\in\cN(\zeta_t^N)}\sum_{j\in\cN(\zeta_t^N): j\neq i}\frac{\sigma^2}{2}\Big[\nabla_{x_i}\cdot\big(C_\ep(\xi(x_i^N(t)-x_j^N(t)))\nabla_{x_j}\big( v^{\ep,z}(x_i^N(t)-x_j^N(t))\phi(x_i^N(t))\psi(x_j^N(t))\big)\big)\nonumber\\
&\quad\quad\quad \quad +\nabla_{x_i}\cdot\big(C_\ep(\xi(x_i^N(t)-x_j^N(t)))\nabla_{x_j}\big( v^{\ep,z}(x_j^N(t)-x_i^N(t))\phi(x^N_j(t))\psi(x_i^N(t))\big)\big)\Big]dt\nonumber\\
&+\frac{1}{N^2}\int_0^T\sum_{i\in\cN(\zeta_t^N)}\sum_{j\in\cN(\zeta_t^N): j\neq i}\frac{\sigma_0^2+\sigma^2}{2}\Big[\Delta  v^{\ep,z}(x_i^N(t)-x_j^N(t))\phi(x_i^N(t))\psi(x_j^N(t)) \nonumber\\
&\quad\quad\quad \quad +\Delta  v^{\ep,z}(x_j^N(t)-x_i^N(t))\phi(x^N_j(t))\psi(x_i^N(t))\Big]dt\nonumber\\
&+\frac{1}{N^2}\int_0^T\sum_{i\in\cN(\zeta_t^N)}\sum_{j\in\cN(\zeta_t^N): j\neq i}\frac{\sigma_0^2+\sigma^2}{2}\Big[v^{\ep,z}(x_i^N(t)-x_j^N(t))\Delta \phi(x_i^N(t))\psi(x_j^N(t))\nonumber\\ 
&\quad\quad\quad\quad +v^{\ep,z}(x_j^N(t)-x_i^N(t))\phi(x_j^N(t))\Delta \psi(x_i^N(t))\Big]dt\nonumber\\
&+\frac{1}{N^2}\int_0^T\sum_{i\in\cN(\zeta_t^N)}\sum_{j\in\cN(\zeta_t^N): j\neq i}\left(\sigma_0^2+\sigma^2\right)\Big[\nabla v^{\ep,z}(x_i^N(t)-x_j^N(t))\cdot\nabla \phi(x_i^N(t))\psi(x_j^N(t))\nonumber\\
&\quad\quad\quad\quad-\nabla v^{\ep,z}(x_j^N(t)-x_i^N(t))\cdot\phi(x_j^N(t))\nabla \psi(x_i^N(t))\Big]dt\nonumber\\
&-\int_0^T\sum_{i\in\cN(\zeta_t^N)}\sum_{j\in\cN(\zeta_t^N): j\neq i}\frac{R_0}{2N}\theta^\ep(x_i^N(t)-x_j^N(t))\nonumber\\
&\quad\quad\quad \cdot  \frac{1}{N^2}\Big[v^{\ep,z}(x_i^N(t)-x_j^N(t))\phi(x_i^N(t))\psi(x_j^N(t))+v^{\ep,z}(x_j^N(t)-x_i^N(t))\phi(x^N_j(t))\psi(x_i^N(t))\Big]dt\nonumber\\
&-\int_0^T\sum_{i\in\cN(\zeta_t^N)}\sum_{j\in\cN(\zeta_t^N): j\neq i}\frac{R_0}{2N}\theta^\ep(x_i^N(t)-x_j^N(t))\nonumber\\
&\quad\quad\quad \cdot \frac{1}{N^2}\sum_{k\in\cN(\zeta_t^N):k\neq i,j}\Big[v^{\ep,z}(x_i^N(t)-x_k^N(t))\phi(x_i^N(t))\psi(x_k^N(t))+v^{\ep,z}(x_k^N(t)-x_i^N(t))\phi(x_k^N(t))\psi(x_i^N(t))\nonumber\\
&\quad\quad\quad\quad   +v^{\ep,z}(x_j^N(t)-x_k^N(t))\phi(x_j^N(t))\psi(x_k^N(t))+v^{\ep,z}(x_k^N(t)-x_j^N(t))\phi(x_k^N(t))\psi(x_j^N(t))\Big]dt\nonumber\\
&+\wt M_D^N(T)+\wt M_J^N(T),\nonumber\\
\end{align}
where $\wt M_D^N(t)$ and $\wt M_J^N(t)$ are two martingales associated with diffusion and jumps respectively. 

With a tedious but straightforward computation (details omitted), we can expand and rewrite the 1st term on the right-hand side of \eqref{long-expansion} further:
\begin{align*}
&\frac{1}{N^2}\int_0^T\sum_{i\in\cN(\zeta_t^N)}\sum_{j\in\cN(\zeta_t^N): j\neq i}\frac{\sigma^2}{2}\big[\nabla_{x_i}\cdot\big(C_\ep(\xi(x_i^N(t)-x_j^N(t)))\nabla_{x_j}\big( v^{\ep,z}(x_i^N(t)-x_j^N(t))\phi(x_i^N(t))\psi(x_j^N(t))\big)\big)\nonumber\\
&\quad\quad\quad \quad +\nabla_{x_i}\cdot\big(C_\ep(\xi(x_i^N(t)-x_j^N(t)))\nabla_{x_j}\big( v^{\ep,z}(x_j^N(t)-x_i^N(t))\phi(x^N_j(t))\psi(x_i^N(t))\big)\big)\big]dt\\
&=\frac{1}{N^2}\int_0^T\sum_{i\in\cN(\zeta_t^N)}\sum_{j\in\cN(\zeta_t^N): j\neq i}\sigma^2\big[-\nabla_{x_i}\cdot\big(C _\ep(\xi(x_i^N(t)-x_j^N(t)))\nabla v^{\ep,z}(x_i^N(t)-x_j^N(t))\big)\phi(x_i^N(t))\psi(x_j^N(t))\\
&\quad\quad\quad\quad- \nabla \phi(x_i^N(t))^T C_\ep(\xi(x_i^N(t)-x_j^N(t)))\nabla v^{\ep,z}(x_i^N(t)-x_j^N(t))\psi(x_j^N(t))\\
%&\quad\quad\quad\quad +\frac{1}{2}\sum_{\alpha,\beta=1}^3\partial_{\alpha}C^{\alpha\beta}_\ep(\xi(x_i^N(t)-x_j^N(t)))\partial_\beta \psi(x_j^N(t))v^{\ep,z}(x_i^N(t)-x_j^N(t))\phi(x_i^N(t))\\
%&\quad\quad\quad\quad +\frac{1}{2}\sum_{\alpha,\beta=1}^3\partial_{\alpha}C^{\alpha\beta}_\ep(\xi(x_i^N(t)-x_j^N(t)))\partial_\beta \phi(x_j^N(t))v^{\ep,z}(x_j^N(t)-x_i^N(t))\psi(x_i^N(t))\\
&\quad\quad\quad\quad+\nabla v^{\ep,z}(x_i^N(t)-x_j^N(t))^T C_\ep(\xi(x_i^N(t)-x_j^N(t)))\nabla \psi(x_j^N(t))\phi(x_i^N(t))\\
&\quad\quad\quad\quad+ \nabla \phi(x_i^N(t))^T C_\ep(\xi(x_i^N(t)-x_j^N(t)))\nabla\psi(x_j^N(t))v^{\ep,z}(x_i^N(t)-x_j^N(t))\big]dt,
\end{align*}
in which we have used $C_\ep(-x)=C_\ep(x)^T$ and \eqref{equiv-3iv}.

By the symmetry of $\theta$ and the preceding computation, we get the following simplified expansion from the action of generator to $F_3(\zeta_t^N)$:
\begin{align}
&F_3(\eta_T^N)-F_3(\eta_0^N) \label{ini}\\
&=-\frac{\sigma^2}{N^2}\int_0^T\sum_{i\in\cN(\zeta_t^N)}\sum_{j\in\cN(\zeta_t^N): j\neq i}\nabla_{x_i}\cdot\big(C _\ep(\xi(x_i^N(t)-x_j^N(t)))\nabla v^{\ep,z}(x_i^N(t)-x_j^N(t))\big)\phi(x_i^N(t))\psi(x_j^N(t))dt\label{off-1}\\
&-\frac{\sigma^2}{N^2}\int_0^T\sum_{i\in\cN(\zeta_t^N)}\sum_{j\in\cN(\zeta_t^N): j\neq i} \nabla \phi(x_i^N(t))^T C_\ep(\xi(x_i^N(t)-x_j^N(t)))\nabla v^{\ep,z}(x_i^N(t)-x_j^N(t))\psi(x_j^N(t))dt  \label{cross}\\
%&+\frac{\sigma^2}{2N^2}\int_0^T\sum_{i\in\cN(\zeta_t^N)}\sum_{j\in\cN(\zeta_t^N): j\neq i}\sum_{\alpha,\beta=1}^3\partial_{\alpha}C^{\alpha\beta}_\ep(\xi(x_i^N(t)-x_j^N(t)))\partial_\beta \psi(x_j^N(t))v^{\ep,z}(x_i^N(t)-x_j^N(t))\phi(x_i^N(t))dt \label{new-term}\\
%&+\frac{\sigma^2}{2N^2}\int_0^T\sum_{i\in\cN(\zeta_t^N)}\sum_{j\in\cN(\zeta_t^N): j\neq i}\sum_{\alpha,\beta=1}^3\partial_{\alpha}C^{\alpha\beta}_\ep(\xi(x_i^N(t)-x_j^N(t)))\partial_\beta \phi(x_j^N(t))v^{\ep,z}(x_i^N(t)-x_j^N(t))\psi(x_i^N(t))dt \label{new-term-2}\\
&+\frac{\sigma^2}{N^2}\int_0^T\sum_{i\in\cN(\zeta_t^N)}\sum_{j\in\cN(\zeta_t^N): j\neq i}\nabla v^{\ep,z}(x_i^N(t)-x_j^N(t))^T C_\ep(\xi(x_i^N(t)-x_j^N(t)))\nabla \psi(x_j^N(t))\phi(x_i^N(t))dt\label{cross-11}\\
&+\frac{\sigma^2}{N^2}\int_0^T\sum_{i\in\cN(\zeta_t^N)}\sum_{j\in\cN(\zeta_t^N): j\neq i}\nabla \phi(x_i^N(t))^T C_\ep(\xi(x_i^N(t)-x_j^N(t)))\nabla\psi(x_j^N(t))v^{\ep,z}(x_i^N(t)-x_j^N(t))dt\label{cross-12}\\
&+\frac{\sigma_0^2+\sigma^2}{N^2}\int_0^T\sum_{i\in\cN(\zeta_t^N)}\sum_{j\in\cN(\zeta_t^N): j\neq i}\Delta  v^{\ep,z}(x_i^N(t)-x_j^N(t))\phi(x_i^N(t))\psi(x_j^N(t))dt   \label{2nd-or}\\
&+\frac{\sigma_0^2+\sigma^2}{2N^2}\int_0^T\sum_{i\in\cN(\zeta_t^N)}\sum_{j\in\cN(\zeta_t^N): j\neq i}v^{\ep,z}(x_i^N(t)-x_j^N(t))\big[\Delta \phi(x_i^N(t))\psi(x_j^N(t))+\phi(x_j^N(t))\Delta \psi(x_i^N(t))\big]dt  \label{2nd-test}\\
&+\frac{\sigma_0^2+\sigma^2}{N^2}\int_0^T\sum_{i\in\cN(\zeta_t^N)}\sum_{j\in\cN(\zeta_t^N): j\neq i}\nabla v^{\ep,z}(x_i^N(t)-x_j^N(t))\cdot\big[\nabla \phi(x_i^N(t))\psi(x_j^N(t))-\phi(x_j^N(t))\nabla\psi(x_i^N(t))\big]dt  \label{cross-2}\\
&-\frac{R_0}{N^3}\int_0^T\sum_{i\in\cN(\zeta_t^N)}\sum_{j\in\cN(\zeta_t^N): j\neq i}\theta^\ep(x_i^N(t)-x_j^N(t))v^{\ep,z}(x_i^N(t)-x_j^N(t))\phi(x_i^N(t))\psi(x_i^N(t))dt \label{coag-3iag}\\
&-\frac{R_0}{N^3}\int_0^T\sum_{i\in\cN(\zeta_t^N)}\sum_{j\in\cN(\zeta_t^N): j\neq i}\theta^\ep(x_i^N(t)-x_j^N(t))\nonumber\\
&\quad\quad\quad\cdot  \sum_{k\in\cN(\zeta_t^N):k\neq i,j}\big[v^{\ep,z}(x_i^N(t)-x_k^N(t))\phi(x_i^N(t))\psi(x_k^N(t))+v^{\ep,z}(x_k^N(t)-x_i^N(t))\phi(x_k^N(t))\psi(x_i^N(t))\big]dt  \label{coag-off}\\
&+\wt M_D^N(T)+\wt M_J^N(T).\label{big-mg}
\end{align}

We look for all the terms that contain the highest (i.e. second) partial derivatives of $v^{\ep,z}$, which are \eqref{2nd-or} and \eqref{off-1}. In a sense, they are the most dangerous in terms of regularity (not for fixed $\ep$ but as $\ep\to0$). Their sum 
\begin{align*}
\eqref{2nd-or}+\eqref{off-1}&=\frac{1}{N^2}\int_0^T\sum_{i\in\cN(\zeta_t^N)}\sum_{j\in\cN(\zeta_t^N): j\neq i}\Big[(\sigma_0^2+\sigma^2)\Delta  v^{\ep,z}(x_i^N(t)-x_j^N(t))\\
&\quad  -\sigma^2\nabla_{x_i}\big(C _\ep(\xi(x_i^N(t)-x_j^N(t)))\nabla v^{\ep,z}(x_i^N(t)-x_j^N(t)) \big)\Big]\phi(x_i^N(t))\psi(x_j^N(t))dt.
\end{align*}
Recall definiton of $v^{\ep,z}$ in terms of $u^\ep$ \eqref{def-v}, and use the cell equation \eqref{cell} (which we repeat here) 
\begin{align*}
(\sigma_0^2+\sigma^2)\Delta u^\ep (x)&-\sigma^{2}\nabla\cdot\left(C _\ep\left(\xi x\right)  \nabla u^{\ep}\left(  x\right)\right)  =R_{0}%
\theta^\ep\left(  x\right)  \Big(  1+\frac{u^{\ep}\left(  x\right) }{N} \Big)\\
&\text{with }\quad x=x_i^N(t)-x_j^N(t)
\end{align*}
for the terms that involve $u^{\ep}(x_i^N(t)-x_j^N(t))$, while leaving those terms with $u^{\ep}(x_i^N(t)-x_j^N(t)+z)$ as they are, gives
\begin{align*}
\eqref{2nd-or}+\eqref{off-1}&=\frac{1}{N^2}\int_0^T\sum_{i\in\cN(\zeta_t^N)}\sum_{j\in\cN(\zeta_t^N): j\neq i}\Big[(\sigma_0^2+\sigma^2)\Delta  u^{\ep}(x_i^N(t)-x_j^N(t)+z)\\
&\quad  -\sigma^2\nabla_{x_i}\big(C _\ep(\xi(x_i^N(t)-x_j^N(t)))\nabla u^{\ep}(x_i^N(t)-x_j^N(t)+z) \big)\Big]\phi(x_i^N(t))\psi(x_j^N(t))dt\\
&\quad -\frac{R_0}{N^2}\int_0^T\sum_{i\in\cN(\zeta_t^N)}\sum_{j\in\cN(\zeta_t^N): j\neq i}\theta^\ep(x_i^N(t)-x_j^N(t))\Big(1+\frac{1}{N}u^{\ep}(x_i^N(t)-x_j^N(t)) \Big)\phi(x_i^N(t))\psi(x_j^N(t))dt.
\end{align*}
It turns out, a posteriori (but already noted in \cite{hr}), that a term like \eqref{coag-3iag} with the same argument $x_i^N(t)-x_j^N(t)$ in both functions $\theta^\ep$ and $u^\ep$ is not negligible, whereas if one argument is $x_i^N(t)-x_j^N(t)$ and the other is $x_i^N(t)-x_j^N(t)+z$, then such term is negligible (such as \eqref{cov-3iag} and \eqref{diag-2}). Thus, we have to combine the above with the term \eqref{coag-3iag} to ``kill'' a non-negligible term, and finally obtain 
\begin{align}
&\eqref{2nd-or}+\eqref{off-1}+\eqref{coag-3iag}\nonumber\\
&=\frac{1}{N^2}\int_0^T\sum_{i\in\cN(\zeta_t^N)}\sum_{j\in\cN(\zeta_t^N): j\neq i}\Big[(\sigma_0^2+\sigma^2)\Delta  u^{\ep}(x_i^N(t)-x_j^N(t)+z)-R_0\theta^\ep(x_i^N(t)-x_j^N(t))\Big]\phi(x_i^N(t))\psi(x_j^N(t))dt \label{stoss-bis-2}\\
&  \quad\quad-\frac{\sigma^2}{N^2}\int_0^T\sum_{i\in\cN(\zeta_t^N)}\sum_{j\in\cN(\zeta_t^N): j\neq i}\nabla_{x_i}\big(C _\ep\big(\xi(x_i^N(t)-x_j^N(t))\big)\nabla u^{\ep}(x_i^N(t)-x_j^N(t)+z)\big)\phi(x_i^N(t))\psi(x_j^N(t))dt \label{cov-3iag}\\
&  \quad\quad-\frac{R_0}{N^3}\int_0^T\sum_{i\in\cN(\zeta_t^N)}\sum_{j\in\cN(\zeta_t^N): j\neq i}\theta^\ep(x_i^N(t)-x_j^N(t))u^{\ep}(x_i^N(t)-x_j^N(t)+z)\phi(x_i^N(t))\psi(x_j^N(t))dt.  \label{diag-2}
\end{align}
The final product of all these manouvres is that in \eqref{stoss-bis-2}, we have made the pre-limit in \eqref{stoss} appear, and on the other hand, we can expect to show that all the rest of the whole It\^o-Tanaka expansion, namely \eqref{ini}, \eqref{cross}, \eqref{cross-11}, \eqref{cross-12},  \eqref{2nd-test}, \eqref{cross-2}, \eqref{coag-off}, \eqref{cov-3iag}, \eqref{diag-2} as well as \eqref{qv-3iff}, \eqref{qv-3iff-2} and \eqref{qv-jump} below, are all negligible in $L^1(\P)$ in the limit first $N\to\infty$ and then $|z|\to0$. This will be done in Section \ref{sec:neg}. Since the whole expansion \eqref{ini}-\eqref{big-mg}  is an identity, the negligibility of all these ``minor'' terms will prove the desired negligibility of the ``main'' term in \eqref{stoss}.

\medskip
Finally, the martingales \eqref{big-mg} should be bounded in $L^2(\P)$ sense as follows.
\begin{align*}
\wt M_{D}^N(t)&:=\wt M_{D,1}^N(t)+\wt M_{D,2}^N(t)\\
\wt M_{D,1}^N(T)&:=\frac{\sigma}{N^2}\int_0^T\sum_{i\in\cN(\zeta_t^N)}\sum_{j\in\cN(\zeta_t^N): j\neq i}\Big[\nabla v^{\ep,z}(x_i^N(t)-x_j^N(t))\phi(x_i^N(t))\psi(x_j^N(t))\\ 
&\quad\quad+v^{\ep,z}(x_i^N(t)-x_j^N(t))\nabla\phi(x_i^N(t))\psi(x_j^N(t))  -\nabla v^{\ep,z}(x_j^N(t)-x_i^N(t))\phi(x_j^N(t))\psi(x_i^N(t)) \\
&\quad\quad +v^{\ep,z}(x_j^N(t)-x_i^N(t))\phi(x_j^N(t))\nabla \psi(x_i^N(t))  \Big] \cdot\sum_{k\in  \N}\lambda_ke_k^\ep(x_i^N(t)) dW_k(t).\\
\wt M_{D,2}^N(T)&:=\frac{\sigma_0}{N^2}\int_0^T\sum_{i\in\cN(\zeta_t^N)}\sum_{j\in\cN(\zeta_t^N): j\neq i}\Big[\nabla v^{\ep,z}(x_i^N(t)-x_j^N(t))\phi(x_i^N(t))\psi(x_j^N(t))\\ 
&\quad\quad+v^{\ep,z}(x_i^N(t)-x_j^N(t))\nabla\phi(x_i^N(t))\psi(x_j^N(t))  -\nabla v^{\ep,z}(x_j^N(t)-x_i^N(t))\phi(x_j^N(t))\psi(x_i^N(t)) \\
&\quad\quad +v^{\ep,z}(x_j^N(t)-x_i^N(t))\phi(x_j^N(t))\nabla \psi(x_i^N(t))  \Big]\cdot dB_i(t).
\end{align*}
Hence, by It\^o isometry we can bound their quadratic variation by
\begin{align}\label{qv-3iff}
&\E|\wt M_{D,1}^N(T)|^2=\frac{\sigma^2}{N^4}\int_0^T\sum_{k\in  \N}\E\Bigg|\sum_{i\in\cN(\zeta_t^N)}\sum_{j\in\cN(\zeta_t^N): j\neq i}\Big[\nabla v^{\ep,z}(x_i^N(t)-x_j^N(t))\phi(x_i^N(t))\psi(x_j^N(t))\nonumber\\ 
&\quad\quad+v^{\ep,z}(x_i^N(t)-x_j^N(t))\nabla\phi(x_i^N(t))\psi(x_j^N(t))  -\nabla v^{\ep,z}(x_j^N(t)-x_i^N(t))\phi(x_j^N(t))\psi(x_i^N(t))\nonumber \\
&\quad\quad +v^{\ep,z}(x_j^N(t)-x_i^N(t))\phi(x_j^N(t))\nabla \psi(x_i^N(t))  \Big] \cdot \lambda_ke_k^\ep(x_i^N(t))\Bigg|^2dt.\\
&\E|\wt M_{D,2}^N(T)|^2\le\frac{2\sigma_0^2}{N^4}\sum_{i\in\cN(\zeta_t^N)}\int_0^T\Bigg|\sum_{j\in\cN(\zeta_t^N): j\neq i}\Big[\nabla v^{\ep,z}(x_i^N(t)-x_j^N(t))\phi(x_i^N(t))\psi(x_j^N(t))\nonumber\\ 
&\quad\quad+v^{\ep,z}(x_i^N(t)-x_j^N(t))\nabla\phi(x_i^N(t))\psi(x_j^N(t))  -\nabla v^{\ep,z}(x_j^N(t)-x_i^N(t))\phi(x_j^N(t))\psi(x_i^N(t)) \nonumber\\
&\quad\quad +v^{\ep,z}(x_j^N(t)-x_i^N(t))\phi(x_j^N(t))\nabla \psi(x_i^N(t))  \Big]\Bigg|^2 dt.\label{qv-3iff-2}
\end{align}
Further, we can bound the second moment of the jump part of the martingale by the carr\'e-du-champ formula cf. \cite[Proposition 8.7]{DN}
\begin{align}
&\E|\wt M_J^N(T)|^2\le \int_0^T\sum_{i\in\cN(\zeta_t^N)}\sum_{j\in\cN(\zeta_t^N): j\neq i}\frac{R_0}{N}\theta^\ep(x_i^N(t)-x_j^N(t))\Bigg|\frac{1}{N^2}v^{\ep,z}(x_i^N(t)-x_j^N(t))\phi(x_i^N(t))\psi(x_j^N(t))\Bigg|^2dt \nonumber\\
&\quad\quad +\int_0^T\sum_{i\in\cN(\zeta_t^N)}\sum_{j\in\cN(\zeta_t^N): j\neq i}\frac{R_0}{N}\theta^\ep(x_i^N(t)-x_j^N(t))\Bigg|\frac{1}{N^2}\sum_{k\in\cN(\zeta_t^N): k\neq i,j}\Big[v^{\ep,z}(x_i^N(t)-x_k^N(t))\phi(x_i^N(t)) \psi(x_k^N(t))\nonumber\\
&\quad\quad +v^{\ep,z}(x_k^N(t)-x_i^N(t))\phi(x_k^N(t)) \psi(x_i^N(t))+v^{\ep,z}(x_j^N(t)-x_k^N(t))\phi(x_j^N(t))\psi(x_k^N(t))\nonumber\\
&\quad\quad +v^{\ep,z}(x_k^N(t)-x_j^N(t))\phi(x_k^N(t))\psi(x_j^N(t))\Big]\Bigg|^2dt \nonumber\\
&=\frac{R_0}{N^5}\int_0^T\sum_{i\in\cN(\zeta_t^N)}\sum_{j\in\cN(\zeta_t^N): j\neq i}\theta^\ep(x_i^N(t) -x_j^N(t))\big|v^{\ep,z}(x_i^N(t)-x_j^N(t))\phi(x_i^N(t))\psi(x_j^N(t))\big|^2dt \nonumber\\
&\quad\quad + \frac{4R_0}{N^5}\int_0^T\sum_{i\in\cN(\zeta_t^N)}\sum_{j\in\cN(\zeta_t^N): j\neq i}\theta^\ep(x_i^N(t)-x_j^N(t))\nonumber\\
&\quad\quad \cdot\Bigg|\sum_{k\in\cN(\zeta_t^N): k\neq i,j}\Big[v^{\ep,z}(x_i^N(t)-x_k^N(t))\phi(x_i^N(t))\psi(x_k^N(t))+v^{\ep,z}(x_k^N(t)-x_i^N(t))\phi(x_k^N(t))\psi(x_i^N(t))\Big]\Bigg|^2dt.\label{qv-jump}
\end{align}

\section{The negligibility of various terms}\label{sec:neg}
We set out to prove that all the terms \eqref{ini}, \eqref{cross}, \eqref{cross-11},  \eqref{cross-12},   \eqref{2nd-test}, \eqref{cross-2}, \eqref{coag-off}, \eqref{cov-3iag}, \eqref{diag-2}, \eqref{qv-3iff}, \eqref{qv-3iff-2} and \eqref{qv-jump} are negligible in the order of limits $\ep\to0$ then $|z|\to0$ (we refer to it simply as ``negligible'' in the sequel), which by the It\^o-Tanaka identity \eqref{ini}-\eqref{big-mg} of Section \ref{sec:tanaka} yields Proposition \ref{stoss}, and in turn Theorem \ref{thm:main}.

Recall once more the $\ep$-dependent cell equation $u^\ep:\R^3\to\R$
\begin{align}\label{cell-bis}
\sigma_0^2\Delta u^\ep(x)+\sigma^2\nabla\cdot\Big(\omega\big(\frac{\xi x}{\ep}\big)  \nabla u^{\ep}\left(  x\right)\Big)  =R_{0}%
\theta^\ep\left(  x\right)  \Big(  1+\frac{u^{\ep}\left(  x\right) }{N} \Big),
\end{align}
and the unscaled cell equation
\begin{align}\label{cell-unscaled-bis}
\sigma_0^2\Delta u(x)+\sigma^{2}\nabla\cdot\big(\omega(\xi x)\nabla u(x)\big)  =R_{0}%
\theta(x)\big(  1+u(x) \big).
\end{align}
They are related by $u^\ep(x)=\ep^{-1}u(x/\ep)$ for any $x\in\R^3$, hence whenever $u$ exists, $u^\ep$ also exists. 
Crucial for our purpose are pointwise estimates of $u^\ep(x)$ stated in lemmas below. They are proved in \cite[Lemma 3.5]{hr} in the case when the elliptic operator is $\Delta$. To generalize those estimates to our case of uniformly elliptic divergence-form operator with variable coefficients, we will derive all estimates for the $\ep$-independent $u(x)$, and then transfer them to $u^\ep(x)$ by their scaling relation.

Denote the second-order divergence-form operator in the $\ep$-independent equation \eqref{cell-unscaled-bis} by
\begin{align*}
\cA_x&:=\sigma_0^2\Delta+\sigma^2\nabla\cdot\left(\omega(\xi x)\nabla\right)-R_0\theta(x)\\
&=\sigma_0^2\Delta+\sigma^2\text{Tr}\left(\omega(\xi x)\nabla^2 \right)-R_0\theta(x).
\end{align*}
(The equivalence between divergence and non-divergence form is due to \eqref{equiv-3iv}.)
It is uniformly elliptic, since for any $\xi\in\R^3$ and $x\in\R^3$
\begin{align*}
\xi^T\left(\sigma_0^2I_3+\sigma^2 \omega(\xi x)\right)\xi\ge\sigma_0^2|\xi|^2,
\end{align*}
where the matrix $\omega(\cdot)$ is nonnegative definite and $\sigma_0>0$. %Since $a=Id-\overline C\in \mathsf C^{1,\alpha}$ and $C$ has compact support, rewriting $\cA_x$ as a nondivergence-form operator plus lower-order terms, 
%\begin{align*}
%\cA_x=\sigma_0^2\Delta+\sum_{\alpha,\beta=1}^3\Big[\xi \partial_\alpha a^{\alpha,\beta}(\xi x)\partial_\beta+a^{\alpha,\beta}(\xi x)\partial^2_{\alpha\beta}\Big]-R_0\theta(x),
%\end{align*}
%we see that all the coefficients are bounded and $\mathsf C^\alpha$. 

Let us consider the parabolic operator $\partial_t-\cA_x$, and let $p(t; x,y)$ denote its fundamental solution (i.e. heat kernel). By parabolic regularity theory \cite[Theorem 1, page 483]{ilyin}, $p(t; x,y)$ and its first two spatial derivatives satisfy, for a constant $M$ that depends only on $\sigma_0, \sigma, \xi, R_0$ and the $\mathsf C^\alpha$-norm of $a$ and $\theta$ (these data are all given and fixed):
\begin{equation}\label{hke}
\begin{aligned}
p(t; x,y)&\le \frac{M}{t^{3/2}}e^{-\frac{|x-y|^2}{Mt}},\\
|\nabla_x p(t; x,y)|&\le \frac{M}{t^{3/2+1/2}}e^{-\frac{|x-y|^2}{Mt}},\\
|\nabla_x^2p(t; x,y)|&\le \frac{M}{t^{3/2+1}}e^{-\frac{|x-y|^2}{Mt}}.
\end{aligned}
\end{equation}
Since we are in $\R^3$, the Green function $g(x,y)>0$ (i.e. fundamental solution) of $\cA_x$ exists and is related to the heat kernel by
\begin{align*}
g(x,y)=\int_0^\infty p(t;x,y)dt, \quad x, y\in\R^3.
\end{align*}
Thus, by \eqref{hke} we have 
\begin{align}\label{est-green}
g(x,y)\le \int_0^\infty\frac{M}{t^{3/2}}e^{-\frac{|x-y|^2}{Mt}}dt\stackrel{s=\frac{|x-y|^2}{t}}{=}M\int_0^\infty \frac{s^{-1/2}e^{-s/M}}{|x-y|}ds\le M'|x-y|^{-1},
\end{align}
for some finite constant $M'=M'(M)$, since $s^{-1/2}e^{-s/M}$ is integrable at both $0$ and $\infty$. Similarly, 
\begin{align}\label{est-green-grad}
|\nabla_xg(x,y)|&\le \left|\int_0^\infty\nabla_x p(t;x,y)dt\right|\le \int_0^\infty|\nabla_xp(t;x,y)|dt\nonumber\\
&\le \int_0^\infty\frac{M}{t^{3/2+1/2}}e^{-\frac{|x-y|^2}{Mt}}dt=M\int_0^\infty \frac{e^{-s/M}}{|x-y|^2}ds\le M'|x-y|^{-2},\\
|\nabla_x^2g(x,y)|&\le \left|\int_0^\infty\nabla_x^2 p(t;x,y)dt\right|\le \int_0^\infty|\nabla_x^2p(t;x,y)|dt\nonumber\\
&\le \int_0^\infty\frac{M}{t^{3/2+1}}e^{-\frac{|x-y|^2}{Mt}}dt=M\int_0^\infty \frac{s^{1/2}e^{-s/M}}{|x-y|^3}ds\le M'|x-y|^{-3},\label{est-green-hess}
\end{align}
since $e^{-s/M}$ and  $s^{1/2}e^{-s/M}$ are both integrable at both $0$ and $\infty$.

Since \eqref{cell-unscaled-bis} can be written as $\cA_xu(x)=R_0\theta(x)$, one can find a solution given by
\begin{align}\label{exp-sol}
u(x)=-\int_{\R^3}g(x,y)R_0\theta(y)dy, \quad x\in\R^3.
\end{align}
(The negative sign is just a convention, due to the fact that $\cA_x$ is a negative operator while we take $g(x,y)>0$.)
By elliptic regularity theory \cite[Theorem 4.3.1]{krylov}, since $\theta\in \mathsf C^\alpha(\R^3)$ and coefficients of $\cA_x$ are symmetric, $\mathsf C^\alpha$ and bounded, we have $u(x)\in\mathsf C^{2, \alpha}(\R^3)$.

\begin{lemma}\label{lem:unif-bds}
There exists a finite constant $M=M(\sigma_0, \sigma, \xi, R_0, \omega, \theta)$ such that for all $x\in\R^3$  and $\ep\in(0,1)$, we have
\begin{align}
|u^\ep(x)|&\le M\left(|x|\vee\ep\right)^{-1}, \label{unif-bd}\\
|\nabla u^\ep(x)|&\le M\left(|x|\vee\ep\right)^{-2},  \label{unif-grad-bd}
\end{align}
and for all $|x|\ge 2\ep$ we have 
\begin{align}
|\nabla^2 u^\ep(x)|&\le M|x|^{-3}.\label{unif-hess-bd}
\end{align}
\end{lemma}
\begin{proof}
We will first derive the estimate for $u(x)$, then transfer them to $u^\ep(x)$. Since $u$ solves $\cA_xu(x)=R_0\theta(x)$, we can write
\begin{align}\label{u-int-exp}
u(x)=-\int_{\R^3}g(x,y)R_0\theta(y)dy
\end{align}
where $g(x,y)$ is the Green function of $\cA_x$. Recall that $\theta$ is nonnegative and has compact support in $B(0,1)$. If $|x|\ge 2$, then for any $y\in \text{supp}(\theta)$, $|x-y|\ge |x|-|y|\ge |x|-1\ge |x|/2$. Thus, by \eqref{est-green}, $g(x,y)\le M'|x|^{-1}$. Hence, in this case
\begin{align*}
|u(x)|\le \int_{\R^3}g(x,y)R_0\theta(y)dy\le \int_{\R^3}M'|x|^{-1}R_0\theta(y)dy\le M'R_0 |x|^{-1}.
\end{align*}
If on the other hand, $|x|\le 2$, then for any $y\in \text{supp}(\theta)$, $x-y\in B(0,3)$, and hence 
\begin{align*}
|u(x)|\le \int_{\R^3}g(x,y)R_0\theta(y)dy\le R_0\|\theta\|_{L^\infty} \int_{B(0,1)} M' |x-y|^{-1}dy\lesssim \int_{B(0,3)}|y|^{-1}dy\lesssim 1.
\end{align*}
Thus, we conclude the estimate
\begin{align*}
|u(x)|&\le M(|x|\vee 1)^{-1},\\
|u^\ep(x)|&=\ep^{-1}|u(x/\ep)|\le M (|x|\vee \ep)^{-1}.
\end{align*}
Turning to the gradient estimate of $u$, 
\begin{align*}
\nabla u(x) = -\int_{\R^3}\nabla_x g(x,y)R_0\theta(y)dy.
\end{align*}
Hence, if $|x|\ge2$ then for any $y\in \text{supp}(\theta)$, $|x-y|\ge |x|/2$, and by \eqref{est-green-grad},
\begin{align*}
|\nabla u(x)|&\le \int_{\R^3}|\nabla_x g(x,y)|R_0\theta(y)dy\le \int_{\R^3}M'|x|^{-2}R_0\theta(y)dy\le M'R_0 |x|^{-2}.
\end{align*}
If on the other hand, $|x|\le 2$, then for any $y\in \text{supp}(\theta)$, $x-y\in B(0,3)$, and hence 
\begin{align*}
|\nabla u(x)|\le \int_{\R^3}|\nabla_xg(x,y)|R_0\theta(y)dy\le R_0\|\theta\|_{L^\infty} \int_{B(0,1)} M' |x-y|^{-2}dy\lesssim \int_{B(0,3)}|y|^{-2}dy\lesssim 1.
\end{align*}
Thus, we conclude the estimate
\begin{align*}
|\nabla u(x)|&\le M(|x|\vee 1)^{-2},\\
|\nabla u^\ep(x)|&=\ep^{-2}|u(x/\ep)|\le M (|x|\vee \ep)^{-2}.
\end{align*}
Turning to the estimate on the Hessian of $u$, 
\begin{align*}
\nabla^2 u(x) = -\int_{\R^3}\nabla_x^2 g(x,y)R_0\theta(y)dy.
\end{align*}
We only consider the case that $|x|\ge2$, hence for any $y\in \text{supp}(\theta)$, $|x-y|\ge |x|/2$. By \eqref{est-green-hess}, we thus have that 
\begin{align*}
|\nabla^2 u(x)|&\le \int_{\R^3}|\nabla_x^2 g(x,y)|R_0\theta(y)dy\le \int_{\R^3}M'|x|^{-3}R_0\theta(y)dy\le M'R_0 |x|^{-3}.
\end{align*}
By scaling relation, this yields that for any $x$ such that $|x|\ge 2\ep$,
\begin{align*}
|\nabla^2 u^\ep(x)|&=\ep^{-3}|u(x/\ep)|\le M|x|^{-3}.
\end{align*}
\end{proof}

\begin{lemma}
There exists a finite constant $M=M(\sigma_0, \sigma, \xi, R_0, \omega, \theta)$ such that for all $x\in\R^3$ with $|x|\ge 2|z|+2\ep$, $\ep\in(0,1)$, we have 
\begin{align}
|u^\ep(x+z)-u^\ep(x)|\le M|z||x|^{-2}, \label{diff-est}\\
|\nabla u^\ep(x+z)-\nabla u^\ep(x)|\le M|z||x|^{-3}. \label{diff-est-grad}
\end{align}
\end{lemma}
\begin{proof}
We will first derive the estimate for $u(x)$, then transfer them to $u^\ep(x)$.  We consider only  the regime $|x|\ge 2|z|+2$. By \eqref{u-int-exp},
\begin{align*}
u(x+z)-u(x)=-\int_{\R^3}\left(g(x+z,y)-g(x,y)\right)R_0\theta(y)dy
\end{align*}
Since 
\begin{align*}
g(x+z,y)-g(x,y)=\int_0^1\partial_sg(x+sz,y)ds=\int_0^1z\cdot\nabla_xg(x+sz,y)ds,
\end{align*}
we have
\begin{align*}
|g(x+z,y)-g(x,y)|\le \int_0^1|z\cdot\nabla_xg(x+sz,y)|ds\le |z|\sup_{s\in[0,1]}|\nabla_xg(x+sz,y)|.
\end{align*}
Since $|x|\ge 2|z|+2$, we have for any $y\in\text{supp}(\theta)$ and $s\in[0,1]$, 
\[
|x+sz-y|\ge|x|-|sz|-|y|\ge |x|/2 +|x|/2-|z|-1\ge |x|/2.
\]
Hence, by \eqref{est-green-grad}, for some finite constant $M$, 
\[
|g(x+z,y)-g(x,y)|\le M|z||x|^{-2}.
\]
Hence, 
\begin{align*}
|u(x+z)-u(x)|\le \int_{\R^3}\left|g(x+z,y)-g(x,y)\right|R_0\theta(y)dy\le R_0M|z||x|^{-2}.
\end{align*}
By scaling relations, this implies that for any $|x|\ge 2|z|+2\ep$, 
\begin{align*}
|u^\ep(x+z)-u^\ep(x)|=\ep^{-1}|u(x/\ep+z/\ep)-u(x/\ep)|\lesssim \ep^{-1}|z/\ep||x/\ep|^{-2}=|z||x|^{-2},
\end{align*}
where $|x/\ep|\ge 2|z/\ep|+2$ hence the above bound applies.

Turning to the estimates for the gradient, we have 
\begin{align*}
\nabla u(x+z)-\nabla u(x)=-\int_{\R^3}\left(\nabla_xg(x+z,y)-\nabla_xg(x,y)\right)R_0\theta(y)dy.
\end{align*}
Since 
\begin{align*}
\nabla_xg(x+z,y)-\nabla_xg(x,y)=\int_0^1\partial_s\nabla_xg(x+sz,y)ds=\int_0^1z\cdot\nabla_x^2g(x+sz,y)ds,
\end{align*}
we have 
\begin{align*}
|\nabla_xg(x+z,y)-\nabla_xg(x,y)|\le \int_0^1|z\cdot\nabla_x^2g(x+sz,y)|ds\le |z|\sup_{s\in[0,1]}|\nabla_x^2g(x+sz,y)|.
\end{align*}
For $|x|\ge 2|z|+2$, any $y\in\text{supp}(\theta)$ and $s\in[0,1]$, we have $|x+sz-y|\ge |x|/2$, thus by \eqref{est-green-hess} for some finite constant $M$, we have 
\[
|\nabla_xg(x+z,y)-\nabla_xg(x,y)|\le M|z||x|^{-3}.
\]
Hence, 
\begin{align*}
|\nabla u(x+z)-\nabla u(x)|\le \int_{\R^3}\left|\nabla_xg(x+z,y)-\nabla_xg(x,y)\right|R_0\theta(y)dy\le R_0M|z||x|^{-2}\le R_0M|z||x|^{-3}.
\end{align*}
By scaling, this implies that for any $|x|\ge 2|z|+2\ep$, 
\begin{align*}
|\nabla u^\ep(x+z)-\nabla u^\ep(x)|=\ep^{-2}|\nabla u(x/\ep+z/\ep)-\nabla u(x/\ep)|\lesssim \ep^{-2}|z/\ep||x/\ep|^{-3}=|z||x|^{-3}.
\end{align*}
\end{proof}

\begin{lemma}\label{lem:sol-bounds}
The $\mathsf C^{2, \alpha}$-solution \eqref{exp-sol} of \eqref{cell-unscaled-bis} satisfies $-1\le u(x)\le 0$ for all $x\in\R^3$. Hence $-N\le u^\ep(x)\le 0$.
\end{lemma}

\begin{proof}
The argument is similar to \cite[page 64, Step 6]{hr}. Denote here 
\[
\wt \cA_x:=\sigma_0^2\Delta+\sigma^2\nabla\cdot\left(\omega(\xi x)\nabla\right)
\]
and \eqref{cell-unscaled-bis} can be written as $\wt\cA_xu(x)=\theta(x)(1+u(x))$. 
For each $\delta\in(0,1)$, let $\chi_\delta:\R\to[0,1]$ be smooth, identically $1$ on $(-\infty, -\delta)$, identically $0$ on $(0,\infty)$, and decreasing on $[-\delta,0]$, and let
\begin{align*}
\phi_\delta(r) := (1+r)\chi_\delta(1+r),
\end{align*}
which is smooth and non-3ecreasing. Further assume that as $\delta\to0$, $\chi_\delta \to 1_{(-\infty, 0)}$ in $\mathsf C^2$-norm. Clearly, we have the property that $|\phi_\delta(r)|\le 2r$ for $r\ge1$ and $\phi'_\delta(r)\ge0$.

Integration by parts in the ball $B(0,R)$, $R\ge1$, we have 
\begin{align*}
\int_{B(0,R)}\phi_\delta(u(x))\wt\cA_x u(x)du =& -\int_{B_R}\phi'_\delta(u(x))\nabla u(x)^T\left(\sigma_0^2 I_3+\sigma^2\omega(\xi x)\right)\nabla u(x)dx\\&+ \int_{\partial B(0,R)}\phi_\delta(u(x))\left(\sigma_0^2 I_3+\sigma^2\omega(\xi x)\right)\nabla u(x)\cdot d\vec n(x),
\end{align*}
where $\vec n(x)$ denotes the unit outward normal vector. 
By \eqref{unif-bd} and \eqref{unif-grad-bd},
\begin{align*}
\left| \int_{\partial B(0,R)}\phi_\delta(u(x))\left(\sigma_0^2 I_3+\sigma^2\omega(\xi x)\right)\nabla u(x)\cdot d\vec n(x)\right|\lesssim R^{-1}R^{-2}|\partial B(0,R)|=O(R^{-1}).
\end{align*}
Hence taking $R\to\infty$,  we obtain that 
\begin{align}\label{quadratic-form}
\int_{\R^3}\phi_\delta(u(x))\wt\cA_x u(x)du =& -\int_{\R^3}\phi'_\delta(u(x))\nabla u(x)^T\left(\sigma_0^2 I_3+\sigma^2\omega(\xi x)\right)\nabla u(x)dx\le0 .
\end{align}
On the other hand,
\begin{align*}
&\int_{\R^3}\phi_\delta(u(x))\wt\cA_x u(x)du =\int_{\R^3}\phi_\delta(u(x))\theta(x)\left(1+u(x)\right)dx=\int_{\R^3}\theta(x)\left(1+u(x)\right)^2\chi_\delta(1+u(x))dx\ge0,
\end{align*}
hence in view of \eqref{quadratic-form} the right-hand side is $0$. Letting $\delta\to0$ in \eqref{quadratic-form} , since $\phi'_\delta\to 1_{(-\infty, -1)}$ by construction, we get 
\begin{align*}
\int_{\R^3}1_{\{u(x)\le-1\}}\nabla u(x)^T\left(\sigma_0^2 I_3+\sigma^2\omega(\xi x)\right)\nabla u(x)dx=0.
\end{align*}
Since the quadratic form in the above expression is positive unless $\nabla u(x)=0$, we conclude that $\nabla u(x)=0$ on the set $\{x\in\R^3: u(x)\le -1\}$. Since $u$ is continuous, we must have $u(x)=-1$ on this set, and so $u(x)\ge -1$ for all $x\in\R^3$. Further, by \eqref{exp-sol} and $g(x,y)>0$, we have  $u(x)\le 0$ for all $x\in\R^3$.
\end{proof}

\medskip

Now we proceed to bound all terms that are expected to be negligible. We will make extensive use of the auxiliary free particle system $\{y_i^\ep\}_{i=1}^\infty$ \eqref{free-sys} introduced in Section \ref{sec:emp}. Recall that under a natural coupling, whenever a true particle $x_i^N$ is active by time $t$, its trajectory on $[0,t]$ must coincide with that of the free particle $y_i^\ep$. 
We start with the term \eqref{diag-2}.
\begin{proposition}
\begin{align}\label{diag-2-bis}
\lim_{|z|\to0}\limsup_{\ep\to0}\E\Bigg|\frac{1}{N^3}\int_0^T\sum_{i\in\cN(\zeta_t^N)}\sum_{j\in\cN(\zeta_t^N): j\neq i}\theta^\ep(x_i^N(t)-x_j^N(t))u^{\ep}(x_i^N(t)-x_j^N(t)+z)\phi(x_i^N(t))\psi(x_j^N(t))dt\Bigg|=0.
\end{align}
\end{proposition}
\begin{proof}
Since the limit $|z|\to0$ is taken only after $\ep\to0$, we can always assume that $|z|\ge 2\ep$. 
The pre-limit on the left-hand side of \eqref{diag-2-bis} can be bounded from above by
\begin{align*}
&\le \frac{1}{N^3}\E\int_0^T\sum_{i\in\cN(\zeta_t^N)}\sum_{j\in\cN(\zeta_t^N): j\neq i}\theta^\ep(x_i^N(t)-x_j^N(t))|u^{\ep}(x_i^N(t)-x_j^N(t)+z)||\phi(x_i^N(t))||\psi(x_j^N(t))|dt.
\end{align*}
Since the integrand and summands are nonnegative, we can further bound it above by
\begin{align*}
&\le \frac{1}{N^3}\E\int_0^T\sum_{i\neq j=1}^N\theta^\ep(y_i^\ep(t)-y_j^\ep(t))|u^{\ep}(y_i^\ep(t)-y_j^\ep(t)+z)||\phi(y^\ep_i(t))||\psi(y^\ep_j(t))|dt.
\end{align*}
Using exchangeability of the free particles, we further have that 
\begin{align*}
&\le \frac{1}{N}\E\int_0^T\theta^\ep(y_1^\ep(t)-y_2^\ep(t))|u^{\ep}(y_1^\ep(t)-y_2^\ep(t)+z)||\phi(y^\ep_1(t))||\psi(y^\ep_2(t))|dt\\
&\le \frac{1}{N}\int_0^T\iint_{(\R^3)^2}\theta^\ep(y_1-y_2)|u^{\ep}(y_1-y_2+z)|f_t^{12,\ep}(y_1,y_2)|\phi(y_1)\psi(y_2)|dy_1dy_2dt\\
&\le C_0T\frac{1}{N}\iint_{(\R^3)^2}\theta^\ep(y_1-y_2)|u^{\ep}(y_1-y_2+z)||\phi(y_1)\psi(y_2)|dy_1dy_2,
\end{align*}
where $f_t^{12,\ep}(y_1,y_2)$ denotes the joint density of $(y_1^\ep(t), y_2^\ep(t))$ which satisfies the  bound \eqref{2-joint}. Since $|y_1-y_2|\le\ep$ in the support of the function $\theta^\ep$, and $|z|\ge 2\ep$, by \eqref{unif-bd} we have 
\begin{align*}
|u^{\ep}(y_1-y_2+z)|
\lesssim |z|^{-1}.
\end{align*}
Further, $\iint_{(\R^3)^2}\theta^\ep(y_1-y_2)|\phi(y_1)\psi(y_2)|dy_1dy_2\lesssim 1$, $N=\ep^{-1}$, hence we have
\begin{align*}
\frac{1}{N}\iint_{(\R^3)^2}\theta^\ep(y_1-y_2)|u^{\ep}(y_1-y_2+z)|dy_1dy_2\lesssim \ep|z|^{-1}.
\end{align*}
This yields \eqref{diag-2-bis}.
\end{proof}

\medskip
We continue with \eqref{cov-3iag}.
\begin{proposition}
\begin{align}
&\lim_{|z|\to0}\limsup_{\ep\to0} \nonumber\\&\quad\E\Bigg|\frac{1}{N^2}\int_0^T\sum_{i\in\cN(\zeta_t^N)}\sum_{j\in\cN(\zeta_t^N): j\neq i}\nabla\cdot\big(C _\ep\big(x_i^N(t)-x_j^N(t)\big)\nabla u^{\ep}(x_i^N(t)-x_j^N(t)+z)\big)\phi(x_i^N(t))\psi(x_j^N(t))dt\Bigg|=0.\label{cov-diag-bis}
\end{align}
\end{proposition}

\begin{proof}
Since the limit $|z|\to0$ is taken only after $\ep\to0$, we can always assume that $|z|\ge 2\ep$. In view of \eqref{equiv-3iv}, the pre-limit on the left-hand side of \eqref{cov-3iag-bis} can be bounded from above as follows 
\begin{align*}
&=\E\Bigg|\frac{1}{N^2}\int_0^T\sum_{i\in\cN(\zeta_t^N)}\sum_{j\in\cN(\zeta_t^N): j\neq i}\sum_{\alpha,\beta=1}^3C _\ep^{\alpha\beta}\big(x_i^N(t)-x_j^N(t)\big)\partial^2_{\alpha\beta} u^{\ep}(x_i^N(t)-x_j^N(t)+z)\phi(x_i^N(t))\psi(x_j^N(t))dt\Bigg|\\
&\le\frac{1}{N^2}\E\int_0^T\sum_{i\in\cN(\zeta_t^N)}\sum_{j\in\cN(\zeta_t^N): j\neq i}\sum_{\alpha,\beta=1}^3|C _\ep^{\alpha\beta}\big(x_i^N(t)-x_j^N(t)\big)||\partial^2_{\alpha \beta}u^{\ep}(x_i^N(t)-x_j^N(t)+z)||\phi(x_i^N(t))\psi(x_j^N(t))|dt\\
&\le \frac{1}{N^2}\E\int_0^T\sum_{i\neq j=1}^N\sum_{\alpha,\beta=1}^3|C _\ep^{\alpha\beta}\big(y_i^\ep(t)-y_j^\ep(t)\big)||\partial^2_{\alpha \beta}u^{\ep}(y_i^\ep(t)-y_j^\ep(t)+z)||\phi(y_i^\ep(t))\psi(y_j^\ep(t))|dt\\
&\le \E\int_0^T\sum_{\alpha,\beta=1}^3|C _\ep^{\alpha\beta}\big(y_1^\ep(t)-y_2^\ep(t)\big)||\partial^2_{\alpha \beta}u^{\ep}(y_1^\ep(t)-y_2^\ep(t)+z)||\phi(y_1^\ep(t))\psi(y_2^\ep(t))|dt\\
&\le C_0T\iint_{(\R^3)^2}\sum_{\alpha,\beta=1}^3|C _\ep^{\alpha\beta}\big(y_1-y_2\big)||\partial^2_{\alpha \beta}u^{\ep}(y_1-y_2+z)||\phi(y_1)\psi(y_2)|dy_1dy_2,
\end{align*}
using the auxiliary free system, its exchangeability, and that $(y^\ep_1(t), y^\ep_2 (t))$ has a uniformly bounded density, see Proposition \ref{ppn:gaussian}. Since we have assumed that $C _\ep(x)=C (x/\ep)$ for some matrix function $C$ with compact support in $B(0,1)\subset\R^3$, we have $|y_1-y_2|\le\ep$ in the support of $C _\ep^{\alpha\beta}$ in the above display. By  \eqref{unif-hess-bd} and since $|z| \ge 2\ep$, we have 
\begin{align*}
|\partial^2_{\alpha \beta}u^{\ep}(y_1-y_2+z)|&\lesssim |z|^{-3}, \quad \alpha,\beta=1,2,3.
\end{align*}
Thus, we have 
\begin{align*}
&\iint_{(\R^3)^2}\sum_{\alpha,\beta=1}^3|\partial^2_{\alpha \beta}u^{\ep}(y_1-y_2+z)||C _\ep^{\alpha\beta}\big(y_1-y_2\big)||\phi(y_1)\psi(y_2)|dy_1dy_2\\ 
&\lesssim |z|^{-3}\iint\sum_{\alpha,\beta=1}^3|C _\ep^{\alpha\beta}(y)||\phi(y+y_2)\psi(y_2)|dydy_2\lesssim |z|^{-3}\int\sum_{\alpha,\beta=1}^3|C _\ep^{\alpha\beta}(y)|dy\\
&=\ep^3|z|^{-3}\int\sum_{\alpha,\beta=1}^3|C ^{\alpha\beta}(y)|dy,
\end{align*}
where the factor $\ep^3$ comes from change of variables.
This yields \eqref{cov-diag-bis} since the function $C $ is smooth and compactly supported.
\end{proof}

\medskip
We continue with \eqref{2nd-test}. There are two terms with similar structure, and it suffices to consider one of them.
\begin{proposition}
\begin{align}
&\lim_{|z|\to0}\limsup_{\ep\to0}\nonumber\\&\quad \E\Bigg|\frac{1}{N^2}\int_0^T\sum_{i\in\cN(\zeta_t^N)}\sum_{j\in\cN(\zeta_t^N): j\neq i}\big[u^{\ep}(x_i^N(t)-x_j^N(t)+z)-u^{\ep}(x_i^N(t)-x_j^N(t))\big]\Delta \phi(x_i^N(t))\psi(x_j^N(t))dt \Bigg|=0.\label{2nd-test-bis}
\end{align}
\end{proposition}

\begin{proof}
The pre-limit on the left-hand side of \eqref{2nd-test-bis} can be bounded from above by
\begin{align*}
&\le \frac{1}{N^2}\E\int_0^T\sum_{i\in\cN(\zeta_t^N)}\sum_{j\in\cN(\zeta_t^N): j\neq i}\big|u^{\ep}(x_i^N(t)-x_j^N(t)+z)-u^{\ep}(x_i^N(t)-x_j^N(t))\big||\Delta \phi(x_i^N(t))\psi(x_j^N(t))|dt\\
&\le \frac{1}{N^2}\E\int_0^T\sum_{i\neq j=1}^N\big|u^{\ep}(y_i^\ep(t)-y_j^\ep(t)+z)-u^{\ep}(y_i^\ep(t)-y_j^\ep(t))\big||\Delta \phi(y_i^\ep(t))\psi(y_j^\ep(t))|dt\\
&\le \E\int_0^T\big|u^{\ep}(y_1^\ep(t)-y_2^\ep(t)+z)-u^{\ep}(y_1^\ep(t)-y_2^\ep(t))\big||\Delta \phi(y_1^\ep(t))\psi(y_2^\ep(t))|dt\\
&\le C_0T\iint_{(\R^3)^2}\big|u^{\ep}(y_1-y_2+z)-u^{\ep}(y_1-y_2)\big||\Delta \phi(y_1)\psi(y_2)|dy_1dy_2\\
&=C_0T\iint_{(\R^3)^2}\big|u^{\ep}(y+z)-u^{\ep}(y)\big||\Delta \phi(y+y_2)\psi(y_2)|dydy_2,
\end{align*}
using the auxiliary free system, its exchangeability, and that $(y^\ep_1(t), y^\ep_2 (t))$ has a uniformly bounded density, see Proposition \ref{ppn:gaussian}. By \eqref{diff-est} and \eqref{unif-bd}, we have 
\begin{align*}
\big|u^{\ep}(y+z)-u^{\ep}(y)\big|\lesssim \begin{cases}
|z||y|^{-2}, \quad \text{if }|y|\ge 2|z|+2\ep,\\[5pt]
(\ep\vee|y+z|)^{-1}+(\ep\vee|y|)^{-1}, \quad\text{for all }y. \end{cases}
\end{align*}
Note that for any $y_2\in\text{supp}(\psi)$, $y+y_2\in\text{supp}(\phi)$, we have $y\in \mathbb K$ for some compact set $\K$ which is determined by the supports of $\phi,\psi$. Hence, bounding $|\Delta \phi(y+y_2)|\le\|\phi\|_{C^2}$, we have
\begin{align}\label{interm-1}
&\int_{\K}\big|u^{\ep}(y+z)-u^{\ep}(y)\big| dy \nonumber\\
&\lesssim \int_{B(0,2|z|+2\ep)}(\ep\vee|y+z|)^{-1}dy+\int_{B(0,2|z|+2\ep)}(\ep\vee|y|)^{-1}dy+\int_{\K\backslash B(0,2|z|+2\ep)}|z||y|^{-2}dy \nonumber\\
&\le 2\int_{B(0,3|z|+2\ep)}(\ep\vee|y|)^{-1}dy+\int_{\K\backslash B(0,2|z|+2\ep)}|z||y|^{-2}dy\nonumber\\
&\lesssim(3|z|+2\ep)^2+|z|.
\end{align}
Then we integrate in $|\psi(y_2)|dy_2$ which is finite. This yields \eqref{2nd-test-bis}.
\end{proof}

\begin{remark}
The same proof as above also yields the negligibility of the initial $F_3(\eta_0^N)$ and terminal $F_3(\eta_T^N)$ terms in \eqref{ini} individually, since each of them also involve the difference of $u^\ep$ and the heat kernel estimates are uniform in $[0,T]$. Thus we omit discussing them.
\end{remark}

We continue with \eqref{cross-2}. There are two terms with similar structure, and it suffices to consider one of them.
\begin{proposition}
\begin{align}
&\lim_{|z|\to0}\limsup_{\ep\to0}\nonumber\\ &\quad \E\Bigg|\frac{1}{N^2}\int_0^T\sum_{i\in\cN(\zeta_t^N)}\sum_{j\in\cN(\zeta_t^N): j\neq i}\big[\nabla u^{\ep}(x_i^N(t)-x_j^N(t)+z)-\nabla u^{\ep}(x_i^N(t)-x_j^N(t))\big]\cdot\nabla \phi(x_i^N(t))\psi(x_j^N(t))dt\Bigg|=0.\label{cross-2-bis}
\end{align}
\end{proposition}
 
\begin{proof}
The pre-limit on the left-hand side of \eqref{cross-2-bis} can be bounded from above by
\begin{align*}
&\le \frac{1}{N^2}\E\int_0^T\sum_{i\in\cN(\zeta_t^N)}\sum_{j\in\cN(\zeta_t^N): j\neq i}\big|\nabla u^{\ep}(x_i^N(t)-x_j^N(t)+z)-\nabla u^{\ep}(x_i^N(t)-x_j^N(t))\big||\nabla \phi(x_i^N(t))\psi(x_j^N(t))|dt\\
&\le \frac{1}{N^2}\E\int_0^T\sum_{i\neq j =1}^N\big|\nabla u^{\ep}(y_i^\ep(t)-y_j^\ep(t)+z)-\nabla u^{\ep}(y_i^\ep(t)-y_j^\ep(t))\big||\nabla \phi(y_i^\ep(t))\psi(y_j^\ep(t))|dt\\
&\le\E\int_0^T\big|\nabla u^{\ep}(y_1^\ep(t)-y_2^\ep(t)+z)-\nabla u^{\ep}(y_1^\ep(t)-y_2^\ep(t))\big||\nabla \phi(y_1^\ep(t))\psi(y_2^\ep(t))|dt\\
&\le C_0T\iint_{(\R^3)^2}\big|\nabla u^{\ep}(y_1-y_2+z)-\nabla u^{\ep}(y_1-y_2)\big||\nabla \phi(y_1)\psi(y_2)|dy_1dy_2\\
&=C_0T\iint_{(\R^3)^2}\big|\nabla u^{\ep}(y+z)-\nabla u^{\ep}(y)\big||\nabla \phi(y+y_2)\psi(y_2)|dy,
\end{align*}
using the auxiliary free system, its exchangeability, and that $(y^\ep_1(t), y^\ep_2 (t))$ has a uniformly bounded density, see Proposition \ref{ppn:gaussian}. By \eqref{diff-est-grad} and \eqref{unif-grad-bd}, we have 
\begin{align*}
\big|\nabla u^{\ep}(y+z)-\nabla u^{\ep}(y)\big|\lesssim \begin{cases}
|z||y|^{-3}, \quad \text{if }|y|\ge 2|z|+2\ep,\\[5pt]
(\ep\vee|y+z|)^{-2}+(\ep\vee|y|)^{-2}, \quad\text{for all }y.\end{cases}
\end{align*}
Hence, for some $r>2|z|+2\ep$ to be chosen and a compact set $\K$ determined by the supports of $\phi,\psi$, we have
\begin{align*}
&\int_{\K}\big|\nabla u^{\ep}(y+z)-\nabla u^{\ep}(y)\big||\nabla \phi(y+y_2)|dy\\
&\lesssim \int_{B(0,r)}(\ep\vee|y+z|)^{-2}dy+\int_{B(0,r)}(\ep\vee|y|)^{-2}dy+\int_{\K\backslash B(0,r)}|z||y|^{-3}dy\\
&\le 2\int_{B(0,|z|+r)}(\ep\vee|y|)^{-2}dy+\int_{\K\backslash B(0,r)}|z||y|^{-3}dy\\
&\lesssim (|z|+r)+|z|r^{-3}|\K|.
\end{align*}
Choosing $r=4|z|^{\frac{1}{4}}>2|z|+2\ep$ (we can always assume $2\ep\le |z|\le1/2$) yields
\begin{align}\label{triple-int}
\int_{\K}\big|\nabla u^{\ep}(y+z)-\nabla u^{\ep}(y)\big||\nabla \phi(y+y_2)|dy\lesssim |z|^{\frac{1}{4}}.
\end{align}
Then we integrate $|\psi(y_2)|dy_2$ which is finite. This yields \eqref{cross-2-bis}.
\end{proof}

\begin{remark}
The same proof for \eqref{cross-2-bis} can also yield the negligibility of the terms \eqref{cross} and \eqref{cross-11}, since one can simply bound the covariance terms $\max_{\alpha,\beta=1}^3|C^{\alpha,\beta}_\ep(\xi(x_i^N(t)-x_j^N(t))|\le \|C\|_{L^\infty}:=\max_{\alpha,\beta=1}^3 \|C^{\alpha,\beta}\|_{L^\infty(\R^3)}$ and then they are in the form of an averaged double sum involving the difference of $\nabla u^\ep$. It is not necessary to take advantage of the smallness given by the support of $C_\ep$. Similarly, the proof given to prove \eqref{2nd-test-bis} can yield the negligibility of the term \eqref{cross-12}, since it also involves the difference of $u^\ep$, and one can bound $\max_{\alpha,\beta=1}^3|C^{\alpha,\beta}_\ep(\xi(x_i^N(t)-x_j^N(t))|\le \|C\|_{L^\infty}$.
\end{remark}
We continue with \eqref{coag-off}. Although there are two terms there, they can be bounded in the same way, hence it suffices to consider one of them.
\begin{proposition}
\begin{align}\label{coag-off-bis}
\lim_{|z|\to0}\limsup_{\ep\to0}\E\Bigg|\frac{1}{N^3}\int_0^T&\sum_{i,j,k\in\cN(\zeta_t^N): j\neq i, k\neq i,j}\theta^\ep(x_i^N(t)-x_j^N(t))  \nonumber\\
&\cdot\big[u^{\ep}(x_i^N(t)-x_k^N(t)+z)-u^{\ep}(x_i^N(t)-x_k^N(t))\big]\phi(x_i^N(t))\psi(x_k^N(t))dt\Bigg|=0.
\end{align}
\end{proposition}

\begin{proof}
The pre-limit on the left-hand side of \eqref{coag-off-bis} can be bounded from above by
\begin{align*}
& \frac{1}{N^3}\E\int_0^T\sum_{i,j,k\in\cN(\zeta_t^N): j\neq i, k\neq i,j}\theta^\ep(x_i^N(t)-x_j^N(t))\big|u^{\ep}(x_i^N(t)-x_k^N(t)+z)-u^{\ep}(x_i^N(t)-x_k^N(t))\big||\phi(x_i^N(t))\psi(x_k^N(t))|dt\\
&\le\frac{1}{N^3}\E\int_0^T\sum_{i,j,k=1}^N1_{\{j\neq i, k\neq i,j\}}\theta^\ep(y_i^\ep(t)-y_j^\ep(t))\big|u^{\ep}(y_i^\ep(t)-y_k^\ep(t)+z)-u^{\ep}(y_i^\ep(t)-y_k^\ep(t))\big||\phi(y_i^\ep(t))\psi(y_k^\ep(t))|dt\\
&\le \E\int_0^T\theta^\ep(y_1^N(t)-y_3^N(t))\big|u^{\ep}(y_1^\ep(t)-y_2^\ep(t)+z)-u^{\ep}(y_1^\ep(t)-y_2^\ep(t))\big||\phi(y_1^\ep(t))\psi(y_2^\ep(t))|dt,
\end{align*}
using the auxiliary free system and its exchangeability. Since the triple $(y_1^\ep(t), y_2^\ep(t), y_3^\ep(t))$ has a joint density $f_t^{123, \ep}(y_1,y_2,y_3)$ that satisfies the bound
\eqref{2-joint},
we can further bound the previous display by
\begin{align*}
&=\int_0^T\iiint_{(\R^3)^3}\theta^\ep(y_1-y_3)\big|u^{\ep}(y_1-y_2+z)-u^{\ep}(y_1-y_2)\big||\phi(y_1)\psi(y_2)|f_t^{123,\ep}(y_1,y_2,y_3)dy_1dy_2dy_3dt\\
&\le C_0T\iiint_{(\R^3)^3}\theta^\ep(y_1-y_3)\big|u^{\ep}(y_1-y_2+z)-u^{\ep}(y_1-y_2)\big||\phi(y_1)\psi(y_2)|dy_1dy_2dy_3\\
&=C_0T\iint_{(\R^3)^2}\big|u^{\ep}(y_1-y_2+z)-u^{\ep}(y_1-y_2)\big||\phi(y_1)\psi(y_2)|dy_1dy_2\\
&=C_0T\iint_{(\R^3)^2}\big|u^{\ep}(y+z)-u^{\ep}(y)\big||\phi(y+y_2)\psi(y_2)|dy,
\end{align*}
where $\int\theta^\ep(y_1-y_3)dy_3=1$ is used. The last expression is bounded by $\lesssim \ep^2+(3|z|+2\ep)^2+|z|$ as already analyzed in \eqref{interm-1}. This yields \eqref{coag-off-bis}.
\end{proof}

\medskip
We continue with \eqref{qv-3iff}. By the elementary inequality $(a+b+c)^2\le 3(a^2+b^2+c^2)$, it can be seen that it suffices to control two type of terms, which are treated in the next two propositions. Recall that $v^{\ep,z}$ is a shorthand \eqref{def-v}.
\begin{proposition}\label{ppn:cov-sum}
\begin{align}
\lim_{|z|\to0}\limsup_{\ep\to0}\frac{1}{N^4}\int_0^T\sum_{k\in  \N}\E\Bigg|\sum_{i\in\cN(\zeta_t^N)}\sum_{j\in\cN(\zeta_t^N): j\neq i}v^{\ep,z}(x_i^N(t)-x_j^N(t))\nabla\phi(x_i^N(t))\cdot \lambda_ke_k^\ep(x_i^N(t))\psi(x_j^N(t))\Bigg|^2dt=0.\label{qv-3iff-bis}
\end{align}
\end{proposition}
\begin{proof}
The pre-limit on the left-hand side of \eqref{qv-3iff-bis} can be expanded as
\begin{align*}
&\frac{1}{N^4}\int_0^T\sum_{k\in  \N}\E\Bigg|\sum_{i\in\cN(\zeta_t^N)}\sum_{j\in\cN(\zeta_t^N): j\neq i}v^{\ep,z}(x_i^N(t)-x_j^N(t))\nabla\phi(x_i^N(t))\cdot \lambda_ke_k^\ep(x_i^N(t))\psi(x_j^N(t))\Bigg|^2dt\\
&=\frac{1}{N^4}\int_0^T\sum_{k\in  \N}\E\sum_{i\neq m} \sum_{j\neq n}v^{\ep,z}(x_i^N(t)-x_m^N(t))\psi(x_m^N(t))\nabla\phi(x_i^N(t))^T \lambda_ke_k^\ep(x_i^N(t))\otimes \lambda_ke_k^\ep(x_j^N(t))\\
&\quad\quad\quad\quad\quad\quad\quad\quad \quad\quad \cdot\nabla\phi(x_j^N(t))\psi(x_n^N(t)) v^{\ep,z}(x_j^N(t)-x_n^N(t))dt\\
&=\frac{1}{N^4}\int_0^T\E\sum_{i\neq m} \sum_{j\neq n}v^{\ep,z}(x_i^N(t)-x_m^N(t))\psi(x_m^N(t))\nabla\phi(x_i^N(t))^TC_\ep(x_i^N(t)-x_j^N(t))\\     
&\quad\quad\quad\quad\quad\quad\quad\quad \quad\quad \cdot\nabla\phi(x_j^N(t)) \psi(x_n^N(t))v^{\ep,z}(x_j^N(t)-x_n^N(t))dt\\
&=\frac{1}{N^4}\int_0^T\E\sum_{i=j\neq m,n}v^{\ep,z}(x_i^N(t)-x_m^N(t))\psi(x_m^N(t))\nabla\phi(x_i^N(t))^T\nabla\phi(x_i^N(t)) \psi(x_n^N(t))v^{\ep,z}(x_i^N(t)-x_n^N(t))dt\\
&\quad +\frac{1}{N^4}\int_0^T\E\sum_{i\neq j, i\neq m, j\neq n}v^{\ep,z}(x_i^N(t)-x_m^N(t))\psi(x_m^N(t))\nabla\phi(x_i^N(t))^TC_\ep(x_i^N(t)-x_j^N(t))\\ 
&\quad\quad\quad\quad\quad\quad\quad\quad \quad\quad \cdot\nabla\phi(x_j^N(t)) \psi(x_n^N(t))v^{\ep,z}(x_j^N(t)-x_n^N(t))dt\\
&\le \frac{1}{N^4}\int_0^T\E\sum_{i=j\neq m,n}|v^{\ep,z}(x_i^N(t)-x_m^N(t))|| v^{\ep,z}(x_i^N(t)-x_n^N(t))|\phi(x_i^N(t))^2|\psi(x_m^N(t))\psi(x_n^N(t))|dt\\
&\quad +\frac{9}{N^4}\int_0^T\E\sum_{i\neq j, i\neq m, j\neq n}|v^{\ep,z}(x_i^N(t)-x_m^N(t))|\big\|C_\ep\big(x_i^N(t)-x_j^N(t)\big)\big\|_\infty\\ 
&\quad\quad\quad\quad\quad\quad\quad\quad\quad\quad  \cdot|v^{\ep,z}(x_j^N(t)-x_n^N(t))||\phi(x_i^N(t))\phi(x_j^N(t))\psi(x_m^N(t))\psi(x_n^N(t))|dt,
\end{align*}
where recall \eqref{matrix-max} notation for maximum of matrix entries. 
By the bound \eqref{unif-bd} and the coupling with the auxiliary free system, we further bound the above by
\begin{align*}
&\lesssim \frac{\ep^{-1}}{N^4}\int_0^T\E\sum_{i\neq m,n}|v^{\ep,z}(x_i^N(t)-x_m^N(t))|\phi(x_i^N(t))^2|\psi(x_m^N(t))\psi(x_n^N(t))|dt\\
&\quad+\frac{\ep^{-1}}{N^4}\int_0^T\E\sum_{i\neq j, i\neq m, j\neq n}|v^{\ep,z}(x_i^N(t)-x_m^N(t))|\big\|C_\ep\big(x_i^N(t)-x_j^N(t)\big)\big\|_\infty |\phi(x_i^N(t))\phi(x_j^N(t))\psi(x_m^N(t))\psi(x_n^N(t))| dt\\
&\le \frac{1}{N^3}\int_0^T\E\sum_{i\neq m,n}|v^{\ep,z}(y_i^\ep(t)-y_m^\ep(t))|\phi(y_i^\ep(t))^2|\psi(y_m^\ep(t))\psi(y_n^\ep(t))|dt\\
&\quad+\frac{1}{N^3}\int_0^T\E\sum_{i\neq j, i\neq m, j\neq n}|v^{\ep,z}(y_i^\ep(t)-y_m^\ep(t))|\big\|C_\ep\big(y_i^\ep(t)-y_j^\ep(t)\big)\big\|_\infty |\phi(y_i^\ep(t))\phi(y_j^\ep(t))\psi(y_m^\ep(t))\psi(y_n^\ep(t))| dt.
\end{align*}
By exchangeability of the free system, we can further write the above as
\begin{align*}
&\le \int_0^T\E|v^{\ep,z}(y_1^\ep(t)-y_2^\ep(t))|\phi(y_1^\ep(t))^2|\psi(y_2^\ep(t))\psi(y_3^\ep(t))|dt\\
&\quad\quad+N\int_0^T\E|v^{\ep,z}(y_1^\ep(t)-y_3^\ep(t))|\big\|C_\ep\big(y_1^\ep(t)-y_2^\ep(t)\big)\big\|_\infty |\phi(y_1^\ep(t))\phi(y_2^\ep(t))\psi(y_3^\ep(t))\psi(y_4^\ep(t))| dt.
\end{align*}
Recall that a $\ell$-tuple $\big(y_1^\ep(t),..., y_\ell^\ep(t)\big)$ of free particles, $\ell=3,4$,  have a uniformly bounded joint density, see \eqref{2-joint}. Hence we may rewrite the above as
\begin{align*}
&=\int_0^T\iint_{(\R^3)^3}|v^{\ep,z}(y_1-y_2)|f_t^{123,\ep}(y_1,y_2,y_3)\phi(y_1)^2|\psi(y_2)\psi(y_3)|dy_1dy_2dt\\
&\quad\quad +N\int_0^T\iiint_{(\R^3)^4}|v^{\ep,z}(y_1-y_3)|\|C_\ep(y_1-y_2)\|_\infty f_t^{1234,\ep}(y_1,y_2,y_3, y_4)\phi(y_1)\phi(y_2)\psi(y_3)\psi(y_4)dy_1dy_2dy_3dy_4dt\\
&\le C_0T\int_{\K}|v^{\ep,z}(y)|dy+C_0T N\iint_{\K^2}|v^{\ep,z}(y_3)| \|C(y_2/\ep)\|_\infty dy_2dy_3\\
&\le C_0T\int_{\K}|v^{\ep,z}(y)|dy+C_0T N\ep^3\int_{\K}|v^{\ep,z}(y)|dy\int_{\R^3} \|C(z)\|_\infty dz,
\end{align*}
where $\K\subset\R^3$ is a compact set determined by the supports of $\phi,\psi$.
Note that $N\ep^3=\ep^2$ and we already analyzed in \eqref{interm-1} that $\int_{\K}|v^{\ep,z}(y)|dy\lesssim \ep^2+(3|z|+2\ep)^2+|z|$. This yields \eqref{qv-3iff-bis}.
\end{proof}

\begin{proposition}\label{ppn:grad}
\begin{align}
\lim_{|z|\to0}\limsup_{\ep\to0}\frac{1}{N^4}\int_0^T\sum_{k\in  \N}\E\Bigg|\sum_{i\in\cN(\zeta_t^N)}\sum_{j\in\cN(\zeta_t^N): j\neq i}\phi(x_i^N(t))\nabla v^{\ep,z}(x_i^N(t)-x_j^N(t))\cdot \lambda_ke_k^\ep(x_i^N(t))\psi(x_j^N(t))\Bigg|^2dt=0.\label{qv-3iff-bis-2}
\end{align}
\end{proposition}

\begin{proof}
As the in the last proposition, the pre-limit on the left-hand side of \eqref{qv-3iff-bis-2} can be expanded as
\begin{align*}
&\frac{1}{N^4}\int_0^T\sum_{k\in  \N}\E\Bigg|\sum_{i\in\cN(\zeta_t^N)}\sum_{j\in\cN(\zeta_t^N): j\neq i}\phi(x_i^N(t))\nabla v^{\ep,z}(x_i^N(t)-x_j^N(t))\cdot \lambda_ke_k^\ep(x_i^N(t))\psi(x_j^N(t))\Bigg|^2dt\\
&=\frac{1}{N^4}\int_0^T\E\sum_{i\neq m} \sum_{j\neq n}\phi(x_i^N(t))\psi(x_m^N(t))\nabla v^{\ep,z}(x_i^N(t)-x_m^N(t))^TC_\ep(x_i^N(t)-x_j^N(t))\\
&\quad\quad\quad\quad\quad\quad\quad\quad \cdot \nabla v^{\ep,z}(x_j^N(t)-x_n^N(t))\phi(x_j^N(t)) \psi(x_n^N(t))dt\\
&\le \frac{1}{N^4}\int_0^T\E\sum_{i=j\neq m,n}|\nabla v^{\ep,z}(x_i^N(t)-x_m^N(t))||\nabla v^{\ep,z}(x_i^N(t)-x_n^N(t))|\phi(x_i^N(t))^2|\psi(x_m^N(t))\psi(x_n^N(t))|dt\\
&+\frac{1}{N^4}\int_0^T\E\sum_{i\neq j, i\neq m, j\neq n}|\nabla v^{\ep,z}(x_i^N(t)-x_m^N(t))|\big\|C_\ep\big(x_i^N(t)-x_j^N(t)\big)\big\|_\infty\\
&\quad\quad\quad\quad\quad\quad\quad\quad \cdot |\nabla v^{\ep,z}(x_j^N(t)-x_n^N(t))||\phi(x_i^N(t))\phi(x_j^N(t))\psi(x_m^N(t))\psi(x_n^N(t))|dt,
\end{align*}
We bound the two terms separately. First we treat the first term and we divide it further according to $m=n$ or $m\neq n$. In case $i=j$, $m=n$ we have 
\begin{align*}
&\frac{1}{N^4}\int_0^T\E\sum_{i=j\neq m=n}|\nabla v^{\ep,z}(x_i^N(t)-x_m^N(t))||\nabla v^{\ep,z}(x_i^N(t)-x_m^N(t))|\phi(x_i^N(t))^2\psi(x_m^N(t))^2dt\\
&\le \frac{1}{N^4}\int_0^T\E\sum_{i=j\neq m=n}|\nabla v^{\ep,z}(y_i^\ep(t)-y_m^\ep(t))||\nabla v^{\ep,z}(y_i^\ep(t)-y_m^\ep(t))|\phi(y_i^\ep(t))^2\psi(y_m^\ep(t))^2dt\\
&\le \frac{1}{N^2}\int_0^T\E|\nabla v^{\ep,z}(y_1^\ep(t)-y_2^\ep(t))||\nabla v^{\ep,z}(y_1^\ep(t)-y_2^\ep(t))|\phi(y_1^\ep(t))^2\psi(y_2^\ep(t))^2dt\\
&\lesssim \frac{\ep^{-1}}{N^2}\int_0^T\E|\nabla v^{\ep,z}(y_1^\ep(t)-y_2^\ep(t))|\phi(y_1^\ep(t))^2\psi(y_2^\ep(t))^2dt,
\end{align*}
and we have analysed in \eqref{triple-int} that such a term (even without the $1/N$ factor) is negligible. In case $i=j$, $m\neq n$, we have
\begin{align*}
&\frac{1}{N^4}\int_0^T\E\sum_{i=j\neq m\neq n}|\nabla v^{\ep,z}(x_i^N(t)-x_m^N(t))||\nabla v^{\ep,z}(x_i^N(t)-x_n^N(t))|\phi(x_i^N(t))^2|\psi(x_m^N(t))\psi(x_n^N(t))|dt\\
&\le\frac{1}{N^4}\int_0^T\E\sum_{i=j\neq m\neq n}|\nabla v^{\ep,z}(y_i^\ep(t)-y_m^\ep(t))||\nabla v^{\ep,z}(y_i^\ep(t)-y_n^\ep(t))|\phi(y_i^\ep(t))^2|\psi(y_m^\ep(t))\psi(y_n^\ep(t))|dt\\
&\le \frac{1}{N}\int_0^T\E|\nabla v^{\ep,z}(y_1^\ep(t)-y_2^\ep(t))||\nabla v^{\ep,z}(y_1^\ep(t)-y_3^\ep(t))|\phi(y_1^\ep(t))^2|\psi(y_2^\ep(t))\psi(y_3^\ep(t))|dt.
\end{align*}
Recall that the joint density of a triple of free particles $\left(y_1^\ep(t), y_2^\ep(t), y_3^\ep(t)\right)$ has the bound \eqref{2-joint}, hence we have (even without the $1/N$ factor)
\begin{align*}
&\int_0^T\E|\nabla v^{\ep,z}(y_1^\ep(t)-y_2^\ep(t))||\nabla v^{\ep,z}(y_1^\ep(t)-y_3^\ep(t))|\phi(y_1^\ep(t))^2|\psi(y_2^\ep(t))\psi(y_3^\ep(t))|dt\\
&=\int_0^T\iiint_{(\R^3)^3}|\nabla v^{\ep,z}(y_1-y_2)||\nabla v^{\ep,z}(y_1-y_3)|\phi(y_1)^2|\psi(y_2)\psi(y_3)|f_t^{123,\ep}(y_1,y_2,y_3)dy_1dy_2dy_3\\
&\le TC_0\iiint_{(\R^3)^3}|\nabla v^{\ep,z}(y_1-y_2)||\nabla v^{\ep,z}(y_1-y_3)\phi(y_1)^2|\psi(y_2)\psi(y_3)|dy_1dy_2dy_3\\
&\lesssim\left(\int_{\K}|\nabla v^{\ep,z}(y)|dy\right)^2,
\end{align*}
where $\K$ is a compact set determined by the supports of $\phi,\psi$.
Since from \eqref{triple-int} we know that $\int_{\K}|\nabla v^{\ep,z}(y)|dy$ is negligible, so is the whole term.

To treat the second main term involving the covariance $C_\ep$, we proceed as in the proof of Proposition \ref{ppn:cov-sum}, using the bound \eqref{unif-grad-bd} and coupling with the auxiliary system
\begin{align*}
&\frac{1}{N^4}\int_0^T\E\sum_{i\neq j, i\neq m, j\neq n}|\nabla v^{\ep,z}(x_i^N(t)-x_m^N(t))|\big\|C_\ep\big(x_i^N(t)-x_j^N(t)\big)\big\|_\infty\\
&\quad\quad\quad\quad  \quad\quad\quad\quad \cdot|\nabla v^{\ep,z}(x_j^N(t)-x_n^N(t))||\phi(x_i^N(t))\phi(x_j^N(t))\psi(x_m^N(t))\psi(x_n^N(t))|dt\\
&\lesssim \frac{\ep^{-2}}{N^4}\int_0^T\E\sum_{i\neq j, i\neq m, j\neq n}|\nabla v^{\ep,z}(x_i^N(t)-x_m^N(t))| \big\|C_\ep\big(x_i^N(t)-x_j^N(t)\big)\big\|_\infty|\phi(x_i^N(t))\phi(x_j^N(t))\psi(x_m^N(t))\psi(x_n^N(t))| dt\\
&\le \ep^{-2}\int_0^T\E|\nabla v^{\ep,z}(y_1^\ep(t)-y_2^\ep(t))|\big\|C_\ep(y_1^\ep(t)-y_3^\ep(t))\big\|_\infty |\phi(x_1^N(t))\phi(x_2^N(t))\psi(x_3^N(t))\psi(x_4^N(t))| dt\\
&\lesssim \ep^{-2}\ep^3\int_{\K}|\nabla v^{\ep,z}(y)|dy\int_{\R^3} \|C(z)\|_\infty dz.
\end{align*}
Notice that we have a factor of $\ep$ and in addition $\int_{\K}|\nabla v^{\ep,z}(y)|dy$ is negligible. This completes the proof of \eqref{qv-3iff-bis-2}.
\end{proof}

\medskip
We continue with \eqref{qv-3iff-2}. By the elementary inequality $(a+b)^2\le 2(a^2+b^2)$, it suffices to control separately two terms in the following proposition.
\begin{proposition}
\begin{align}\label{qv-3iff-2-bis}
&\lim_{|z|\to0}\limsup_{\ep\to0}\frac{1}{N^4}\sum_{i\in\cN(\zeta_t^N)}\int_0^T\Bigg|\sum_{j\in\cN(\zeta_t^N): j\neq i}\nabla v^{\ep,z}(x_i^N(t)-x_j^N(t))\phi(x_i^N(t))\psi(x_j^N(t))\Bigg|^2 dt=0.\\
&\lim_{|z|\to0}\limsup_{\ep\to0}\frac{1}{N^4}\sum_{i\in\cN(\zeta_t^N)}\int_0^T\Bigg|\sum_{j\in\cN(\zeta_t^N): j\neq i}v^{\ep,z}(x_i^N(t)-x_j^N(t))\nabla\phi(x_i^N(t))\psi(x_j^N(t))\Bigg|^2 dt=0.\nonumber
\end{align}
\end{proposition}
\begin{proof}
The proof of the two statements are similar and for brevity we only treat the first one \eqref{qv-3iff-2-bis}. The pre-limit of the left-hand side of \eqref{qv-3iff-2-bis} can be bounded by
\begin{align*}
&\frac{1}{N^4}\sum_{i\in\cN(\zeta_t^N)}\int_0^T\Bigg|\sum_{j\in\cN(\zeta_t^N): j\neq i}\nabla v^{\ep,z}(x_i^N(t)-x_j^N(t))\phi(x_i^N(t))\psi(x_j^N(t))\Bigg|^2 dt\\
&=\frac{1}{N^4}\sum_{i\in\cN(\zeta_t^N)}\int_0^T\sum_{\substack{j, m\in\cN(\zeta_t^N): \\j, m\neq i}}\nabla v^{\ep,z}(x_i^N(t)-x_j^N(t))\cdot \nabla v^{\ep,z}(x_i^N(t)-x_m^N(t))\phi(x_i^N(t))^2\psi(x_j^N(t))\psi(x_m^N(t)) dt\\
&\le \frac{1}{N^4}\sum_{i\in\cN(\zeta_t^N)}\int_0^T\sum_{j, m\in\cN(\zeta_t^N): j=m\neq i}|\nabla v^{\ep,z}(x_i^N(t)-x_j^N(t))||\nabla v^{\ep,z}(x_i^N(t)-x_j^N(t))|\phi(x_i^N(t))^2\psi(x_j^N(t))^2dt\\
&\quad+\frac{1}{N^4}\sum_{i\in\cN(\zeta_t^N)}\int_0^T\sum_{\substack{j, m\in\cN(\zeta_t^N):\\ j\neq m, \, j,m\neq i}}|\nabla v^{\ep,z}(x_i^N(t)-x_j^N(t))||\nabla v^{\ep,z}(x_i^N(t)-x_m^N(t))|\phi(x_i^N(t))^2\psi(x_j^N(t))\psi(x_m^N(t)) dt.
\end{align*}
However, notice that this type of terms are already treated in the proof of Proposition \ref{ppn:grad}.
\end{proof}

\medskip
We continue with \eqref{qv-jump}. There are two terms and we bound them separately in the next two propositions.
\begin{proposition}
\begin{align}
\lim_{|z|\to0}\limsup_{\ep\to0}\E\Bigg|\frac{1}{N^5}\int_0^T\sum_{i\in\cN(\zeta_t^N)}\sum_{j\in\cN(\zeta_t^N): j\neq i}\theta^\ep(x_i^N(t) -x_j^N(t))\big|v^{\ep,z}(x_i^N(t)-x_j^N(t))\phi(x_i^N(t))\psi(x_j^N(t))\big|^2dt\Bigg|=0.\label{qv-jump1-bis}
\end{align}
\end{proposition}
\begin{proof}
By \eqref{unif-bd} and coupling with the auxiliary free system, the pre-limit on the left-hand side of \eqref{qv-jump1-bis} can bounded by
\begin{align*}
&\le (\ep^{-1})^2\frac{1}{N^5}\int_0^T\E\sum_{i\in\cN(\zeta_t^N)}\sum_{j\in\cN(\zeta_t^N): j\neq i}\theta^\ep(x_i^N(t) -x_j^N(t))\phi(x_i^N(t))^2\psi(x_j^N(t))^2dt\\
&\le\frac{1}{N^3}\int_0^T\E\sum_{i\neq j=1}^N\theta^\ep(y_i^\ep(t) -y_j^\ep(t))\phi(y_i^\ep(t))^2\psi(y_j^\ep(t))^2dt\\
&\le \frac{1}{N}\int_0^T\E\; \theta^\ep(y_1^\ep(t) -y_2^\ep(t))\phi(y_1^\ep(t))^2\psi(y_2^\ep(t))^2dt
\\
&= \frac{1}{N}\int_0^T\iint_{(\R^3)^2} \theta^\ep(y_1-y_2)\phi(y_1)^2\psi(y_2)^2f_t^{12,\ep}(y_1,y_2)dy_1dy_2dt=O(N^{-1}),
\end{align*}
where $\int\theta^\ep =1$ is used. This yields \eqref{qv-jump1-bis}.
\end{proof}

\begin{proposition}
\begin{align}
\lim_{|z|\to0}\limsup_{\ep\to0} \frac{1}{N^5}\E\int_0^T\sum_{i\in\cN(\zeta_t^N)}&\sum_{j\in\cN(\zeta_t^N): j\neq i}\theta^\ep(x_i^N(t)-x_j^N(t))\nonumber \\
&\cdot\Bigg|\sum_{k\in\cN(\zeta_t^N): k\neq i,j}v^{\ep,z}(x_i^N(t)-x_k^N(t))\phi(x_i^N(t))\psi(x_k^N(t))\Bigg|^2dt=0.\label{qv-jump2-bis}
\end{align}
\end{proposition}

\begin{proof}
Th prelimit on the left-hand side of \eqref{qv-jump2-bis} can be bounded by
\begin{align*}
&=\frac{1}{N^5}\E\int_0^T\sum_{i\neq j}\sum_{k, \ell\neq i,j}\theta^\ep(x_i^N(t)-x_j^N(t))v^{\ep,z}(x_i^N(t)-x_k^N(t))v^{\ep,z}(x_i^N(t)-x_\ell^N(t))\phi(x_i^N(t))^2\psi(x_k^N(t))\psi(x_\ell^N(t))dt\\
&\lesssim  \ep^{-1}\frac{1}{N^5}\E\int_0^T\sum_{i\neq j; k,\ell \neq i,j}\theta^\ep(x_i^N(t)-x_j^N(t))|v^{\ep,z}(x_i^N(t)-x_k^N(t))|\phi(x_i^N(t))^2|\psi(x_k^N(t))\psi(x_\ell^N(t))|dt,
\end{align*}
where we have used \eqref{unif-bd}. Using the auxiliary free system, its exchangeability, and that $(y^\ep_1(t),...,  y^\ep_4 (t))$ has a uniformly bounded density, see Proposition \ref{ppn:gaussian}, we further bound the above by
\begin{align*}
&\le \frac{1}{N^4}\int_0^T\sum_{i\neq j; k,\ell \neq i,j}\theta^\ep(y_i^\ep(t)-y_j^\ep(t))|v^{\ep,z}(y_i^\ep(t)-y_k^\ep(t))|\phi(y_i^\ep(t))^2|\psi(y_k^\ep(t))\psi(y_\ell^\ep(t))|dt\\
&\le  \int_0^T \theta^\ep(y_1^\ep(t)-y_2^\ep(t))|v^{\ep,z}(y_1^\ep(t)-y_3^\ep(t))|\phi(y_1^\ep(t))^2|\psi(y_3^\ep(t))\psi(y_4^\ep(t))|dt\\
&\le \int_0^T\iiint_{(\R^3)^4}\theta^\ep(y_1-y_2)|v^{\ep,z}(y_1-y_3)\phi(y_1)^2|\psi(y_3)\psi(y_4)|f_t^{1234,\ep}(y_1,y_2,y_3, y_4)dt\\
&\lesssim\int_{\K}|v^{\ep,z}(y_3)|dy_3,
\end{align*}
where $\K$ is a compact set determined by the supports of $\phi,\psi$, and we used that $\int_{\R^3}\theta^\ep(y_1-y_2)dy_2=1$ for any $y_1$. 
We already analyzed in \eqref{interm-1} that $\int_{\K}|v^{\ep,z}(y_3)|dy_3\lesssim \ep^2+(3|z|+2\ep)^2+|z|$. This yields \eqref{qv-jump2-bis}.
\end{proof}

\medskip
This completes the verification that all the terms \eqref{ini}, \eqref{cross}, \eqref{cross-11}, \eqref{cross-12}, \eqref{2nd-test}, \eqref{cross-2}, \eqref{coag-off}, \eqref{cov-3iag}, \eqref{diag-2}, \eqref{qv-3iff}, \eqref{qv-3iff-2} and \eqref{qv-jump} are negligible in the sense that $\lim_{|z|\to0}\limsup_{\ep\to0}\E|(*)|=0$, where $(*)$ stands for each of these terms.

\section{Potential theory and consequence for particle coalescence}\label{sec:potential}
Let us summarize the result of our main theorem by saying that we have
introduced a non-inertial model for particle coalescence, driven by a common
noise, and we have obtained in the scaling limit a \abbr{PDE} for the limit density with
a mean coalescence rate $\overline{\mathcal{R}}$ given by
\[
\overline{\mathcal{R}}=\overline{\mathcal{R}}\left(  \xi, \sigma^2,R_{0}\right)
:=R_{0}\int_{\mathbb{R}^3}\theta\left(  x\right)  \left(  1+u\left(
x\right)  \right)  dx
\]
where $u:\mathbb{R}^3\rightarrow\mathbb{R}$ solves the auxiliary
cell-equation
\[
\sigma_{0}^{2}\Delta u(x)+\sigma^2\nabla\cdot\left(\omega(\xi x)\nabla u(x)\right)
=R_{0}\theta(x)\left(  1+u(x)\right)  .
\]
Let us develop a few potential theory preliminaries and deduce properties of
$\overline{\mathcal{R}}\left(  \xi, \sigma^2,R_{0}\right)  $.

\medskip
Given a uniformly elliptic, bounded and $\mathsf C^{1, \alpha}$ matrix-valued
function $A:\mathbb{R}^3\rightarrow\mathbb{R}^{3\times 3}$, for some $\alpha\in(0,1)$, call
$G_{A}\left(  x,y\right)  $ the associated Green kernel. Given a non-negative integrable $\mathsf C^\alpha$ function $\theta:\mathbb{R}%
^3\rightarrow\mathbb{R}$ with compact support in $B(0,1)$, for every $\gamma>0$ consider
the elliptic problem%
\begin{align}\label{elliptic}
\nabla\cdot\left(  A\left(  x\right)  \nabla u^{\gamma}\left(
x\right)  \right)  =\gamma\theta\left(  x\right)  \left(  1+u^{\gamma}\left(
x\right)  \right)  .
\end{align}
There exists a unique $\mathsf C^2(\R^3)$ solution $u^\gamma$ to \eqref{elliptic}, in the class of those $u$ with $|u(x)|\to0$ as $|x|\to\infty$. The proof is a modification of \cite[Theorem 6.1]{hr}. Here we only repeat the proof of a weaker result, namely uniqueness within the class of those $u\in \mathsf C^2$ such that $|u(x)|=O(|x|^{-1})$ and $|\nabla u(x)|=O(|x|^{-2})$. In \eqref{exp-sol} we  constructed explicitly  a solution $u^\gamma$ to \eqref{elliptic} that satisfies these decay conditions cf. Lemma \ref{lem:unif-bds} (but we use a different Green function there). To prove this claim, consider $u_1,u_2$ two solutions of \eqref{elliptic} in this class and let $\wt u=u_1-u_2$. Then 
\[
\nabla\cdot\left(  A\left(  x\right)  \nabla \wt u\left(
x\right)  \right)  =\gamma\theta\left(  x\right)  \wt u\left(
x\right) .
\]
Multiplying $\wt u$ on both sides and integrating by parts in the ball $B(0,R)$, we get 
\[
-\int_{B(0,R)}\nabla \wt u(x)^TA(x)\nabla \wt u(x)dx +\int_{\partial B(0,R)}u(x)A(x)\nabla \wt u(x) \cdot d\vec n(x)= \int_{B(0,R)}\gamma \theta(x)\wt u(x)^2,
\]
where $\vec n(x)$ denotes the unit outward normal vector.
Since $A(x)$ is component-wise bounded, by the decay assumption, 
\[
\left|\int_{\partial B(0,R)}u(x)A(x)\nabla \wt u(x) \cdot d\vec n(x)\right|\lesssim R^{-1}R^{-2}|\partial B(0,R)|=O(R^{-1}).
\]
By  the positive-definiteness of $A$, taking $R\to\infty$ we conclude that 
\[
\int_{\R^3}\nabla \wt u(x)^TA(x)\nabla \wt u(x)dx=\int_{\R^3} \gamma \theta(x)\wt u(x)^2 =0
\]
This forces $\wt u(x)=u_1(x)-u_2(x)=0$ for all $x\in\R^3$.

We proceed to define capacity for our divergence-form operator, a generalization of classical capacity in the case of $\Delta$. Given a compact set $K\subset\R^3$, consider the problems:%
\begin{align*}
C_{1}\left(  K,A\right) & :=\inf\left\{  \int_{\mathbb{R}^3}\nabla\psi\left(
x\right)  ^{T}A\left(  x\right)  \nabla\psi\left(  x\right)  dx; \; \psi\left(
x\right)  \text{ is smooth, }\psi\geq1\text{ in a neighborhood of }K\right\}\\
C_{2}\left(  K,A\right)  &:=\sup\left\{  \mu\left(  K\right)  ;\mu\text{
measure supported in }K \text{ s.t. }\int_{\R^3}G_{A}\left(  x,y\right)  \mu\left(
dy\right)  \leq1\text{ for all }x\in\R^3\right\}  .
\end{align*}
By \cite[Theorem 5.5.5 (i), Lemma 5.5.2, Definition 5.4.1]{ag} and \cite[Theorem 4.1 and Eq. (6.5)]{lsw}, $C_{1}\left(  K,A\right)  =C_{2}\left(  K,A\right)  $ and thus both can serve as equivalent definitions of capacity associated with operator $A$. Call it $\text{Cap}\left(
K,A\right)  $. 

\begin{lemma}\label{lem:hard-sphere}
If $K_{\theta}\subset\R^3$ is the support of $\theta$, then
\[
\lim_{\gamma\rightarrow\infty}\gamma\int_{\R^3}\theta\left(
x\right)  \left(  1+u^{\gamma}\left(  x\right)  \right)  dx={\rm Cap}\left(
K_{\theta},A\right)  .
\]
\end{lemma}
\begin{proof}
Using the equivalent definition $C_2(K, A)$, the proof is the same as that of \cite[Theorem 6.2]{hr}.
\end{proof}

\medskip

Coming back to the problem of understanding $\overline{\mathcal{R}}\left(
\xi,\sigma^2,R_{0}\right)  $, the previous results imply:

\begin{corollary}
For every $\sigma_{0},\sigma>0$,
\[
\lim_{R_{0}\rightarrow\infty}\overline{\mathcal{R}}\left(  \xi,\sigma^2%
,R_{0}\right)  ={\rm Cap}\left(  K_{\theta},\left[  \sigma_{0}^{2}I_3%
+\sigma^2\omega\left(  \xi\cdot\right)  \right]  \right)  .
\]

\end{corollary}

Call \textquotedblleft hard sphere coalescence\textquotedblright\ rate the
number
\[
\overline{\mathcal{R}}_{\infty}\left(  \xi,\sigma^2\right)  :=\text{Cap}\left(
K_{\theta},\left[  \sigma_{0}^{2}I_3+\sigma^2\omega\left(  \xi\cdot\right)  \right]
\right),
\]
over which we put our main attention.
\begin{remark}
We call it hard sphere coalescence, since as $R_0\to\infty$, the microscopic coalescence rate \eqref{rate} tends to infinity, thus two particles that do interact (i.e. are $\ep$-close) surely annihilate. That is, in this limit, our stochastic interaction rule becomes deterministic. This reminds us of deterministic collision as in Boltzmann's theory of hard spheres. While we do not claim there is any essential connection to that area, see \cite{reza} for a study of particle approximation to Boltzmann equation with stochastic collision rate, where a rate similar to \eqref{rate} appears.
\end{remark}

Call $\text{Cap}\left(  K_{\theta}\right)  $ the classical capacity of
$K_{\theta}$ (i.e. the one associated with $\Delta$). Our main result for the
theory of turbulent coalescence is:

\begin{theorem}%
\[
\lim_{\xi\rightarrow0}\overline{\mathcal{R}}_{\infty}\left(\xi,  \sigma^2\right)
=\sigma_{0}^{2}\rm{Cap}\left(  K_{\theta}\right)
\]%
\[
\lim_{\xi\rightarrow\infty}\overline{\mathcal{R}}_{\infty}\left(\xi,  \sigma^2 \right)  =\left(  \sigma_{0}^{2}+\sigma^2\right)  \rm{Cap}\left(
K_{\theta}\right)  .
\]

\end{theorem}

\begin{proof}
Using the equivalent definition $C_{1}(K,A)$, the claims follow from the
monotonicity of quadratic form in terms of the matrix ordering, since
$\lim_{\xi\rightarrow\infty}\omega(\xi x)=I_3$ and $\lim_{\xi\rightarrow0}\omega(\xi
x)=0$ for every $x$.
\end{proof}

\medskip
The first claim says that the coalescence rate becomes very small as the scale ratio $\xi$ tends to zero, since $\sigma_{0}^{2}$ is assumed very small with respect to  the
other constants. The second claim is expressed as a limit as $\xi
\rightarrow\infty$, in order to have a sharp mathematical result but from the
viewpoint of physical interpretation it must be understood as a property just
for non-small $\xi$. It clearly states that $\overline{\mathcal{R}}_{\infty
}\left(\xi,  \sigma^2\right)  $ increases with $\sigma^2$, which was the other one
of our desired properties.

\bigskip
\noindent
\textbf{Acknowledgments.}
We thank two anonymous referees
who provided very insightful comments which allowed us to reach more precise
and quantitative results and improve the exposition, and clarify also certain
limitations. 

We thank Andrea Papini, Alessandra Lanotte and Jeremie Bec for important
discussions, Francesco Russo and the organizers of the Winter School
``Stochastic and deterministic analysis of irregular models'', Luminy 2024 which
motivated the development of the result presented here. 

The research of the first author is funded by the European Union (ERC, NoisyFluid, No. 101053472). Views and opinions expressed are however those of the authors only and do not necessarily reflect those of the European Union or the European Research
Council. Neither the European Union nor the granting authority can be held
responsible for them. The research of the second author is funded by the Deutsche Forschungsgemeinschaft (DFG, German Research Foundation) under Germany’s Excellence Strategy EXC 2044-390685587, Mathematics M\"unster: Dynamics-Geometry-Structure.

\bigskip\noindent
\textbf{Availability of data and materials.} Not applicable (no associated data).

\bigskip\noindent
\textbf{Conflict of interest.} The authors have no relevant financial or non-financial interests to disclose.


\begin{thebibliography}{99}
\bibitem {Abrahan}J. Abrahamson. Collision rates of small particles in a
vigorously turbulent fluid. {\it{Chem. Eng. Sci.}} 30 (1975) 1371-1379.
\bibitem{ag}
D.H. Armitage, S.J. Gardiner. {\it{Classical potential theory.}} Springer-Verlag London 2001.
\bibitem{aron}
D.G. Aronson. Bounds for the fundamental solution of a parabolic equation. {\it{Bull. Amer. Math.
Soc.}} 73 (1967), 890-896.
\bibitem {Bec2005}
J. Bec. Multifractal concentration of inertial particles in smooth random flows. {\it{J. Fluid Mech.}} 528 (2005), 255-277.
\bibitem {Bec}J. Bec, K. Gustavsson, B. Mehlig. Statistical models for the dynamics of heavy particles in turbulence. {\it{Annu. Rev. Fluid Mech.}} 56 (2024), 189-213.
%\bibitem{CF}
%M. Coghi, F. Flandoli. Propagation of chaos for interacting particles subject to environmental noise. {\it{Ann. Appl. Probab.}} 26 (2016), 1407-1442.
\bibitem{DN}
R.W.R. Darling, J.R. Norris. Differential equation approximations for
Markov chains. {\it{Probab. Surv.}}  5 (2008) 37-79.
\bibitem {Dou}Z. Dou, A. D. Bragg, A.L. Hammond, Z. Liang, L.R. Collins, H. Meng. Effects of Reynolds number and Stokes number on particle-pair relative velocity in isotropic turbulence: a systematic experimental study. {\it{J. Fluid Mech.}} 839 (2018), 271-292.
\bibitem {Falkovich}G. Falkovich, A. Pumir, Sling effect in collisions of water droplets in turbulent clouds. {\it{J. Atmos. Sci.}} 64 (2007), 4497-4505.
%\bibitem{FlaHuangPapini}
%F. Flandoli, R. Huang, A. Papini. Interacting particles with coalescence, in preparation for the Proceedings Winter School in Marseille-Luminy, January 8-12, 2024.
\bibitem{fh}
F. Flandoli, R. Huang. The KPP equation as a scaling limit of locally interacting Brownian particles. {\em{J. Differ. Equations}} 303 (2021), 608-644.
\bibitem{fh-alea}
F. Flandoli, R. Huang. Coagulation dynamics under environmental noise:
Scaling limit to SPDE. {\it{ALEA, Lat. Am. J. Probab. Math. Stat.}} 19 (2022), 1241-1292.
\bibitem{FlaLuo} 
F. Flandoli, D. Luo. Mean field limit of point vortices with
environmental noises to deterministic 2D Navier-Stokes equations. {\it{Commun. Math. Stat.}} (2024), \url{https://doi.org/10.1007/s40304-024-00406-5}
\bibitem{FouHauMis} 
N. Fournier, M. Hauray, S. Mischler. Propagation of chaos for
the 2D viscous vortex model. {\it{J. Eur. Math. Soc.}} 16 (2014), 1423-1466.
\bibitem{hr}
A. Hammond, F. Rezakhanlou. The kinetic limit of a system of coagulating
Brownian particles. {\em{Arch. Rational Mech. Anal.}} 185 (2007), 1-67.
\bibitem{h-ejp}
R. Huang. Scaling limit for a second-order particle system with
local annihilation. {\it{Electron. J. Probab.}} 28 (2023), article no. 81, 1-31.
\bibitem{ilyin}
A.M. Ilyin, A.S. Kalashnikov, O.A. Oleynik. Second-order linear equations of parabolic
type. {\it{J. Math. Sci. (N.Y.)}}, 108 (2002), 435–542.
\bibitem{KL}
C. Kipnis, C. Landim. {\it{Scaling Limits of Interacting Particle Systems}}, Grundlehren der Mathematischen Wis-senschaften (Fundamental Principles of Mathematical Sciences) 320, Springer-Verlag, Berlin, 1999.
%\bibitem{Kunita}
%H. Kunita. {\it{Stochastic Flows and Stochastic Differential Equations}}. Cambridge Studies in Advanced Mathematics 24. Cambridge Univ. Press, Cambridge, 1990.
\bibitem{krylov}
N.V. Krylov. {\it{Lectures on elliptic and parabolic equations in Hölder spaces}}, Graduate Studies in Mathematics 12, American Mathematical Society, 1996.
\bibitem{lsw}
W. Littman, G. Stampacchia, H.F. Weinberger. Regular points for elliptic equations with discontinuous coefficients. {\it{Ann. Sc. Norm. Super. Pisa Cl. Sci. (3)}} 17 (1963), 43-77.
\bibitem{LXZ}
D. Luo, B. Xie, G. Zhao. Quantitative estimates for SPDEs on the full space with
transport noise and $L^p$-initial data. {\it arXiv:2410.21855}, \url{https://arxiv.org/abs/2410.21855v1}
\bibitem {Mehlig}B. Mehlig, V. Uski, M. Wilkinson. Colliding particles in highly turbulent flows. {\it{Phys. Fluids}} 19 (2007) 098107, 1-3.
\bibitem {Papini0}A. Papini. {\it{Turbulence enhancement of coagulating processes}}. Ph.D. thesis, Scuola Normale Superiore, 2023.
\bibitem {Papini}A. Papini, F. Flandoli, R. Huang. Turbulence enhancement of
coagulation: The role of eddy diffusion in velocity, {\it{Physica D: Nonlin.
Phenom.}} 448 (2023), 133726, 1-16.
\bibitem {Pumir}A. Pumir, M. Wilkinson. Collisional aggregation due to turbulence. {\it{Annu. Rev. Condens. Matter Phys.}} 7 (2016), 141-170.
\bibitem{reza} F. Rezakhanlou.  Boltzmann-Grad limits for stochastic hard sphere models. {\it{Comm. Math. Phys.}} 248 (2004), 553-637.
\bibitem{st} P.G. Saffman, J.S. Turner. On the collision of drops in turbulent clouds. {\it J. Fluid Mech.} 1 (1956), 16-30.
\bibitem {Wilkin}M. Wilkinson, B. Mehlig, V. Bezuglyy. Caustic activation of rain showers. {\it{Phys. Rev. Lett.}} 97 (2006), 048501, 1-4.
\end{thebibliography}
\end{document}